\documentclass[10pt,dvipsnames]{amsart}
\usepackage[utf8]{inputenc}

\usepackage{amssymb,amsthm,amsmath,amstext,amscd}
\usepackage{eucal}
\usepackage{enumitem}
\usepackage{epsfig}
\usepackage{mathrsfs}
\usepackage{bm}
\usepackage[utf8]{inputenc}
\usepackage{comment}
\usepackage[all]{xy}
\usepackage[left=2.7cm,top=2.7cm,right=2.7cm,bottom=2.7cm]{geometry}
\usepackage{dynkin-diagrams}

\theoremstyle{plain}
\newtheorem{theorem}{Theorem}[section]
\newtheorem{conjecture}[theorem]{Conjecture}
\newtheorem*{conjecture*}{Conjecture}
\newtheorem{prop}[theorem]{Proposition}
\newtheorem{lemma}[theorem]{Lemma}
\newtheorem{defi}[theorem]{Definition}
\newtheorem{coro}[theorem]{Corollary}

\theoremstyle{definition}
\newtheorem{remark}[theorem]{Remark}
\newtheorem{definition}[theorem]{Definition}
\newtheorem*{ack}{Acknowledgements}

\def\CC{{\mathbb{C}}}

\def\FF{{\mathbb{F}}}

\def\OO{{\mathbb{O}}}

\def\AA{{\mathbb{A}}}
\def\PP{{\mathbb{P}}}
\def\FF{{\mathbb{F}}}
\def\QQ{{\mathbb{Q}}}\def\ZZ{{\mathbb{Z}}}

\def\cO{{\mathcal{O}}}
\def\cE{{\mathcal{E}}}
\def\cA{{\mathcal{A}}}
\def\cB{{\mathcal{B}}}
\def\cF{{\mathcal{F}}}
\def\cG{{\mathcal{G}}}
\def\cI{{\mathcal{I}}}

\def\cN{{\mathcal{N}}}

\def\cS{{\mathcal{S}}}\def\cU{{\mathcal{U}}}
\def\cC{{\mathcal{C}}}\def\cQ{{\mathcal{Q}}}
\def\cF{{\mathcal{F}}}

\def\ra{{\rightarrow}}
\def\lra{{\longrightarrow}}
\def\ft{{\mathfrak t}}\def\fs{{\mathfrak s}}
\def\fc{{\mathfrak c}}\def\fe{{\mathfrak e}}
\def\fg{{\mathfrak g}}\def\fh{{\mathfrak h}}
\def\fso{\mathfrak{so}}

\def\fsl{\mathfrak{sl}}

\DeclareMathOperator{\diag}{diag}

\DeclareMathOperator{\Hom}{Hom}
\DeclareMathOperator{\rank}{rank}

\DeclareMathOperator{\id}{id}

\DeclareMathOperator{\Jac}{Jac}

\DeclareMathOperator{\GL}{GL}
\DeclareMathOperator{\SL}{SL}
\DeclareMathOperator{\SU}{SU}
\DeclareMathOperator{\SO}{SO}
\DeclareMathOperator{\Sp}{Sp}

\DeclareMathOperator{\PGL}{PGL}
\DeclareMathOperator{\Spin}{Spin}
\DeclareMathOperator{\corank}{corank}

\usepackage[]{xcolor}



\title{G\"opel varieties}

\author{V. Benedetti, M. Bolognesi, D. Faenzi, L. Manivel}

\begin{document}
\sloppy

\begin{abstract}
We show that the Coble hypersurfaces, uniquely characterized by the remarkable property that their singular loci are an abelian surface and a Kummer threefold, respectively, belong to a family of hypersurfaces exhibiting similar behavior, but 
defined in various types of homogeneous spaces. With the help of Jordan-Vinberg  theory, we show how these hypersurfaces can be parametrized by G\"opel type varieties inside projectivized representations of complex reflection groups.
\end{abstract}

\maketitle
\setcounter{tocdepth}{1}
\tableofcontents

\smallskip
\section{Introduction}
 
 The Coble hypersurfaces are characterized by the very unusual properties that their singular loci are an abelian surface for the Coble cubic in $\PP^8$, and a Kummer threefold for the Coble quartic in $\PP^7$ \cite{beauville-coble}. Being hypersurfaces, they are easy to handle. The classical Göpel variety parametrizes the Coble quartics one obtains 
when the Kummer threefold varies, hence the abelian threefold from which it is constructed, hence the genus three curve whose Jacobian is this threefold; it can thus be considered as a moduli space for this kind of objects, endowed with some sort of additional structure \cite{do, rsss}. For Coble cubics, the analoguous object is the classical Burkhardt quartic.

\smallskip
In \cite{coblequadric, BBFM2} we observed that one can construct other interesting objects, such as moduli spaces of vector bundles on curves, as singular loci of hypersurfaces, at the price of extending the class of ambient spaces we are ready 
to accept, so as to include Grassmannians, quadrics and other homogeneous spaces. 
Once we are ready for that, we observe that Coble type hypersurfaces do exist in many interesting situations; this is the consequence of natural constructions 
in terms of \emph{orbital degeneracy loci}, introduced in \cite{bfmt, bfmt2}. These constructions are based on Lie theory and require an action of an algebraic 
group, typically an irreducible representation. It appears that there exist 
a series of special representations having rich connections with such classical objects as curves or low genus or low-dimensional abelian varieties, endowed 
with some additional data. These representations were identified in \cite{GSW}, and simultaneously in a series of works by Bhargava and his school (see e.g. \cite{bhargava-ho}). Intriguingly, most of them are derived from the exceptional Lie algebras, and one could argue that substantial pieces of classical algebraic geometry are,  in a way, shadows of the exceptional Lie groups, observed from the 
depth of Plato's cave. These representations do not appear in the works of the classics, but they shed a new light on them and open new perspectives.

\smallskip
The aim of this paper is to discuss, for each of these special representations, which Coble type hypersurfaces can be constructed, how they can be parametrized and what kind of moduli space are obtained from these constructions. The main property shared by these special representations, say $V$ acted on by $G_0$,  is that they are pieces of some cyclic grading of a simple Lie algebra $\fg$.  This allows to put into action the very powerful version of Jordan theory developed in this setting by Kac and Vinberg. The main actor %character 
in this theory is the Cartan subspace $\fc\subset V$, which plays the role of the space of diagonal matrices in usual Jordan theory. It is acted on by a complex reflection group $W_\fc$, and the resulting hyperplane arrangement contains a wealth of information; by Vinberg's version of Chevalley's restriction  theorem, the quotient of $\fc$ by  $W_\fc$ is 
isomorphic to the GIT quotient $V/\hspace{-1mm}/G_0$. As a consequence, any $G_0$-equivariant map from $V$ to another representation $U$ restricts 
to a 
$W_\fc$-equivariant map from $\fc$ to some $W_\fc$-module $\Theta\subset U$.

\smallskip
The connection with geometry begins when one observes that the centralizer of $\fc$ in $G_0$ is a finite group of Heisenberg type, similar to the theta-groups associated to polarized abelian varieties. This is of course a strong hint that such an abelian variety could be constructed geometrically, and this is exactly what 
the theory of orbital degeneracy loci is good for: starting from a point in $\fc$, we can define sections of several equivariant vector bundles on auxiliary homogeneous spaces $Z$, bundles whose spaces of global sections are  isomorphic to   $V$. Orbital degeneracy loci are then  loci where such a section degenerates in some natural way. It was observed in \cite{bhargava-ho, GSW} that various types of abelian varieties can be obtained that way, as well as the Coble hypersurfaces.

\smallskip
In this paper we take one step further by showing that very often, there is a codimension one orbital degeneracy locus of special interest, stratified by successive singular loci that includes the relevant abelian variety; we call such a locus a 
\emph{Coble type hypersurface}. 
Sometimes we can even prove that this special singular locus uniquely determines the hypersurface, as in the  classical case of the Coble cubics and quartics. According to Jordan-Vinberg theory, we can parametrize these hypersurfaces by restricting to the Cartan subspace; they form a linear system isomorphic to a  $m$-dimensional $W_\fc$-submodule $\Theta_m$ of some $S^d\fc^\vee$, parametrizing Heisenberg-invariant polynomials. We can then define the \emph{G\"opel variety} as the image of the rational map  $\gamma: \PP(\fc)\dashrightarrow \PP(\Theta_m)$.  
In several cases this variety has a nice moduli interpretation. It may even happen that we get several such varieties, when several hypersurfaces of interest can be constructed in different auxiliary spaces. 

\smallskip
To summarize, the successive steps we take are as follows.
\begin{itemize}
    \item Start from a cyclic grading of a simple Lie algebra $\fg$; this yields a so called $\theta$-representation $V$ of a group $G_0$. 
    \item Identify a Cartan subspace $\fc$, together with its complex reflection group $W_\fc$; its centralizer turns out to be a Heisenberg type subgroup of $G_0$. 
    \item Choose an auxiliary $G_0$-homogeneous space $Z$ inside which a nonzero $v\in V$ defines a special hypersurface of degree $\delta$; study the geometry of this hypersurface and its singularity stratification. 
    \item Interpret the equation of this hypersurface, if it is given by a degree $d$ polynomial function on $V$, as the image of a $G_0$-equivariant morphism  $$\Gamma : S^dV\lra H^0(Z,\cO_Z(\delta))$$
    \item Restrict this morphism to $\fc$ and denote by $\Theta_m$ the $m$-dimensional span of the image, which is naturally endowed with a $W_\fc$-representation structure, to obtain a $W_\fc$-equivariant rational map $$\gamma : \PP(\fc)\dashrightarrow \PP(\Theta_m).$$
    \item Whenever possible, describe the image $\cG$ of $\gamma$: this is our G\"opel type variety;  finally, provide a modular interpretation. 
\end{itemize}

\medskip
We apply this program to the series of gradings and representations from the table 
below. The first column presents an affine Dynkin diagram $D$ 
of $\fg$, where the black vertex corresponds to the simple root $\alpha_0$ 
that defines the grading, as in \cite[section 9]{vinberg}. 
The second column gives the subalgebra $\fg_0=Lie(G_0)$ and the $\fg_0$-module $V=\fg_1$; the semisimple part of $\fg_0$ has Dynkin diagram $D_0=D\backslash \{\alpha_0\}$; the highest weight of $V$ is the sum of the fundamental weights corresponding to the vertices of $D_0$ connected to $\alpha_0$ in $D$ (with a coefficient $2$ in the third case, because of the double edge); these vertices are located by a cross.

\medskip 
\begin{small}
 $$\begin{array}{c|c|c|c|c|c|l}
 \fg & \fg_0, V & Z & \delta  & d &  & \mbox{§}\\ 
\hline
\hline
 &&&&&& \\
\dynkin[extended, edge length=5mm] E{ooot*too} 
&  \begin{split}\fsl_5\times\fsl_5 \\ \wedge^2\CC^5 \otimes \CC^5 \end{split} &
 \PP^4\times\PP^4&(2,1)& 2& \mbox{\cite{bhargava-ho}} & \mbox{§ } \ref{sec_toy_5} \\
  &&&&& \\
 \hline
 \dynkin[extended, edge length=5mm] E{ott*to} 
 & \begin{split}
     \fsl_3\times \fsl_3\times \fsl_3 \\
     \CC^3\otimes \CC^3\otimes \CC^3
 \end{split} & \PP^2 \times \PP^2 \times \PP^2 & 2 &4& \mbox{Rubik's cubes \cite{bhargava-ho}} & \mbox{§ } \ref{sec_toy_rubik}  \\ 
  &&&&& \\
  \hline
  &&&&& \\
 \dynkin[edge length=5mm, odd]A[2]{ooo.oot*} & \begin{split}
     \fso_{2m} \\
    S^{\langle 2 \rangle}\CC^{2m} \end{split} & OG(k,2m) & 2 & k& \begin{split}
        \mbox{Spin bundles on hyperelliptic} \\
        \mbox{curves \cite{DR77, Ram1981, Bhosle1984}}
    \end{split} & \mbox{§ } \ref{sec_hyperell} \\
   &&&&& \\
  \hline
  &&&&& \\
\dynkin[extended, edge length=5mm] E{o*otooo} & \begin{split}
    \fsl_8 \\ \wedge^4 \CC^8
\end{split} &\PP^7,\check{\PP}^7& 4 &7&\mbox{Coble quartic \cite{rsss}} & \mbox{§ } \ref{sec_coble_quartic_rsss} \\
 &&G(2,8)&2&3& \mbox{Coble quadric \cite{coblequadric}} & \mbox{§ } \ref{sec_coble_quadric_genus3} \\
&&Fl_8 &2&4&\\
  &&&&& \\
  \hline
  &&&&& \\
\dynkin[extended, edge length=5mm] E{o*otoooo} &  \begin{split}
    \fsl_9 \\
    \wedge^3 \CC^9
\end{split} & \PP^8&3&4& \mathrm {Coble \,cubic} & \mbox{§ } \ref{sec_genustwo_GS}\\ && \check{\PP}^8& 6 &4&  \mbox{dual sextic} \\
&& G(6,9) & 2 &4& \mbox{Coble quadric \cite{BBFM2}} \\
  &&&&& \\
  \hline
  &&&&& \\
 \dynkin[extended, edge length=5mm] E{oooot*to} & \begin{split}
     \fsl_4\times \fso_{10} \\ \CC^4 \otimes \Delta_{+}
 \end{split} & 
 G(2,4) &2&4& \mbox{quadratic complex \cite{liu-manivel}} \\
 & & Fl_4 & 2 & 8& \mbox{Coble quadric} & \mbox{§ } \ref{sec_fortuples_spinors}\\
 & & \PP^3 & 4 & 12& \mbox{Kummer surface} \\
 & & \QQ^8 & 4 & 8 & \\
 & & OG(2,10) & 2 &4&\\
  &&&&& \\
  \hline
  &&&&& \\
 \dynkin[extended, edge length=5mm] E{*otooooo} &  \begin{split}
     \fso_{16} \\ \Delta_+
 \end{split} &\QQ^{14}&4&8& \mbox{Coble quartic} & \mbox{§ } \ref{sec_quartic_spin}\\
 && OG(2,16)& 2 &4&\mbox{Coble quadric} & \mbox{§ } \ref{sec_quadric_spin}\\
   
 \end{array}$$
 
\end{small}

 \bigskip
 In the second table below we describe the corresponding little Weyl groups $W_\fc$; the notation $G_N$ refers to the Shephard-Todd classification \cite{st}. Their representations $\Theta_m$ are always irreducible (of dimension $m$). They are always Macdonald representations, in the sense that they are generated by certain products of equations of reflection hyperplanes. This is typically the case of the classical Specht modules, which are the irreducible representations of symmetric groups; for instance we get the Specht module $[k,k]$, a representation of the symmetric group $S_k$ of dimension $m=\frac{(2k)!}{k!(k+1)!}$, the $k$-th Catalan number.
 
 $$\begin{array}{c|c|c|c}
  W_{\fc} & \Theta_m & \mathrm{G\ddot{o}pel} & \mathrm{Modular\; interpretation}  \\
 \hline
 \hline
    G_{16} & \Theta_3 & v_2(\PP^1) & \mathrm{genus\,one\,curves+degree\,5\,line\,bundle} \\
 \hline
   G_{25} & \Theta_{12} & \mathrm{projection\;of\;} v_4(\PP^2) &  \mathrm{genus\,one\,curves+two\,degree\,3\,line\,bundles}
 \\
\hline
S_{2k} & [k,k]  &  (\PP^1)^{2k}/\hspace{-1mm}/PGL_2 & \mathrm{hyperelliptic\,curves\,/\,} 2k\,\mathrm{points\;in}\,\PP^1 \\ 
\hline
 W(E_7) & \Theta_{15}, \Theta_{21} &  (\PP^2)^7/\hspace{-1mm}/PGL_3 & \mathrm{plane\;quartics}  \\
\hline 
G_{32} & \Theta_5 &  \mathrm{Burkhardt\;quartic} & 
\mathrm{genus\,two\,curves+level\,3\,structure} \\
\hline
 G_{31} & \Theta_5, \Theta_9 & \mathrm{Igusa\,quartic} & \mathrm{genus\,two\,curves+level\,2\,structure}  \\
&  \Theta'_5 & \mathrm{Segre\,cubic} & \\
\hline W(E_8) & \Theta_{50},  \Theta_{84} & & 
\mathrm{special\,genus\,four\,curves} 
 \end{array}$$
 
 \medskip

 \bigskip\noindent {\it Relation to previous work}. The G\"opel variety was defined 
 in \cite{rsss} as the moduli space of Coble quartics in $\PP^7$, parametrized by the so-called G\"opel functions \cite{do} as a subvariety of $\PP^{14}$; it is cut out by $35$ quadrics and $35$ cubics.  
 The moduli 
 space of Coble cubics in $\PP^8$ was computed in \cite{gs} and shown to coincide with the classical Burkhardt quartic in $\PP^4$.  The moduli space of dual Coble sextics has not been considered and seems computationally challenging. More accessible are the Coble quadrics that we defined in \cite{coblequadric} as quadric sections of the Grassmannian $G(6,9)$; we show the corresponding G\"opel type variety is a projection of the fourth-Veronese of $\PP^3$ from the span of the  Burkhardt quartic, and that it is cut out by $300$ quadrics (Theorem \ref{wedge3c9}). We also recover the classical description of the moduli space of $2k$-points in $\PP^1$ by considering such a $2k$-tuple as the base locus of  a pencil of quadrics in $\PP^{2k-1}$; fixing a smooth quadric in the pencil, we define a quadric section of the associated orthogonal Grassmannian $OG(k,2k)$, whose equation  recovers the classical cross-ratio functions \cite{do}. 
 
 The case of $6$ points in $\PP^1$ is also related 
 to $\SL_4\times \Spin_{10}$; on the side of $\SL_4$, we recover the classical Segre cubic primal and Castelnuovo-Richmond (or Igusa) quartic; but on the side of $\Spin_{10}$ we show that one can also define Coble type hypersurfaces, namely quartic sections of the quadric $\QQ^{8}$ and quadric sections of the orthogonal Grassmannian $OG(2,10)$. 
 
 \medskip
 But the most difficult case we deal with is the half-spin representation of $\Spin_{16}$, which is known to be related to a family of special genus four curves. 
 Using spinor bundles, we show how to construct Coble type hypersurfaces as  
 quartic sections of the quadric $\QQ^{14}$, and quadric sections of the orthogonal Grassmannian $OG(2,16)$. We show that the two associated G\"opel type varieties, both seven dimensional,   span (the projectivization of) two Macdonald representations of the Weyl group $W(E_8)$, of dimension $84$ (Proposition \ref{84}) and $50$ (Theorem \ref{50}), respectively. Conjecturally, these two Göpel type hypersurfaces should contain, inside their singular locus, the two moduli spaces of rank two vector bundles on the genus four curve, with fixed determinant of even or odd degree. For the even moduli space of bundles on a genus four curve with no vanishing theta-null, 
 the existence of such a hypersurface was proved in \cite{OxburyPauly}; the case we address here is a disjoint, degenerate one. Also we expect the (smooth) odd moduli space to coincide with the singular locus of our Coble type hypersurface in $OG(2,16)$; it would be interesting to prove some statement of this type for a general curve of genus four. For curves of genus two and three such statement we established in \cite{coblequadric, bfmt2}. 
 
 \medskip
 We find it remarkable that a unified perspective can be given on all these moduli problems, through Coble type hypersurfaces that are geometrically meaningful. In quite a few cases the precise geometry of these hypersurfaces is still to be understood; their successive singular strata are often interesting varieties, but some of them have not been identified yet.  We also note that
 similar Coble type phenomena have been observed in related, but different situations: for example, the moduli space of $8$ points in $\PP^1$ embeds in $\PP^{13}$ as the singular locus of a uniquely defined cubic \cite{HMSV12}. It would be very nice to discover other relevant examples. 
 
\begin{ack}
All authors  were supported by the project FanoHK ANR-20-CE40-0023. D. F. and V. B. were partially supported by SupToPhAG/EIPHI ANR-17-EURE-0002, Région Bourgogne-Franche-Comté, Bridges ANR-21-CE40-0017.
\end{ack}

 \section{Preliminaries}
 
 A grading of the Lie algebra $\fg$ of a semisimple simply connected group $G$ is a decomposition $\fg=\oplus_i \fg_i$ such that $[\fg_i,\fg_j]\subset \fg_{i+j}$. In this paper the index $i$ runs over $\ZZ$ or $\ZZ_m=\ZZ/m\ZZ$. Vinberg developed in \cite{vinberg} the structure theory of $\ZZ/m\ZZ$-graded Lie algebras. Such a grading is always given by an automorphism $\theta$ of $\fg$, such that $\theta^m=\id$. Vinberg studied the action of $G^\theta$ (the subgroup of  $\theta$-invariant elements in $G$), whose Lie algebra is $\fg_0$, on a graded piece, say $\fg_1$; this is often called a $\theta$-representation of a $\theta$-group, the latter being either $G^\theta$ or its universal cover $G_0$. He showed that there exist maximal abelian subspaces $\fc\subset \fg_1$ consisting of semisimple elements; they are called Cartan subspaces and are all conjugate under $G^\theta$; moreover they parametrize closed $G^\theta$-orbits in $\fg_1$. The rank of a graded Lie algebra is $\dim(\fc)$. A Cartan subspace $\fc$ is acted on by the \emph{little Weyl group} $W_\fc:=N_{G^\theta}(\fc)/Z_{G^{\theta}}(\fc)$, which is a finite reflection group (here $N$ stands for normalizer, and $Z$ for centralizer). This extends the definition of the usual Weyl group of a Cartan subalgebra in a semisimple Lie algebra. He then proved a Chevalley type 
 restriction theorem relating $G^\theta$-invariants in $\fg_1$ and $W_\fc$-invariants in $\fc$  via restriction to the Cartan subspace:
 $$ \CC[\fg_1]^{G^\theta} \simeq \CC[\fc]^{W_\fc} .$$
We will see many examples of graded Lie algebras throughout the paper. 

\smallskip
Complex reflection groups $W_\fc$ acting on a vector space $\fc$ were completely classified in the famous work of Shephard and Todd \cite{st}. Among their representations, one can distinguish the so-called Macdonald representations, defined as follows. Consider a system $S:=\{c_1,\cdots ,c_k\}\subset \fc^\vee$ of non-trivial linear forms, defining reflection hyperplanes of reflections $s_1,\cdots s_k\in W_\fc$. We can construct a  $W_\fc$-representation inside $\CC[\fc^\vee]_k$ as the linear span $$\langle w\cdot (c_1 c_2 \cdots c_k)\rangle_{w\in W_\fc}\subset \CC[\fc]_k.$$ 
When it is irreducible, we call it a \emph{Macdonald representation}. Macdonald proved that irreducibility always holds when $W_\fc$ is a genuine Weyl group
and $S$ is any subsystem of the root system \cite{macdonald}. This generalizes the construction of irreducible modules of symmetric groups as Specht modules.  

\smallskip
In all the examples we will meet in the sequel, the Cartan subspace will be the 
space of invariants in $\fg_1$ of a finite Heisenberg subgroup $H$ of $G_0$, the universal cover of $G^\theta$. Conversely, $H$ can be obtained as the subgroup of $G_0$ acting trivially on $\fc$. Finite Heisenberg groups are defined as central extensions
$$ 1\to \ZZ_k \to H \to K \to 1,$$
where $K$ is a finite abelian group \cite{ThetaIII}. Moreover, it is required that the commutator map $$ K\times K \to \ZZ_k, \qquad (x,y)\mapsto \widetilde{x}\widetilde{y}\widetilde{x}^{-1}\widetilde{y}^{-1} ,$$
where $\widetilde{x},\widetilde{y}$ are lifts of $x,y$ to $H$, is a  non-degenerate  skew-symmetric pairing.

 \section{A toy case}
 \label{sec_toy_5}

We start with the  case of  $\SL_5\times \wedge^2\SL_5$, or quintuples of skew-symmetric matrices of order five. This is classically related to quintic elliptic curves \cite[section 4.4]{bhargava-ho}. From our point of view, the corresponding Göpel variety will just be a conic; but treating this case in some detail will allow to introduce in a simple setting the approach we will apply  to more interesting, and also more involved cases.  

From the Jordan-Vinberg perspective, 
this example comes from a $\ZZ_5$-grading of $\fe_8$, namely 
$$\fe_8=\fsl(A_5)\times \fsl(B_5)\oplus (A_5\otimes\wedge^2B_5)\oplus (\wedge^2A_5\otimes B_5^\vee)
\oplus (\wedge^2A_5^\vee\otimes B_5)\oplus (A_5^\vee\otimes\wedge^2B_5^\vee).$$
So $G_0=\SL(A_5)\times \SL(B_5)$ acts on $\fg_1=\wedge^2A_5\otimes B_5$. The rank, i.e. the dimension of a Cartan subspace, is only two in this case, and the 
corresponding little Weyl group  is the complex reflection group $G_{16}$ \cite{vinberg}, of order $600$, with invariants of 
degrees $20$ and $30$; this group is a product of $\ZZ_5$ by the binary isocahedral group \cite{st}. 

\subsection{A Cartan subspace}
Let $a_1, \ldots , a_5$ be a basis of $A_5$, and $b_1, \ldots, b_5$ be a basis of $B_5$. We will use the notation $b_{ij}=b_i\wedge b_j\in\wedge^2B_5$.

\begin{lemma}\label{basiscartan}
The subspace $\fc\subset A_5\otimes \wedge^2B_5$  generated by 
$$\begin{array}{rcl}
h_1 & = & a_1\otimes b_{12}+a_2\otimes b_{23}+a_3\otimes b_{34}+a_4\otimes b_{45}+a_5\otimes b_{51}, \\ 
h_2 & = & a_1\otimes b_{35}+a_2\otimes b_{41}+a_3\otimes b_{52}+a_4\otimes b_{13}+a_5\otimes b_{24},
\end{array}$$
is a Cartan subspace of $\fg_1$. 
\end{lemma} 

\proof We need to prove that $h_1, h_2$ are semisimple elements of $e_8$, and commute. For the latter claim, observe that the Lie bracket $\wedge^2\fg_1\lra \fg_2$ must be proportional, by $G_0$-equivariance, to the morphism 
$$\wedge^2(A_5\otimes \wedge^2B_5)\lra \wedge^2A_5\otimes B_5^\vee$$
deduced from the map $S^2(\wedge^2B_5)\lra \wedge^4B_5\simeq B_5^\vee$, where the latter identification is defined by the choice of a generator of $\wedge^5B_5$, say $b_1\wedge\cdots\wedge b_5$. Explicitly, the Lie bracket of $a_p\otimes b_{ij}$ with $a_q\otimes b_{kl}$ is $a_{pq}\otimes b_{ijkl}$, where $b_{ijkl}$ identifies with $\epsilon b_m^\vee$ if $(ijklm)$ is a permutation of $(12345)$ of signature $\epsilon$, and zero otherwise. Then the claim that $[h_1,h_2]=0$ follows from the observation that if  $a_p\otimes b_{p,p+1}$ is one of the terms of $h_1$, and  $a_q\otimes b_{q+2,q-1}$ one of the terms of $h_2$, then either $p=q$ of the pairs $(p,p+1)$ and $(q+2,q-1)$ have non-empty intersection. 

In order to prove that $h_1$ is semisimple, we let $x_p^1=a_p\otimes b_{p,p+1}$, and we observe that the subalgebra $\fs$ of $\fe_8$ generated by these five elements of $\fg_1$ is isomorphic to $\fsl_5$ with its standard $\ZZ_5$-grading: in the canonical basis $e_1,\ldots e_5$ of $\CC^5$, $(\fsl_5)_k$ is generated by 
matrices of the form $e_p^\vee\otimes e_{p+k}$ (with the additional condition that the trace vanishes when $k=0$). Explicitly, the isomorphism $\theta: \fsl_5\simeq \fs$ is obtained by sending $e_p^\vee\otimes e_{p+1}$ to $a_{2p-1}\otimes b_{2p-1,2p}$. In particular,
$$\theta^{-1}(h_1)=\begin{pmatrix} 0&0&0&0&1 \\1&0&0&0&0 \\0&1&0&0&0 \\0&0&1&0&0 \\0&0&0&1&0 \end{pmatrix}  $$
is semisimple; so $h_1$ is semisimple as well. The case of $h_2$ is completely similar, and the claim follows. 
\qedsymbol

\subsection{The Heisenberg group}
This is the  stabilizer $H$ of $\fc$ inside $G_0=\SL(A_5)\times \SL(B_5)$. 
We can describe it explicitly. 

\begin{lemma} 
The Heisenberg group is a non-trivial central extension 
$$1\lra \ZZ_5\lra H \lra (\ZZ_5)^2\to 1.$$
\end{lemma}

\proof We can identify $H$ with its image in $\SL(B_5)$, 
%which is the subgroup $\gamma$ \dan{c'est $\Gamma$ ? Est-ce que cette préservation de plans suffit pour déterminer $\Gamma$ ? On a l'action en coordonnées explicite en la base en question ?}
whose induced action on $\wedge^2B_5$ has to preserve the two five-planes $I_1=\langle b_{12}, b_{23}, b_{34}, b_{45}, b_{51}\rangle$
and $I_2=\langle b_{35}, b_{41}, b_{52}, b_{13}, b_{24}\rangle$, and act on these two spaces in the same way
(in their two given basis). Note that if $g\in \SL(B_5)$ preserves $I_1$ and $I_2$, it also has to preserve the 
varieties $D_1$ and $D_2$ of decomposable tensors in $\PP(I_1)$ and $\PP(I_2)$. 

A straightforward computation shows that $D_1$ and $D_2$ are two pentagons, $D_1$ being the union of the five lines joining $b_{i-1,i}$ to $b_{i,i+1}$, while 
$D_2$ is the union of the five lines joining $b_{i-2,i}$ to $b_{i,i+2}$.
As a consequence, $g\in \SL(B_5)$ preserving  $D_1$ and $D_2$ must also preserve their singular loci, which are the 
$\langle b_{i,i+1}\rangle $'s and $\langle b_{i,i+2}\rangle $'s, respectively. Then it also has to preserve the intersections of the corresponding lines in
$B_5$, which are the $\langle b_i\rangle $'s. So we must have $g(b_i)=\beta_ib_{\sigma(i)}$ for some permutation $\sigma$ and some scalars $\beta_i$. Since $\sigma(i)$ and $\sigma(i)+1$ must differ by one, the only possibilities are that 
$\sigma$ is a cycle ($\sigma(i)=i+s$ mod $5$), or a reflection ($\sigma(i)=s-i$ mod $5$); but the latter case is excluded by the fact that $g$ should satisfy $\det(g)=1$. So we may suppose that $g(a_i)=\alpha_i a_{i+s}$,  $g(b_i)=\beta_i b_{i+s}$, with the additional conditions that for any $i$, $\alpha_i\beta_i\beta_{i+1}=\alpha_i\beta_{i+2}\beta_{i+4}=1$. So the $\alpha_i$'s are determined by the $\beta_i$'s, which are subject to the condition that $\beta_i\beta_{i+1}=\beta_{i+2}\beta_{i+4}$ for all $i$. One checks that this is equivalent to 
$$\beta_3=\beta_1^{-1}\beta_2^2, \qquad \beta_4=\beta_1^{-2}\beta_2^3, \qquad \beta_5=\beta_1^{2}\beta_2^{-1},$$
for $\beta_1, \beta_2$ any fifth roots of unity (hence $\beta_i^{-1}\beta_{i+1}$ is independent of $i$). Notice that when $\beta_1=\beta_2$ and $\sigma=\id$, we get $g=\beta_1 \id$. We conclude that
the morphism sending $g\mapsto (\sigma,\beta_1^{-1}\beta_2)$ yields the 
extension
$$1\lra \ZZ_5\lra H \lra (\ZZ_5)^2\to 1,$$
in agreement with Appendix \ref{sec_heis_group}. \qed

\begin{coro} The space of Heisenberg invariants in $\fg_1$ is exactly $\fc$. 
\end{coro} 

\subsection{A Coble type hypersurface}
A generic tensor $v$ in $A_5^\vee\otimes\wedge^2B_5^\vee$ defines several interesting geometric objects as follows. If $U_m$ is an $m$-dimesional vector space, we denote by $\cU_k$ the tautological rank $k$ vector bundle on a Grassmannian $G(k,U_m)$.
\begin{enumerate}
    \item A genus one curve $E$, as the codimension five section of $G(2,B_5)$ cut out by the image of 
    $A_5$ in $\wedge^2B_5^\vee$ (see \cite[Theorem 4.14]{bhargava-ho} for a modular interpretation).
    \item A surface $S$ inside $\PP(A_5)\times\PP(B_5^\vee)$, cut out by a section of the rank six bundle 
    $\cA_1^\vee\boxtimes \wedge^2\cB_4^\vee$.
    \item A  hypersurface $Q$ defined as the vanishing locus of a global section of $\cO(1,2)$ on $\PP(A_5)\times\PP(B_5^\vee)$.
    The singular locus of the hypersurface $Q$ is the surface $S$.
    \item A Fano sixfold $X$ inside $G(2,A_5)\times G(3,B_5)$, obtained as the zero locus of a section of 
      $\cA_2^\vee\boxtimes \wedge^2\cB_3^\vee$. The two projections $p_1$ and $p_2$ are birational maps. The       exceptional locus of $p_1$ is the singular fourfold $Y$ of lines incident to $E$, whose normalisation 
      is a $\PP^3$-bundle over $E$. On the other hand, $p_2$ is the blow-up of the degeneracy locus of the morphism defined as the composition $\wedge^2\cB_3\hookrightarrow \wedge^2B_5\lra A_5^\vee$.
\end{enumerate}

\begin{prop}
The Coble-type hypersurface $Q$ is the unique hypersurface of bidegree $(2,1)$ which is singular along the ruled surface $S$.
\end{prop}

\proof A point $(\ell,h)$ of $\PP(A_5)\times\PP(B_5^\vee)$ is in $S$ if and only if 
the two-form $v(\ell)$ vanishes on the hyperplane $P_h\subset B_5$ defined by $h$. This implies that 
$v(\ell)$ has rank  two, hence defines a point of $E$, and then the hyperplane defined by $h$  has to contain its kernel. So $S$ is a ruled surface over $E$. 

Since $S$ is the zero locus of the bundle $\cE=\cA_1^\vee\boxtimes \wedge^2\cB_4^\vee$
on  $\PP(A_5)\times\PP(B_5^\vee)$, its ideal sheaf is resolved by the usual Koszul complex. The following generalization 
is probably well-known.

\begin{lemma}\label{resolution}
If $Z$ is the zero locus of a section of a vector bundle $\cE$ or rank $e$ 
over some smooth 
variety $X$, and has the expected codimension, then the $k$-th power of its ideal sheaf admits 
a resolution 
$$0\lra S_{k1^{e-1}}\cE^\vee \lra \cdots \lra  S_{k1}\cE^\vee\lra S_{k}\cE^\vee\lra \mathcal I_Z^k\lra 0.$$
\end{lemma}

\noindent {\it Proof of the Lemma.} 
Suppose that $Z=Z(s)$ for some section $s\in H^0(X,\cE)$. Consider inside the 
projective bundle $\PP_X(\cE)$ the locus $\tilde{X}$ of pairs $(x,[e])$ where $s(x)\in [e]$.
On the one hand, the projection to $X$ is an isomorphism outside $Z$, whose preimage is $F=\PP_Z(\cE)$.  By the universal property of blowups, 
we conclude that  $\tilde{X}\simeq 
Bl_Z(X)$, the blowup of $X$ along $Z$, and  that $F$ is the exceptional divisor. 
On the other hand, there is on $\PP_X(\cE)$ a tautological
sequence $0\lra \cO_\cE(-1)\lra \pi^*\cE\lra \cQ\lra 0$, where $\pi$ denotes the projection. 
In particular $s$ defines a section $\tilde{s}$ of $\cQ$ that vanishes exactly along 
 $\tilde{X}$. Hence a Koszul resolution, that we can twist by $\cO_\cE(k)$ to get 
 $$0\lra \wedge^{e-1}\cQ^\vee\otimes \cO_\cE(k)\lra\cdots \lra 
 \cQ^\vee\otimes \cO_\cE(k)\lra 
 \cO_\cE(k)\lra \cO_\cE(k)_{|\tilde{X}}\lra 0.$$
 By Bott's theorem, $\pi_*(\wedge^{p}\cQ^\vee\otimes \cO_\cE(k))=S_{k1^{p}}\cE^\vee$, 
 and the higher direct images vanish. Moreover, $\cO_\cE(1)_{|\tilde{X}}=\cO_{\tilde{X}}(-F)$,
 hence $\pi_*(\cO_\cE(1)_{|\tilde{X}})=\pi_*\cO_{\tilde{X}}(-kF)=\cI_Z^k$. So the push-forward of 
 the (twisted) Koszul complex on $\PP_X(\cE)$ yields the desired complex. \qed 
 
 \medskip
In our case, the resulting resolution  $\cF_\bullet\lra\cI_S^2\lra 0$ has the following terms: 
$$\begin{array}{rcccl}
\cF_0&=&S^2\cE^\vee & = & \cO(-2)\boxtimes (\det(\cB_4)\oplus S_{22}\cB_4), \\
\cF_1&=&S_{21}\cE^\vee & = & \cO(-3)\boxtimes \cF_1,\\ 
\cF_2&=&S_{211}\cE^\vee & = & \cO(-4)\boxtimes \cF_2,\\ 
\cF_3&=&S_{2111}\cE^\vee & = & \cO(-5)\boxtimes\cF_3,\\ 
\cF_4&=&S_{21111}\cE^\vee & = & \cO(-6)\boxtimes\cF_4,\\ 
\cF_5&=&S_{211111}\cE^\vee & = & \cO(-7)\boxtimes \wedge^2\cB_4\otimes \det(\cB_4)^3. 
\end{array}$$
After twisting by $\cO(2,1)$, all these terms are clearly acyclic, except possibly the first and last one. 
Using Bott-Borel-Weil, we check that the first term has $h^0=1$, while the last term has $h^6=1$, all the 
other cohomology groups being zero. Hence
$$h^0(\PP(A_5)\times\PP(B_5^\vee),\cI_S^2(2,1))=\CC,$$
which implies our claim.\qed 

\subsection{Equation of the Coble-type hypersurface}
The morphism associating to $v\in A_5^\vee\otimes\wedge^2B_5^\vee$ the Coble-type hypersurface $Q$ is induced by the $G_0$-equivariant morphism 
$$\Gamma : S^2(A_5^\vee\otimes\wedge^2B_5^\vee)\lra S^2A_5^\vee\otimes B_5$$ 
Restricting to the Cartan subspace $\fc$ we get a quadratic map 
$$\gamma : \PP(\fc) \lra \PP(\Theta_3)\subset \PP(S^2A_5^\vee\otimes B_5)$$
that we can write down explicitly. 
Indeed, let $v=v_1h_1+v_2h_2$ (see Lemma \ref{basiscartan}); its contraction with $x=(x_1, \dots, x_5)\in A_5$ is 
$$\omega = x_1(v_1b_{12}+v_2b_{13})+x_2(v_1b_{23}+v_2b_{24})+x_3(v_1b_{34}+v_2b_{35})+x_4(v_1b_{45}+v_2b_{41})+x_5(v_1b_{51}+v_2b_{52}).$$
For a linear form $y\in B_5^\vee$, the point $([x],[y])$ belongs to the hypersurface $Q$ when $\omega$ is decomposable modulo $y$, 
which we can express as the condition that $\omega\wedge\omega=0$ modulo $y$, or equivalently $\omega\wedge\omega\wedge y=0$. We deduce that 
an equation of $Q$ is $v_1^2Q_1-v_1v_2Q_2-v_2^2Q_3=0$, with
$$\begin{array}{rcl}
Q_1 & = & x_2x_4y_1+x_3x_5y_2+x_4x_1y_3+x_5x_2y_4+x_1x_3y_5, \\
Q_2 & = & x_3x_5y_1+x_4x_1y_2+x_5x_2y_3+x_1x_3y_4+x_2x_4y_5, \\
Q_3 & = & x_2x_3y_1+x_3x_4y_2+x_4x_5y_3+x_5x_1y_4+x_1x_2y_5.
\end{array}$$
In particular the image of $\gamma $ just a conic: this is our Göpel variety for this toy example. Note that the action of $G_{16}$ on $\Theta_3$ descends to the usual  action of the icosahedral group on $\PP^2\simeq \PP(\Theta_3)$.

 \section{Rubik's cubes}
 \label{sec_toy_rubik}
For a more interesting case, consider 
three three-dimensional vector spaces $A_1, A_2, A_3$, and 
 the action of $\SL(A_1)\times \SL(A_2)\times \SL(A_3)$ on the $27$-dimensional tensor product 
 $A_1\otimes A_2\otimes A_3$. This was studied in \cite{bhargava-ho} under the name of $3\times 
 3\times 3$ {\it Rubik's cubes},  in connection with genus one curves (over an arbitrary field
 of characteristic $\ne 2,3$). Indeed, one of their main results is that there exists a canonical 
 bijection between "non-degenerate" orbits in  $A_1\otimes A_2\otimes A_3$, and isomorphism 
 classes of genus one curves, decorated in a certain way \cite[Theorem 3.1]{bhargava-ho}. 
 Over the complex numbers, the stability problem for these orbits was discussed in \cite{ng}.
 We would like to add a chapter to this beautiful story.

 \subsection{A Cartan subspace and the Hesse pencil}
 The fact that $A_1\otimes A_2\otimes A_3$ has especially nice properties 
 is related to the fact that it is a theta-representation, 
 being a graded piece of the $\ZZ_3$-grading 
 $$\fe_6=\fsl(A_1)\times \fsl(A_2)\times \fsl(A_3)\oplus 
 (A_1\otimes A_2\otimes A_3)\oplus (A_1\otimes A_2\otimes A_3)^\vee.$$
 According to \cite{vinberg}, a Cartan subspace $\fc$ has rank three and the associated little 
 Weyl group $W_\fc$ is the complex reflection group denoted $G_{25}$ in the Shephard-Todd
 classification. Its order is $648$, and the degrees of a basis of generators of the invariant ring are $6,9,12$. There are $12$ reflection hyperplanes. Moreover, $G_{25}$ is a degree three extension of the Hessian group, the automorphism group 
 of the Hesse pencil of plane cubics with equation
 $$\lambda (u^3+v^3+w^3)-\mu uvw=0.$$
 For the classical properties of the Hesse pencil we refer to \cite{dolg-hesse}. 
 
Consider the three dimensional subspace $\fc$ of $A_1\otimes A_2\otimes A_3$ generated by 
$$\begin{array}{rcl}
h_1 & = & a_1\otimes b_1\otimes c_1 +a_2\otimes b_2\otimes c_2 + a_3\otimes b_3\otimes c_3 ,\\
h_2 & = & a_1\otimes b_2\otimes c_3 +a_2\otimes b_3\otimes c_1 + a_3\otimes b_1\otimes c_2 ,\\
h_3 & = & a_1\otimes b_3\otimes c_2 +a_2\otimes b_1\otimes c_3 + a_3\otimes b_2\otimes c_1,
\end{array}$$
where $a_1,a_2,a_3\in A_1$, $b_1,b_2,b_3\in A_2$, $c_1,c_2,c_3\in A_3$.

\begin{lemma}\label{cartan1}
The subspace $\fc$ is a Cartan subspace. 
\end{lemma}

\proof We need to prove that $h_1, h_2, h_3$ are semisimple and commute. That they commute easily follows from the fact the Lie bracket of $\fe_6$, restricted to $A_1\otimes A_2\otimes A_3$, must be a non-zero multiple of the morphism
\begin{align*}
\wedge^2(A_1\otimes A_2\otimes A_3) &\lra A_1^\vee\otimes A_2^\vee\otimes A_3^\vee, \\
(a_1\otimes a_2\otimes a_3)\wedge (a'_1\otimes a'_2\otimes a'_3) 
&\mapsto (a_1\wedge a'_1)\otimes (a_2\wedge a'_2)\otimes (a_3\wedge a'_3),
\end{align*}
where we have chosen volume forms on $A_1, A_2, A_3$ in order to identify $\wedge^2A_i$ with $A_i^\vee$. Now, the fact that $h_1$ and $h_2$ commute is a direct consequence of the combinatorial fact that 
any pure tensor in $h_1$ and any pure tensor in $h_2$ have a common factor. 

\smallskip Let us prove that $h_1$ is semisimple. Let $X_i=a_i\otimes b_i\otimes c_i$ and $Y_i=a^\vee_i\otimes b^\vee_i\otimes c^\vee_i$. 
Let also $Z_1$ be the element of $\fsl(A_1)\times \fsl(A_2)\times \fsl(A_3)$
acting diagonally on the basis $(a_1,a_2,a_3)$, $(b_1,b_2,b_3)$, $(c_1,c_2,c_3)$ as $\diag(\frac{2}{3},-\frac{1}{3},-\frac{1}{3})$; and similarly for $Z_2$ and $Z_3$, so that $Z_1+Z_2+Z_3=0$. 
We have commutation relations 
$$\begin{array}{lll}
    [Z_i,X_i]=2X_i, \qquad  &[Z_i,X_j]=-X_j,\qquad  &[X_i,Y_i]=Z_i, 
    \\   \lbrack Z_i,Y_i\rbrack =-2Y_i, &[Z_i,Y_j]=Y_j, &[X_i,X_j]=Y_k,
\end{array}$$
where $i\ne j$ and $(i,j,k)=(1,2,3)$ up to cyclic order. We deduce that the elements $X_i,Y_j,Z_k$ of $\fe_6$ generate a Lie subalgebra isomorphic with $\fsl_3$, with an explicit isomorphism sending $X_i$ to $e_j^\vee \otimes e_k \in \fsl_3$ and $Y_i$ to $e_k^\vee \otimes e_j\in \fsl_3$. In particular 
$$h_1=X_1+X_2+X_3\mapsto\begin{pmatrix} 0&0&1\\1&0&0\\0&1&0\end{pmatrix}$$
is semisimple, and of course $h_2$ and $h_3$ as well. 
\qed 

\medskip
Consider the centralizer $Z^+(\fc)$ of $\fc$ in $\SL(A_1)\times \SL(A_2)\times \SL(A_3)$. We claim this is a group of order $243$. Indeed, let $\tau$ be the diagonal 
element of $\SL_3$ with eigenvalues $1,\zeta,\zeta^2$ for some fixed primitive cubic root of unity 
$\zeta$. Let  $\sigma$ be the transformation that sends $a_i, b_i, c_i$ to $a_{i+1},b_{i+1},c_{i+1}$, where $i$ is considered modulo three. Then $Z^+(\fc)$ is generated by $\sigma$ and the diagonal transformations
of type $(\tau,\tau,\tau)$, $(\id,\tau,\tau^2)$, up to permutation of the order, and $(\zeta_1 \id, 
\zeta_2 \id,\zeta_3 \id)$ for $\zeta_1,\zeta_2,\zeta_3$ cubic roots of unity such that $\zeta_1\zeta_2\zeta_3=1$. The latter transformations act trivially on $A_1\otimes A_2\otimes A_3$
and we deduce an exact sequence 
$$1\lra \ZZ_3^2 \lra Z^+(\fc)\lra Z(\fc)\lra 1.$$
Here $Z(\fc)$ is the image of $Z^+(\fc)$ inside $\SL(A_1\otimes A_2\otimes A_3)$; this is an
abelian group of order $27$, while $Z^+(\fc)$ is  a (non-abelian) Heisenberg group (see \ref{sec_heis_group}). 

 \subsection{Six genus one curves}
 That a generic tensor $t\in A_1\otimes A_2\otimes A_3$ defines a genus one 
 curve is quite  straightforward: we can see $t$ as a morphism from $A_3^\vee$ to $A_1\otimes A_2$; when it is 
 injective, 
 its image is a three-dimensional space of linear forms on $A_1^\vee\otimes A_2^\vee$,
 that cut out $\PP(A_1^\vee)\times \PP(A_2^\vee)\simeq\PP^2\times\PP^2$, generically, along a smooth curve $C_{12}$
 with trivial canonical bundle. Alternatively, we get over $\PP(A_3^\vee)$ 
 a morphism from  $A^\vee_1\otimes\cO(-1)$ to  $A_2\otimes\cO$ that degenerates along a cubic curve $C_3$. This gives 
 six different ways to define the same abstract genus one curve $C$; we denote these copies of $C$ by $C_i\subset \PP(A_i^\vee)$ for $1\leq i\leq 3$ and $C_{ij}\subset \PP(A_i^\vee)\times \PP(A_j^\vee)$ for $1\leq i<j\leq 3$. But we also get natural morphisms between these six models of $C$: typically, $C_1$ is defined by the condition that 
 $t(x,\bullet, \bullet)$ is degenerate; while $C_{12}$ is defined by the condition that 
 $t(x,y, \bullet)=0$, which of course implies that $x$ belongs to $C_1$ and $y$ to $C_2$. 
 Hence a diagram 
 $$\xymatrix{
  & C_{12}\ar[rd]\ar[ld] &  \\
  C_1 & & C_2 \\
  C_{13}\ar[u]\ar[rd] && C_{23}\ar[ld]\ar[u] \\
   & C_3 &
   }$$
 where all the arrows are isomorphisms. 
 
Now suppose that $t=uh_1+vh_2+wh_3$ belongs to $\fc$. Let us set coordinates $(x_1,x_2,x_3)$ on $A_1$, as well as $(y_1,y_2,y_3)$ on $A_2$ and $(z_1,z_2,z_3)$ on $A_3$. 
Since the elliptic curve $C_1$ in $\PP(A_1^\vee)$ is defined by the condition that $t(x,\bullet,\bullet)$ is degenerate, its equation is  
$$0=\det \begin{pmatrix} ux_1 & vx_2 & wx_3\\ wx_2&ux_3&vx_1\\ vx_3&wx_1&ux_2\end{pmatrix} = x_1x_2x_3(u^3+v^3+w^3)-(x_1^3+x_2^3+x_3^3)uvw.$$ 
So we directly recover the Hesse pencil! 

Note that $C_1$ has $6$ marked points, which are typically distinct, namely $[u,v,w]$ and its permutations, plus of course the $9$ base points of 
the pencil. The equations of $C_2$ and $C_3$ are exactly the same, once we replace 
 $(x_1,x_2,x_3)$ by $(y_1,y_2,y_3)$  and $(z_1,z_2,z_3)$, respectively. 
 
The correspondence $p_{12}$ from $C_1$ to $C_2$ is defined by associating to $[x]$ 
such that $t(x,\bullet, \bullet)$ is degenerate, the unique $[y]$ such that  $t(x,y, \bullet)=0$.

\begin{lemma}\label{tau}
$p_{12}$ is the translation of the curve $C$ by 
$$\tau=[u,v,w]-[0,1,-1]=[v,w,u]-[1,-1,0]=[w,u,v]-[1,0,-1]. $$
\end{lemma}

\proof A straightforward computation shows that 
$$y=(u^2x_2x_3-vwx_1^2, v^2x_1x_3-uwx_2^2, w^2x_1x_2-uvx_3^2).$$
This is the restriction to $C_1$ of a standard plane Cremona transformation, 
with three base points that belon to the curve, namely $[u,\epsilon v, \epsilon^2 w]$ for  $\epsilon^3=1$. One can check that
$$\det \begin{pmatrix} u & v & w\\ x_1&x_2&x_3\\ y_1&y_3&y_2\end{pmatrix} =0,$$
which means that the correspondence $p_{12}$ is obtained by composing 
the symmetry of $C_1$ with respect to the point $[u,v,w]$, chosen as the origin of the group law, with the exchange of the last two coordinates. Note that the latter involution
is also a symmetry, with respect to $[0,1,-1]$, one of the base points of the Hesse 
pencil. So $p_{12}$, being the composition of two symmetries, is finally a translation. \qed

\medskip The correspondences $p_{23}$ and $p_{31}$ are given exactly by the same 
formulas. Their inverses  $p_{32}$ and $p_{13}$, as well as $p_{21}$ are obtained
by exchanging for example $v$ and $w$. Note in particular that the composition 
$p_{31}\circ p_{23}\circ p_{12}$ is not the identity, as one could have naively 
expected, but a translation by $3\tau$, as already noticed by Cayley. See \cite[Sect. 5]{dolg-hesse} for a modern treatment. 

 \subsection{A quadric hypersurface of Coble type}
We will be interested in a different geometric construction, based on Cayley's hyperdeterminant \cite[Chapter 14]{GKZ}. This is a quartic polynomial in eight variables, that provides an
equation of the dual hypersurface to $\PP^1\times\PP^1\times\PP^1\subset\PP^7$. In particular, it provides an equivariant map 
$$\mathrm{Ca} : S^4(U_1\otimes U_2\otimes U_3)\lra \det(U_1)^2\otimes  \det(U_2)^2\otimes  \det(U_3)^2,$$
if the three vector spaces $U_1,U_2,U_3$ are two-dimensional. 

In order to be able to use Cayley's hyperdeterminant in our three-dimensional setting, 
we pass to a relative model. 
Denote by $Q_1, Q_2, Q_3$ the rank two quotient vector bundles,  respectively on $\PP(A_1), \PP(A_2), \PP(A_3)$. The tensor $t\in A_1\otimes A_2\otimes A_3$ defines a global section of the exterior product $Q_1\boxtimes Q_2\boxtimes Q_3$. Each fiber of this bundle is a
copy of $\CC^2\otimes\CC^2\otimes \CC^2$, whose $GL_2\times GL_2\times GL_2$-orbits are well-known:
apart from the trivial orbit closures, one has completely decomposable tensors (the cone over $\PP^1\times\PP^1\times \PP^1)$, partially decomposable tensors (three copies of $\PP^1\times \PP^3$), and then a quartic hypersurface whose equation is Cayley's hyperdeterminant.
Correspondingly, the locus in 
$\PP(A_1)\times\PP(A_2)\times \PP(A_3)$ where the image of $v$ in 
$Q_1\boxtimes Q_2\boxtimes Q_3$ becomes completely decomposable, is in general
a smooth surface $\Sigma$. The locus where this image  falls inside the invariant quartic 
is a hypersurface $\mathfrak{H}$  of multidegree $(2,2,2)$, given by the 
Cayley hyperdeterminant seen as a vector bundle morphism 
$$S^4(Q_1\boxtimes Q_2\boxtimes Q_3)\lra \det(Q_1)^2\boxtimes 
\det(Q_2)^2\boxtimes \det(Q_3)^2=\cO(2,2,2).$$
In other words, $\mathfrak{H}$ is a quadric section of the triple Segre product 
$\PP(A_1)\times\PP(A_2)\times \PP(A_3)$, embeddded in $\PP^{26}$. This hypersurface $\mathfrak{H}$ and the 
surface $\Sigma$ are instances of what we call {\it orbital degeneray loci},
about which we refer to \cite{bfmt, bfmt2}.

 \begin{prop}\label{triplets}
The locus $\Sigma$ is an abelian surface. More precisely $\Sigma$ embeds into 
 $\PP(A_1^\vee)\times \PP(A_2^\vee)\times \PP(A_3^\vee)$
 as the surface of colinear triples in $C_1\times C_2\times C_3$. In particular, 
 $\Sigma\simeq C\times C$.
 \end{prop}

\proof 
 By definition, over a point $(L_1,L_2,L_3)$ of $\Sigma$, there exist vectors 
 $q_1,q_2,q_3$ such that 
 $$ t= q_1\otimes q_2\otimes q_3 \qquad \mathrm{mod} \quad 
 L_1\otimes V_2\otimes V_3+V_1\otimes L_2\otimes V_3+V_1\otimes V_2\otimes L_3.$$
 Note that for $t$ generic, $q_i$ cannot belong to $L_i$, as follows from a dimension count. 
 If we contract $t$ by a linear form $\phi_1$ that vanishes on $q_1$ and $L_1$, we get 
 a tensor in $L_2\otimes V_3+V_1\otimes V_2\otimes L_3\subset V_2\otimes V_3$,
 which cannot be of maximal rank. This exactly means that $[\phi_1]$ belongs to $C_1\subset \PP(A_1^\vee)$. In fact, another dimension count shows that this tensor cannot be of rank one, so that we can decompose it as 
 $$t(\phi_1,\bullet ,\bullet) = \ell_2\otimes a_{13}+a_{12}\otimes\ell_3$$
 for some generators $\ell_2$ of $L_2$ and $\ell_3$ of $L_3$, and some vectors 
 $a_{13}\in V_3$ and $a_{12}\in V_2$.
 Similarly we can find $\phi_2\in \langle q_2,L_2\rangle$ and $\phi_3\in \langle q_3,L_3\rangle$ so that we can decompose 
 $$t(\bullet, \phi_2,\bullet) = \ell_1\otimes a_{23}+a_{21}\otimes\ell_3, \qquad 
 t(\bullet ,\bullet, \phi_3) = \ell_1\otimes a_{32}+a_{31}\otimes\ell_2.$$
 Let $\psi_{ij}\in A_j^\vee$ be a nonzero linear form, orthogonal to $\ell_j$ and $a_{ij}$. Then 
 $$t(\phi_1,\psi_{12},\bullet) = t(\phi_1,\bullet,\psi_{13})=t(\psi_{21},\phi_2,\bullet) = t(\bullet,\phi_2,\psi_{23})=t(\psi_{31},\bullet,\phi_3) = t(\bullet,\psi_{32},\phi_3)=0.$$
 In particular $\psi_{ij}$ defines a point in $C_j$, and by Lemma \ref{tau} and the remark following its proof we have in $Jac(C)$ (with some abuse of notations) the relations  
 $$\tau = \psi_{12}-\phi_1=\phi_1-\psi_{13}=\phi_2-\psi_{21}=\psi_{23}-\phi_2=
 \psi_{31}-\phi_3=\phi_3-\psi_{32}.$$
 Since $\ell_1$ is orthogonal to $\phi_1,\psi_{21}$ and $\psi_{31}$, the latter define 
 three colinear points in $\PP(A_1)$, so their sum in $Jac^3(C)$ is the hyperplane class
 $h_1$ of $C_1$. We thus get the relations 
 $$h_1 = \phi_1+\psi_{21}+\psi_{31}, \quad h_2 = \phi_2+\psi_{12}+\psi_{32}, 
 \quad h_3 = \phi_3+\psi_{13}+\psi_{23}, $$
 where $h_i$ is the hyperplane class of $C_i$.
 Summing them up and using the previous relations, we get 
 $$h_1+h_2+h_3=3(\phi_1+\phi_2+\phi_3).$$
 Since the left hand side is constant, we deduce that $\phi_1+\phi_2+\phi_3$ is also
 constant in $Jac^3(C)$, which is exactly our claim. \qed

 \medskip\noindent
{\it Remark.} 
 The hypersurface $\mathfrak{H}$ is, by construction, a quadratic section of $\PP(A_1)\times \PP(A_2)\times \PP(A_3)$
 that is singular along the abelian surface $\Sigma$. We expect $\mathfrak{H}$ to be uniquely characterized by this property, although its singular locus should be strictly bigger. 
 
\subsection{The G\"opel variety}
When $t$ belongs to $\fc$, we can compute explicitly the equation of $\mathfrak{H}$. 
Indeed, on a open subset of $\PP(A_1)\times\PP(A_2)\times\PP(A_3)$, the quotient bundles $Q_1, Q_2, Q_3$ are 
obtained by modding out by relations of type $a_3=x_1a_1+x_2a_2$, $b_3=y_1b_1+y_2b_2$, $c_3=z_1a_1+z_2a_2$, which allows to express the section of  $Q_1\otimes Q_2\otimes Q_3$ defined by $t$ as 
$$\begin{array}{rcl}
\bar{t} & =& 
u(1+x_1y_1z_1)a_1\otimes b_1\otimes c_1+(ux_1y_1z_2+vy_1+wx_1)a_1\otimes b_1\otimes c_2+ \\ 
& & +(ux_1y_2z_1+vx_1+wz_1)a_1\otimes b_2\otimes c_1+(ux_2y_1z_1+vz_1+wy_1)a_2\otimes b_1\otimes c_1+ \\
& & +(ux_1y_2z_2+vy_2+wz_2)a_1\otimes b_2\otimes c_2+(ux_2y_1z_2+vz_2+wx_2)a_2\otimes b_1\otimes c_2 +\\
& & +(ux_2y_2z_1+vx_2+wy_2)a_2\otimes b_2\otimes c_1+u(1+x_2y_2z_2)a_2\otimes b_2\otimes c_2.
\end{array}$$
Now we need to take the hyperdeterminant of this expression. Recall from  \cite[Chapter 14, Proposition 1.7]{GKZ} that for a $2\times 2\times 2$-tensor 
$s=\sum_{i,j,k=1,2}s_{ijk}a_i\otimes b_j\otimes c_k$, the hyperdeterminant is 
$$Det(s)=\sum s_{111}^2s_{222}^2-2\sum s_{111}s_{222}s_{112}s_{221}+4\sum s_{111}s_{122}s_{212}s_{221},$$
where each sum is obtained by permuting $1$ and $2$ on the first, second and third indices
(so that the first sum has $4$ terms, the second  one has $6$, the last one has $2$).
After plugging in the coefficients of  $\bar{t}$, and homogeneizing, we deduce 
the following expression:

$$Q(u,v,w)=\sum_{p+q+r=4}u^pv^qw^rQ_{p,q,r}, $$

$$\begin{array}{rcl}
 Q_{4,0,0} & = & x_1^2y_1^2z_1^2 + x_2^2y_2^2z_2^2 + x_3^2y_3^2z_3^2- 2x_1x_2y_1y_2z_1z_2 -2x_1x_3y_1y_3z_1z_3- 2x_2x_3y_2y_3z_2z_3,\\
 Q_{0,4,0} & = & x_1^2y_3^2z_2^2+x_2^2y_1^2z_3^2+ x_3^2y_2^2z_1^2 - 2x_1x_2y_1y_3z_2z_3- 2x_1x_3y_2y_3z_1z_2- 2x_2x_3y_1y_2z_1z_3 ,\\
 Q_{0,0,4} & = & x_1^2y_2^2z_3^2  + x_2^2y_3^2z_1^2+ x_3^2y_1^2z_2^2- 2x_1x_2y_2y_3z_1z_3 - 2x_1x_3y_1y_2z_2z_3- 2x_2x_3y_1y_3z_1z_2,
 \end{array}$$
 
$$\begin{array}{rcl}
 -\frac{1}{4}Q_{3,1,0} & = &  -\frac{1}{4}Q_{0,1,3} \;\;=\;\;
 x_2x_3y_1y_2z_1z_3+x_1x_2y_1y_3z_2z_3+x_1x_3y_2y_3z_1z_2,\\
  -\frac{1}{4}Q_{3,0,1} & = & -\frac{1}{4}Q_{0,3,1} \;\;=\;\;
   x_1x_2y_2y_3z_1z_3+ x_1x_3y_1y_2z_2z_3+x_2x_3y_1y_3z_1z_2,\\
  -\frac{1}{4}Q_{1,3,0} & = & -\frac{1}{4}Q_{1,0,3} \;\;=\;\;
  x_1x_2y_1y_2z_1z_2+ x_1x_3y_1y_3z_1z_3+ x_2x_3y_2y_3z_2z_3,
 \end{array}$$
  
$$\begin{array}{rcl}
  -\frac{1}{2}Q_{0,2,2} & = & x_1x_2y_1y_2z_3^2+x_2^2y_1y_3z_1z_3+x_1x_3y_2^2z_1z_3
 +x_2x_3y_2y_3z_1^2+ \\
 & & \qquad  +x_2x_3y_1^2z_2z_3+x_1^2y_2y_3z_2z_3+x_1x_2y_3^2z_1z_2+x_3^2y_1y_2z_1z_2 +x_1x_3y_1y_3z_2^2,\\
  -\frac{1}{2}Q_{2,0,2} & = & x_1^2y_1y_2z_1z_3+x_1x_2y_1y_3z_1^2 +x_1x_2y_2^2z_2z_3+ x_1x_3y_1^2z_1z_2+\\
  & & \qquad +x_2^2y_2y_3z_1z_2+x_2x_3y_1y_2z_2^2+x_1x_3y_2y_3z_3^2+x_2x_3y_3^2z_1z_3 +x_3^2y_1y_3z_2z_3,\\
   -\frac{1}{2}Q_{2,2,0} & = & x_1x_2y_1^2z_1z_3+ x_1x_3y_1y_2z_1^2+ x_2^2y_1y_2z_2z_3+ x_1^2y_1y_3z_1z_2+\\
  & & \qquad +x_2x_3y_2^2z_1z_2+ x_1x_2y_2y_3z_2^2+x_2x_3y_1y_3z_3^2+x_3^2y_2y_3z_1z_3+ x_1x_3y_3^2z_2z_3,
 \end{array}$$
  
$$\begin{array}{rcl}
   -\frac{1}{4}Q_{2,1,1} & = & x_2^2y_2^2z_1z_3+ x_2x_3y_1^2z_1^2+ x_1^2y_2y_3z_1^2+ x_1^2y_1^2z_2z_3+\\ & & \qquad +x_2^2y_1y_3z_2^2+ x_1x_3y_2^2z_2^2+x_1x_2y_3^2z_3^2+
   x_3^2y_1y_2z_3^2+x_3^2y_3^2z_1z_2,\\
  -\frac{1}{4}Q_{1,2,1} & = & x_1x_2y_2^2z_1^2+x_2^2y_1^2z_1z_2+x_1^2y_1y_2z_2^2 + 
  x_1x_3y_1^2z_3^2+\\ & & \qquad +x_2^2y_2y_3z_3^2+ x_1^2y_3^2z_1z_3+ x_3^2y_1y_3z_1^2+ x_3^2y_2^2z_2z_3+ x_2x_3y_3^2z_2^2,\\
  -\frac{1}{4}Q_{1,1,2} & = & x_2^2y_1y_2z_1^2+x_1^2y_2^2z_1z_2+x_1x_2y_1^2z_2^2+
  x_1^2y_1y_3z_3^2 +\\ & & \qquad + x_2x_3y_2^2z_3^2+ x_3^2y_1^2z_1z_3 + x_1x_3y_3^2z_1^2 + x_2^2y_3^2z_2z_3 + x_3^2y_2y_3z_2^2.
 \end{array}$$
 
 \smallskip\noindent
 In particular we get independent polynomials, up to the 
 three relations 
  \begin{equation}\label{310=013} 
  Q_{3,1,0}=Q_{0,1,3}, \qquad Q_{0,3,1}=Q_{3,0,1}, \qquad Q_{1,3,0}=Q_{1,0,3}.
  \end{equation}
 We deduce:

 \begin{prop}\label{toy_gopel} 
 The span of the G\"opel variety is a copy of   $\PP^{11}=\PP(\Theta_{12})$, where $\Theta_{12}$ is the Macdonald representation of the Hessian group $G_{25}$ generated by the fourth powers of the $12$ lines in the Hessian arrangement. 
The  G\"opel variety is the image of the bijective projection of $v_4(\PP^2)$ from the 
 plane $$\Pi=\PP\langle u(v^3-w^3), v(w^3-u^3), w(u^3-v^3)\rangle.$$ 
 Its non-normal locus consists in three pinch points. 
 Its ideal in $\PP^{11}$ is generated by $36$ quadrics, admitting $165$  syzygies ($156$ linear and $9$ quadratic syzygies). 
 \end{prop}
 
\begin{proof}
 The fact that the G\"opel variety is the image of the projection of $v_4(\PP_2)$ from $\Pi$ is just a reformulation of equations (\ref{310=013}). 
This projection can be explicitely written as the morphism
 $$[u,v,w]\mapsto [u^4, v^4, w^4, u^3(v+w), v^3(u+w), w^3(u+v), u^2v^2, u^2w^2, v^2w^2, u^2vw, uv^2w, uvw^2],$$ 
 whose image is stable under the action $W_\fc$. Indeed, the $12$ lines of the Hessian arrangement are the three axis in $\PP^2$ plus the nine lines of equations $\ell= \xi u+\zeta v+\eta w=0$ for $\xi^3=\zeta^3=\eta^3=\xi\eta\zeta$.  These equations ensure that in $\ell^4$, 
 $u^3v$ and $u^3w$ have the same coefficient, and the claim easily follows. 
 
 The Jacobian matrix of the above parametrization fails to have maximal rank at the three cordinate points  $(1:0:0)$, $(0:1:0)$, $(0:0:1) \in \PP^2$, whose images in $\PP^{11}$
 are:
\begin{align*}
    p_1&=(1:0:0:0:0:0:0:0:0:0:0:0) \in \PP^{11}, \\
    p_2&=(0:1:0:0:0:0:0:0:0:0:0:0) \in \PP^{11}, \\
    p_3&=(0:0:1:0:0:0:0:0:0:0:0:0) \in \PP^{11}.
\end{align*}
These points are singular for the Göpel variety, as one can prove by noting that the Jacobian matrix of a system of generators of the ideal of the Göpel variety in $\PP^{11}$ has rank 7 at $p_1$, $p_2$ and $p_3$.
Using the explicit equations of secants varieties of Veronese varieties available from \cite{kanev:chordal, landsberg-ottaviani}, we checked with \cite{Macaulay2} that the center of projection $\Pi$ cuts the secant variety of $v_4(\PP^2)$ at the three points
\begin{align*}
    q_1&=(0:1:-1:0:0:0:0:0:0:0:0:0:0:0:0) \in \PP^{14}, \\
    q_2&=(0:0:0:0:0:0:1:0:0:0:0:-1:0:0:0) \in \PP^{14}, \\
    q_3&=(0:0:0:0:0:0:0:0:0:1:0:0:0:-1:0) \in \PP^{14},
\end{align*}
where we use coordinates indexed by a basis of monomials of degree $4$ in $u,v,w$ in the RevLex monomial order. 
So the Göpel variety has precisely three singular points, which must then be the pinch points $p_1$, $p_2$, $p_3$, each of these points having a unique preimage in $\PP^{14}$. This last sentence is justified, since $q_1$, $q_2$ and $q_3$ actually lie on the tangential variety of $v_4(\PP^2)$. In conclusion, the restriction of $\pi$ to $v_4(\PP^2)$ is bijective onto its image, which is the Göpel variety.
Since the latter is singular, we deduce that it is not normal, for otherwise it would be isomorphic to the smooth variety $v_4(\PP^2)$ by Zariski's main theorem.

The last claim is the result of a computation with \cite{Macaulay2} (see attached file {\bf gopel\_rubik}).
\end{proof}

\medskip\noindent
{\it Remark.}  For a generic projection, according to \cite{Macaulay2} the ideal of the image should be generated by $33$ quadrics admitting $123$ linear syzygies (see attached file {\bf gopel\_rubik}). 

\section{Hyperelliptic curves and pencils of quadrics}
\label{sec_hyperell}

The relationship between hyperelliptic curves and pencils of quadrics is very classical \cite{Ne68,NR69,DR77}. In this section 
we observe that this has interesting connections with Vinberg's theory, and we reconsider 
the situation  from our point of view of Coble type hypersurfaces. Our starting point is the classical $\ZZ_2$-grading 
$$\fsl_m= \fso_m\oplus S^{\langle 2\rangle}\CC^m,$$
where $S^{\langle 2\rangle}\CC^m\subset S^2\CC^m$ is the space of traceless symmetric matrices. 
As usual $\fso_m$ is seen as the subalgebra of skew-symmetric matrices in $\fsl_m$, which means that it kills the quadric $Q_1$ whose matrix is the identity. A Cartan subspace $\fc$ is then simply the space 
of diagonal (hence symmetric), traceless matrices, which is also the standard Cartan subalgebra of $\fsl_m$; when $m=2g+2$, the (Heisenberg) group centralizing $\fc$ is an extension of $\ZZ_2^{2g}$ by $\ZZ_4$ when $g+1$ is odd and $\ZZ_2\times \ZZ_2$ when $g+1$ is even (see Appendix \ref{sec_heis_group}). 

\subsection{Coble quadrics in orthogonal Grassmannians}\label{cobleOG}
Given a pencil of quadrics 
$\langle Q_1, Q_2\rangle $ in $\PP^{2n-1}$, 
such that $Q_1$ is non-degenerate, 
the variety of planes which are isotropic with respect to both $Q_1$ and $Q_2$ is a 
subvariety of the copy of the isotropic Grassmannian $OG(2,2n)$ defined by $Q_1$,
described as the zero-locus of the global section of $S^2U^\vee$ defined by an equation of $Q_2$, where $U$ denotes the tautological subbundle. 
If $Q_2$ is general enough, the latter is the singular locus of the  hypersurface 
in $OG(2,2n)$ consisting in planes on which the restriction of $Q_2$ is degenerate. 
This hypersurface is also defined by the condition that the projection of $S^2(S^2U^\vee)$ to $S_{22}U^\vee=\mathcal{O}(2)$ vanishes. 
So we get a quadratic covariant in the space of quadratic sections of $OG(2,2n)$.

\smallskip
The condition for the restriction to be degenerate on a plane $V=\langle e,f\rangle$ is simply 
$$Q_2(e,e)Q_2(f,f)-Q_2(e,f)^2=0,$$
that we need to express in terms of the Pl\"ucker coordinates $(p_{ij})$ of $V$, with 
respect to some fixed basis. For this we may suppose that in the chosen coordinates, 
$$Q_1=x_1^2+\cdots +x_{2n}^2, \quad Q_2=\alpha_1x_1^2+\cdots +\alpha_{2n}x_{2n}^2,$$
with $\alpha_1,\ldots , \alpha_{2n}$ pairwise distinct. The condition then becomes 
$$\sum_{i,j}\alpha_i\alpha_j(e_if_j-e_jf_i)^2=\sum_{i,j}\alpha_i\alpha_jp_{i,j}^2=0.$$

\smallskip\noindent
 This readily extends to any isotropic Grassmannian $OG(k,m)$, in which the 
 locus of $k$-planes on which $Q_2$ degenerates is the quadratic section of 
 equation $C(\alpha)=0$, where 
 \begin{equation}\label{calpha}
 C(\alpha):=\sum_{i_1<\cdots <i_k}\alpha_{i_1}\cdots \alpha_{i_k}p_{i_1,\ldots ,i_k}^2.
 \end{equation}
 
 \subsection{Square-free monomials}
 Because of the form of this equation, it will be useful to understand the representation
 of the symmetric group $S_{2n}$ on the space  $SF_{2n,n}$ of homogeneous polynomials
 of degree $n$ spanned by square-free monomials. More generally, denote by $SF_{m,k}$ the 
  representation
 of the symmetric group $S_{m}$ on the space of homogeneous polynomials
 of degree $k$ spanned by square-free monomials in $m$ variables.
 Its decomposition into irreducible 
 components is as follows, where we use the traditional notation $[\lambda]$ for the irreducible representation of $S_m$ defined by the partition $\lambda$ of $m$ \cite{FH91}.
 
 \begin{prop}\label{specht} For any $m\ge 2k$, $SF_{m,k}$ is multiplicity free and decomposes as 
 $$SF_{m,k}\simeq \bigoplus_{p=0}^k [m-p,p].$$
 \end{prop}
 
 \proof First observe that the hook-length formula gives 
 $$\dim [m-p,p] =\frac{(m-2p+1) m!}{p!(m-p+1)!}= \binom{m+1}{p}-2\binom{m}{p-1}.$$
 We deduce the following sum of dimensions:
 $$ \sum_{p=0}^k \dim [m-p,p] = \sum_{p=0}^{k}\Big(\binom{m}{p}+\binom{m}{p-1}\Big)-2\sum_{p=0}^{k}\binom{m}{p-1} 
 = \binom{m}{k} = \dim SF_{m,k}.$$
  So it suffices to check that each representation $[m-p,p]$, for $p\le k$, embeds inside $SF_{m,k}$. 
  For this we need to recall the description of the irreducible representations of
  $S_{m}$ as Specht modules. Typically, to each tableau $T$ of shape $(m-p,p)$,
  numbered by the integers $1,\ldots, m$, we associate the polynomial 
  \begin{equation}\label{P_T}
  P_T(x)=(\alpha_{a_1}-\alpha_{b_1})\cdots (\alpha_{a_{p}}-\alpha_{b_{p}}),
  \end{equation}
  where $a_1,\ldots , a_{m-p}$ are the integers appearing on the first row of $T$, from left to right, and  $b_1,\ldots , b_{p}$ are the integers appearing on the second row. 
  Then the span of the polynomials $P_T$, for $T$ of shape $(m-p,p)$, inside the space of homogeneous polynomials of degree $p$, is a copy of the irreducible representation
  $[m-p,p]$. Of course these polynomials are not linearly independent, but their relations are well-understood. It is proved in \cite{fulton-YoungTableaux} that the module of relations is 
  generated by the following ones, where $T$ is any tableau of shape $(m-p,p)$:
  \begin{itemize} 
  \item ({\it skew-symmetry relations}) if $S$ if obtained from $T$ by exchanging two numbers 
  in the same column, then $P_S+P_T=0$;
  \item ({\it Pl\"ucker type relations}) given two consecutive columns of length two, let $R$ and $S$ 
  be obtained by fixing the top leftmost entry and permuting cyclically the three other entries, then $P_R+P_S+P_T=0$;
  \item ({\it linear type relations}) given two consecutive columns of length two and one, let 
  $R, S$ be obtained by permuting cyclically their  three entries, then again $P_R+P_S+P_T=0$.
  \end{itemize}
  In the latter case, if  $a,b,c$ are the three entries, the relation is obtained from the obvious 
  linear relation 
  $$(\alpha_a-\alpha_b)+(\alpha_b-\alpha_c)+(\alpha_c-\alpha_a)=0,$$
  multiplied by the contribution of  the other columns. In the previous case, if $a,b,c,d$ are the 
  four entries, $a$ being fixed, the  relation is obtained from the quadratic Pl\"ucker relation 
  $$(\alpha_a-\alpha_b)(\alpha_c-\alpha_d)+(\alpha_a-\alpha_c)(\alpha_d-\alpha_b)+
  (\alpha_a-\alpha_d)(\alpha_b-\alpha_c)=0, $$
  multiplied again by the contribution of  the other columns.
  
  We define a nonzero equivariant map from $[m-p,p]$ to $SF_{m,k}$ as follows. 
  For each tableau $T$ of shape $(m-p,p)$, we let 
  $$ Q_T(\alpha)=P_T(\alpha)
%  (\alpha_{a_1}-\alpha_{b_1})\cdots (\alpha_{a_{n-p}}-\alpha_{b_{n-p}})
  \times e_{k-p}(\alpha_{c_1},\ldots , \alpha_{c_{m-2p}}),$$
  where $\{c_1,\ldots , c_{m-2p}\}$ is the complement of  $\{a_1,\ldots , a_{m-p}, b_1,
  \ldots , b_{p}\}$. As
  usual the elementary symmetric function $e_{k-p}(\alpha_{c_1},\ldots , \alpha_{c_{m-2p}})$ 
  is the sum of all square-free 
  monomials of degree $k-p$ in these variables, so that $Q_T$ belongs to $SF_{m,k}$. 
  By mapping $P_T$ to $Q_T$, we claim that
  we define a non trivial equivariant map from $[m-p,p]$ to $SF_{m,k}$. Indeed, this 
  will be clear if we can prove that all the linear relations between the $P_T$'s
  are also verified by the $Q_T$'s. This is obvious  for the skew-symmetry relations. 
  This is also clear for the Pl\"ucker type relations  $P_R+P_S+P_T=0$, since  $Q_R, Q_S, Q_T$
  are obtained by multiplying  $P_R, P_S, P_T$ by the same polynomial. 
  
  The case of linear type relations
  is slightly different. In this case $P_T=A(\alpha_a-\alpha_b)$ for $A$ defined by the first $p-1$ columns of $T$, and then $P_R=A(\alpha_b-\alpha_c)$, $P_S=(\alpha_c-\alpha_a)A$. To obtain $Q_R$, one needs to multiply $P_R$ by an elementary symmetric function of degree $k-p$ in $m-2p$ variables including $\alpha_c$, and we can decompose this symmetric function as $\alpha_ce_{k-p-1}+e_{k-p}$ where 
  $e_{k-p-1}, e_{k-p}$ involve the remaining variables. Similarly, we obtain $Q_S$ by multiplying 
  $P_S$ by $\alpha_ae_{k-p-1}+e_{k-p}$ and $P_T$ by $\alpha_be_{k-p-1}+e_{k-p}$. Finally, the 
  relation  $Q_R+Q_S+Q_T=0$ follows from the identity 
  $$(\alpha_a-\alpha_b)\alpha_c+(\alpha_b-\alpha_c)\alpha_a+(\alpha_c-\alpha_a)\alpha_b=0,$$
  and the proof is complete. \qed 
  
  \medskip Now consider the equation $C(\alpha)$ as an equivariant
  morphism 
  \begin{equation}\label{rhomk}
  \rho_{m,k}: SF_{m,k} \lra H^0(OG(k,m),\cO(2)),
  \end{equation}
  sending $\alpha_I$ to $p_I^2$, for each $k$-tuple 
  of integers between $1$ and $m$.
  
  \begin{prop}\label{Calpha}
  The image of  $\rho_{m,k}$ is the irreducible representation $[m-k,k]$.
  \end{prop}
  
\proof There is a natural injective equivariant map $\sigma_{m,k-1}$ 
from $SF_{m,k-1}$ to $SF_{m,k}$, sending the monomial $\alpha_J$, 
for each $(k-1)$-tuple $J$
  of integers between $1$ and $m$, to its product by $\sum_{a\notin J}\alpha_a$. The previous Proposition implies that the cokernel is 
  the irreducible representation $[m-k,k]$. Therefore it suffices to 
  prove that $\rho_{m,k}$ vanishes on the image of $\sigma_{m,k-1}$.
  In other words, we need to check that for each $(k-1)$-tuple $J$, 
  $$\sum_{a\notin J}p_{a,J}^2=0.$$
  The quadratic relations between the Pl\"ucker coordinates on $OG(k,m)$ are given by the Pl\"ucker relations for the ordinary Grassmannian $G(k,m)$, plus the relations induced by the isotropy conditions. 
  Note that a space $U$ with Pl\"ucker coordinates $p_I$'s is spanned 
  by vectors obtained by contraction of the tensor $\sum_Ip_Ie_I$,
  hence of the form $x=\sum_J x_J \sum_a p_{a,J} e_a$ (up to some 
  sign ambiguities that will soon disappear). The basis of $e_a$'s being
$Q_1$-orthonormal, the isotropy of $U$ translates to the condition that
  $$\sum_{a}p_{a,J}p_{a,K}=0$$
  for any  $(k-1)$-tuples $J$ and $K$. In particular for $K=J$ we 
  obtain the expected relations (independently of the sign ambiguities!), and the proof is complete. 
  \qed
  
  \subsection{The hyperelliptic G\"opel variety} 
  
  Denote by $\mathfrak{h}\simeq\AA^m$ the Cartan algebra of  
diagonal matrices, and let  $\fc\subset\fh$ be the Cartan subspace of traceless diagonal matrices. From the preceding section, we can define a 
natural rational map 
\begin{equation}\label{hypgop}
C_{m,k}:\mathfrak{h} \dashrightarrow 
\PP([m-k,k]), 
\end{equation}
by the degree $k$ polynomials $P_T$ from (\ref{P_T}).

  \begin{lemma}\label{basegen}
  The base locus of  $C_{m,k}$  is the union of $\binom{m}{k-1}$ linear spaces of dimension $k$, defined by the condition that at least $m-k+1$ coordinates are equal. \end{lemma}

\begin{proof} Easy and left to the reader. 
\end{proof}

\begin{lemma}
$C_{m,k}$ factors through the projection to $\fc$ from the identity matrix.
\end{lemma}

\begin{proof} Obvious since $P_T(\alpha)$ only depends on the differences $\alpha_i-\alpha_j$.
\end{proof}

Now consider the action of $PGL_2$ on  $\mathfrak{h}$ defined by the projective collineations 
\begin{equation}\label{pgl2action}
\alpha_i \mapsto \frac{a\alpha_i + b}{c\alpha_i +d}, \qquad \mathrm{for}\quad ad-bc\neq 0.
\end{equation}

\begin{lemma}\label{invariant}
$C_{m,k}$ is $PGL(2)$-invariant if and only if $m=2k$.
\end{lemma}

\begin{proof} This immediately follows from the relation: 
$$\frac{a\alpha_i + b}{c\alpha_i +d}-\frac{a\alpha_j + b}{c\alpha_j +d}=
\frac{(ad-bc)(\alpha_i-\alpha_j)}{(c\alpha_i +d)(c\alpha_j +d)}.$$
Indeed, we see that for $m=2k$, each polynomial  $P_T$ is multiplied by the same factor $(ad-bc)^k/\prod_p (c\alpha_p +d)$; while for $m\ne 2k$,   the 
multiplication factor clearly depends on $T$.
\end{proof}

\begin{defi}
The hyperelliptic Göpel variety is the image of the map $C_{2k,k}$ in $\PP([k,k])$.
\end{defi}

\begin{prop}\label{imagit}
The hyperelliptic Göpel variety is the GIT quotient $(\PP^1)^{2k}/\hspace{-1mm}/PGL(2)$, with respect to the democratic polarization.
\end{prop}

\begin{proof}  This follows from the classical fact (Kempe's theorem) that the coordinate ring of $(\PP^1)^{2k}/\hspace{-1mm}/PGL(2)$ is generated by degree one invariants, which are parametrized by $[k,k]$ (see e.g. \cite[section 2]{vakil-equations}).
\end{proof}
 
 As a consequence of \cite[Theorem 1.1]{vakil-equations}, the hyperelliptic Göpel variety $C_{2k,k}$ is cut out by quadrics  for any $k>3$.
 
 \medskip 
 In our interpretation, $C_{2k,k}$ sends a pencil of diagonal quadrics to a Coble type hypersurface, which is a quadratic section of $OG(k,2k)$. But the latter is a union of two spinor varieties, on which  $\cO(2)=L^{\otimes 4}$ if $L$ denotes the spinor line bundle, i.e. the positive bundle generating its Picard group. This means that our Coble hypersurfaces are
 in fact  {\it quartic sections} of the spinor varieties. Let us discuss the first values of $k$ in greater detail.

\subsection{Genus two}
The case $k=3$ is special because of the exceptional isomorphism between $\Spin_6$ and $\SL_4$. 
Recall that this follows for example from the fact that the six-dimensional representation 
$\wedge^2\CC^4$ of $\SL_4$ admits an invariant quadric, which is just
the cone over the Grassmannian $G(2,4)$. This induces an identification
$$S^{\langle 2\rangle}\CC^6\simeq S_{22}\CC^4.$$
Geometrically, this identifies a pencil of quadrics in $\PP^5$ with a fixed smooth member,
with a quadratic line complex in $G(2,4)$. How should we interpret our standard
Cartan subspace in $S^{\langle 2\rangle}\CC^6$? Given 
a basis $e_1,e_2,e_3,e_4$ of $\CC^4$, an orthonormal basis of $\wedge^2\CC^4$ is given by
the six tensors $e_{12}+e_{34}$, $i(e_{12}-e_{34})$, $e_{13}-e_{24}$, $i(e_{13}+e_{24})$,
$e_{14}+e_{23}$, $i(e_{14}-e_{23})$, and 
the space of traceless diagonal quadrics is  
generated by 
$$(e_1\wedge e_2)^2+(e_3\wedge e_4)^2, \quad (e_1\wedge e_3)^2+(e_2\wedge e_4)^2, \quad 
(e_1\wedge e_4)^2+(e_2\wedge e_3)^2,$$
$$(e_1\wedge e_2)(e_3\wedge e_4),  \qquad
(e_1\wedge e_3)(e_2\wedge e_4), \qquad (e_1\wedge e_4)(e_2\wedge e_3), $$
where the last three quadrics obey the usual Pl\"ucker relation. In other words:

\begin{lemma}\label{CartanS22}
The space of quadrics of the form
$$\begin{array}{rcl}
v & = & x\big((e_1\wedge e_2)^2+(e_3\wedge e_4)^2\big)+y\big((e_1\wedge e_3)^2+(e_2\wedge e_4)^2\big)+
z\big((e_1\wedge e_4)^2+(e_2\wedge e_3)^2\big)+\\
& & +p(e_1\wedge e_2)(e_3\wedge e_4)-q(e_1\wedge e_3)(e_2\wedge e_4)
+r(e_1\wedge e_4)(e_2\wedge e_3),
\end{array}$$
with $p+q+r=0$, is a Cartan subspace $\fc$ of $S_{22}\CC^4$.
\end{lemma}

The two half-spin modules of $\Spin_6$ have dimension four, and can be identified with 
the natural representation of $\SL_4$ and its dual. This means that the associated 
degeneracy loci will be hypersurfaces in $\PP^3$ and $\check\PP^3$, of degree two with 
respect to the spinor embedding, which means they are quartic surfaces. Morever they 
have degree three with respect to $v$,  hence they are given by some 
equivariant maps 
$$\Gamma : S^3(S_{22}V_4)\lra S^4V_4\otimes (\det V_4)^2, \qquad 
\Gamma^\vee : S^3(S_{22}V_4)\lra S^4V_4^\vee\otimes (\det V_4)^4.$$
One can check that such equivariant morphisms are unique, 
up to scalar. We could then write them down in tensorial form, and restrict 
them to $\fc$ to get degree three rational maps
$$\gamma : \PP(\fc)\dashrightarrow \PP(S^4V_4), \qquad \gamma^\vee : \PP(\fc)\dashrightarrow \PP(S^4V_4^\vee).$$ 
In this case it is in fact quicker to compute directly the degeneracy loci 
of the two conic bundles defined in section \ref{cobleOG}.  Let $u=\frac{p-q}{2}$, $t=\frac{q-r}{2}$, $w=\frac{r-p}{2}$,
so that again $u+t+w=0$, one obtains $\gamma(v)$ in terms of our coordinates $(a,b,c,d)$  
on $V_4$, as
$$x(y^2+z^2-w^2)(a^2b^2+c^2d^2)+y(x^2+z^2-t^2)(a^2c^2+b^2d^2)
+z(x^2+y^2-u^2)(a^2d^2+b^2c^2)+$$
$$+xyz(a^4+b^4+c^4+d^4)-2(x^2w+y^2t+z^2u+utw)abcd.$$
So we recover classical formulas with five coefficients given by certain cubic forms in five variables. 
For this case the relevant representation of $S_6$ is $[3,3]$, whose 
dimension is indeed $5$. 
We get a rational map $\gamma:\PP^4\dasharrow\PP^4$, 
whose image is the Segre cubic \cite{CAG}.  If we think of six points in $\PP^1$ as defining a genus $2$ curve, the map $\gamma$ was expressed by the classics in terms of the $\theta$-constants of the curve. We will come back to Kummer surfaces later on, in connection with  spinors in
ten dimensions.

\subsection{Genus three}
The case $k=4$ is special because of triality, which exchanges the three
eight-dimensional irreducible representations of $\Spin_8$. Correspondingly, the three
closed orbits in their projectivizations, namely $OG(4,8)_+$, $OG(4,8)_-$ and the 
quadric $\QQ_6$, are projectively isomorphic. In particular, our special hypersurface in this case will be a quartic section of a six-dimensional quadric. 

A nice way to make these isomorphisms concrete is to use the algebra $\OO$ of complexified 
octonions. If $v\in\OO$ is a non-zero isotropic vector, then $R_v=\OO v$ and $L_v=v\OO$ 
are four dimensional isotropic spaces. Moreover $R_v\cap L_v=\CC v$, which shows that 
$R_v$ and $L_v$ belong to different families of maximal isotropic spaces; by convention, we can decide that $R_v$ defines a point of $OG(4,8)_+$, and $L_v$ a point of $OG(4,8)_-$. 
Then the map $v\mapsto R_v$ (resp.  $v\mapsto L_v$) defines an equivariant isomorphism 
between $\QQ^6 $ and $OG(4,8)_+$ (resp.  $OG(4,8)_-$). 

In order to be more specific, we may use the standard basis $e_0=1, e_1,\ldots , e_7$
of the Cayley algebra, whose multiplication table is encoded in the Cayley plane, suitably
oriented. 

\begin{center}
\setlength{\unitlength}{4mm}
\begin{picture}(30,12)(-15,-.7)
\put(-5.5,0){ \resizebox{!}{1.5in}{\includegraphics{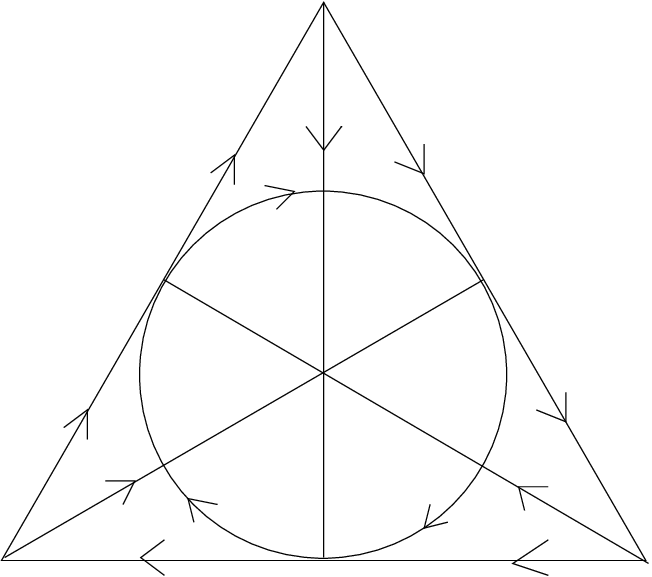}}}
\put(-3.8,4.7){$e_2$}\put(3.35,4.7){$e_6$}
\put(-.3,-.8){$e_4$}\put(-6.1,-.8){$e_1$}
\put(5.5,-.8){$e_5$}\put(-.2,9.9){$e_3$}\put(.9,3.2){$e_7$}
\end{picture} 
\end{center}

\centerline{\scriptsize{The multiplication table of the octonions}}

\medskip
For a generic isotropic vector $v=v_0e_0+v_1e_1+\cdots +v_7e_7$, with $|v|^2=v_0^2+v_1^2
+\cdots +v_7^2=0$, $R_v$ is generated by the four vectors $v, e_1v, e_2v, e_3v$. In particular, its Pl\"ucker coordinates $p_{i_1,\ldots ,i_4}$ are the $70$ maximal minors of the $4\times 8$ matrix
$$\begin{pmatrix}
v_0 & v_1 & v_2 & v_3 & v_4 & v_5 & v_6 & v_7 \\
-v_1 & v_0 & -v_3 & v_2 & v_5 & -v_4 & v_7 & -v_6 \\
-v_2 & v_3 & v_0 & -v_1 & v_6 & -v_7 & -v_4 & v_5 \\
-v_3 & -v_2 & v_1 & v_0 & v_7 & v_6 & -v_5 & -v_4
\end{pmatrix}$$
We know that these Pl\"ucker coordinates can be expressed as degree two polynomials in 
$v_0,\ldots , v_7$. This means that when $|v|^2=0$, we can extract a common quadratic term 
from the degree four maximal minors of the previous matrix. A few typical computations yield
$$\begin{array}{rcl}
p_{0123} & = & (v_0^2+v_1^2+v_2^2+v_3^2)^2, \\
p_{4567} & = & (v_4^2+v_5^2+v_6^2+v_7^2)^2, \\
p_{1236} & = & (v_0^2+v_1^2+v_2^2+v_3^2)(v_1v_5+v_2v_6+v_3v_7-v_0v_4), \\
p_{1246} &= &  (v_4^2+v_5^2+v_6^2+v_7^2)(v_0v_1-v_2v_3)-(v_0^2+v_1^2+v_2^2+v_3^2)(v_4v_5+v_6v_7),\quad \mathrm{etc.}
\end{array}$$
suggesting that we should extract a quadratic factor $v_0^2+v_1^2+v_2^2+v_3^2=-v_4^2-v_5^2-v_6^2-v_7^2$
and renormalize the Pl\"ucker coordinates to 
$$\begin{array}{rcl}
p_{0123} & = & v_0^2+v_1^2+v_2^2+v_3^2, \\
p_{4567} & = & -v_4^2-v_5^2-v_6^2-v_7^2, \\
p_{1236} & = & v_1v_5+v_2v_6+v_3v_7-v_0v_4, \\
p_{1246} &= &  v_2v_3-v_4v_5-v_6v_7-v_0v_1, \quad \mathrm{etc.}
\end{array}$$
Under this normalization our quadratic section of $OG(4,8)_+$ becomes a quartic section of 
$\QQ^6$ that we can write down explicitly.
According to \cite{Macaulay2} the span of the $70$ quartic forms $p_{i_1,\ldots ,i_4}^2$
has dimension $35$. But the span of the G\"opel map in this case turns out to be a vector space $\Theta_{14}$ 
with the following basis:
\begin{align*}
A_1=&v_{0}v_{1}v_{2}v_{3}-v_{4}v_{5}v_{6}v_{7}
& \qquad A_8=&(v_{4}^{2}+v_{5}^{2}+v_{6}^{2}+v_{7}^{2})^2\\
A_2=&v_{0}v_{1}v_{4}v_{5}-v_{2}v_{3}v_{6}v_{7}
& \qquad A_9=&(v_{2}^{2}+v_{3}^{2}+v_{6}^{2}+v_{7}^{2})^2\\
A_3=&v_{0}v_{1}v_{6}v_{7}-v_{2}v_{3}v_{4}v_{5}
& \qquad A_{10}=&(v_{2}^{2}+v_{3}^{2}+v_{4}^{2}+v_{5}^{2})^2\\
A_4=&v_{0}v_{2}v_{4}v_{6}-v_{1}v_{3}v_{5}v_{7}
& \qquad A_{11}=&(v_{1}^{2}+v_{3}^{2}+v_{5}^{2}+v_{7}^{2})^2\\
A_5=&v_{0}v_{2}v_{5}v_{7}-v_{1}v_{3}v_{4}v_{6}
& \qquad A_{12}=&(v_{1}^{2}+v_{3}^{2}+v_{4}^{2}+v_{6}^{2})^2\\
A_6=&v_{0}v_{3}v_{4}v_{7}-v_{1}v_{2}v_{5}v_{6} 
& \qquad A_{13}=&(v_{1}^{2}+v_{2}^{2}+v_{5}^{2}+v_{6}^{2})^2\\
A_7= &v_{0}v_{3}v_{5}v_{6}-v_{1}v_{2}v_{4}v_{7} 
& \qquad A_{14}=&(v_{1}^{2}+v_{2}^{2}+v_{4}^{2}+v_{7}^{2})^2\\
\end{align*}
Note that both septuples are in natural bijection with lines in the Fano plane $\PP^2_{\FF_2}$!

\subsection{Uniqueness}
The locus $M$ in $G(k,2n)$ parametrizing $k$-planes that are isotropic with respect to a generic pencil of quadratic forms can be described as a singular locus as before: choose a non-degenerate form $Q_1$ in the 
pencil $\langle Q_1, Q_2\rangle$, and consider $M$ as a subvariety of the corresponding orthogonal Grassmannian $OG(k,2n)$. The locus of $k$-planes on which $Q_2$ degenerates is a quadric hypersurface 
$H$ in $OG(k,2n)$, and in general $M$ is its $k$-th singular locus, in the sense that the equation 
of $H$ vanishes to order $k$ along $M$. We would like to prove that this property characterizes $H$
uniquely among quadric sections of $OG(k,2n)$, which would follow from the statement that 
$$H^0(OG(k,2n), \mathcal I_M^k(2))=\CC.$$
Note that $M$ is obtained as the zero-locus of a section of $E=S^2U^\vee$ defined by $Q_2$, so the resolution of $\mathcal I_M$ is given by the usual Koszul complex.
The resolution of the powers of this ideal is given by Lemma \ref{resolution} and will allow us to prove:

\begin{prop}\label{unicity}
Uniqueness holds, i.e.
$H^0(OG(k,2n), \mathcal I_M^k(2))=\CC,$ as soon as one of the following conditions holds:
\begin{enumerate}
    \item $\frac{k(k+5)}{2}\le 2n$, 
    \item $k=2,3,4, 5$ and $n\ge k+2$;
    \item $k=2,3$ and $n=k+1$.
\end{enumerate}
\end{prop}

The proof  is postponed to Appendix \ref{uniq}. 
We expect this uniqueness statement to hold in full generality, but we have been unable to control the combinatorics of the Bott-Borel-Weil theorem beyond the conditions just stated. 

\begin{remark}
In fact, for each fixed $1\leq k \leq n$, we have a filtration
$$
M=D_0^k(Q_2) \subset D_1^k(Q_2) \subset \cdots \subset D_{k-1}^k(Q_2)=H\subset D_{k}^k(Q_2)=OG(k,2n),
$$
by the degeneracy loci 
$$
D_i^k(Q_2) \coloneqq \{ [U]\in OG(k,2n)\mid \rank(Q_2|_{U})\leq i \}.
$$
This filtration has the following properties:
\begin{itemize}
    \item $D_{i+1}^k(Q_2)$ is the singular locus of $D_{i}^k(Q_2)$, and $D_{0}^k(Q_2)$ is smooth;
    \item $D_0^n(Q_2)=\emptyset$, while $D_4^n(Q_2)$ can be identified with the quotient of the moduli space $\SU_2(C,\cO_C)$ (of rank $2$ semistable bundles of trivial determinant over a hyperelliptic curve $C$ of genus $n-1$) by the hyperelliptic involution;  $D_3^n(Q_2)$ is the image in this quotient of the fixed locus of the hyperelliptic involution, $D_2^n(Q_2)$ is the Kummer variety of the Jacobian of $C$, and $D_1^n(Q_2)$ is the set of two-torsion points inside the Kummer. Moreover $D_0^{n-1}(Q_2)$ identifies with the Jacobian of $C$ (\cite{DR77};
    \item $D_0^k(Q_2)$ for $1\leq k\leq n-2$ can be identified with moduli spaces of involution invariant Spin bundles on $C$ (\cite{Ram1981}); in particular $D_0^{n-2}(Q_2)$ is isomorphic to the moduli space $\SU_C(2,\cO_C(p))$ for $p\in C$;
    \item $D_i^n(Q_2)$ has also been described, for any $i$, in terms of moduli spaces of (involution invariant) Spin bundles on $C$ (\cite{Bhosle1984}).
\end{itemize}
\end{remark}

\section{Genus three curves and four-forms in eight dimensions}
\label{sec_genus_three}

The geometry of non-hyperelliptic genus three curves is governed by the Lie algebra $\fe_7$ of the exceptional group $E_7$, as shown in  \cite{GSW,rsss,coblequadric}. More precisely, given a general element $h\in \fe_7$ one can construct a general plane quartic (i.e. a general non-hyperelliptic genus three curve) with a level-two structure and a marked flex. Notice that there are 24 flexes on a general plane quartic. There are essentially two ways to recover the curve, that we are going to illustrate. Recall that, by the Chevalley Theorem, if $\fc\subset \fe_7$ is a Cartan subalgebra of $\fe_7$ and $W$ the associated Weyl group, then the GIT quotient $\fe_7/\hspace{-1mm}/E_7$ is isomorphic to the quotient $\fc/W$. 

The group $W$ acts by Cremona transformations on the space $(\PP^2)^7/\hspace{-1mm}/\PGL_3$ of configurations of $7$ points in $\PP^2$. Indeed, this action is generated by consecutive transpositions of the $7$ points and the standard Cremona involution (\cite{do}).
Moreover one can construct a generically finite degree $24$ $W$-equivariant rational map $$\delta:\PP(\fc)\to  (\PP^2)^7/\hspace{-1mm}/\PGL_3 .$$
Let $0\neq h\in \fc$ be a general element, and let $\delta([h])=(p_1,\dots,p_7)$. Then the blow-up of $\PP^2$ at the seven points $p_1,\dots,p_7$ is a del Pezzo surface of degree two. The anticanonical divisor maps this surface onto $\PP^2$ via a $2:1$ map branched along a plane quartic, i.e. a genus three curve $C$. The anticanonical linear system is given by cubics passing through the seven points; among these, there are 24 cuspidal cubics which correspond to the 24 flexes of $C$. The element $[h]$ corresponds to the seven points plus the choice of one cuspidal cubic passing through them. The level-two structure, on the other hand, comes from the action of the Weyl group: $W$ acts transitively on the set of level-two structures of the Jacobian of the curve $C$.

\subsection{Vinberg theory and the Coble quartic from \cite{rsss}}
\label{sec_coble_quartic_rsss}

The second way to recover the curve is via its Jacobian. For this, one needs to decompose $\fe_7$ via the $\ZZ_2$-grading $$ \fe_7=\fg_0\oplus \fg_1=\fsl_8\oplus \wedge^4 V_8,$$
where $V_8$ is an eight dimensional complex vector space. Vinberg's theory (\cite{vinberg}) shows that in this case one can choose a Cartan subalgebra $\fc$  contained inside $\wedge^4V_8\subset \fe_7$; in this case it is also a Cartan subspace of the latter, and the little Weyl group coincides with the genuine Weyl group of $E_7$. Moreover, Chevalley's restriction theorem (in its classical and graded versions) yields the identifications 
$$\fe_7/\hspace{-1mm}/E_7\simeq \wedge^4 V_8/\hspace{-1mm}/\SL_8\simeq  \fc/W.$$ 

Having this in mind, consider on $\PP(V_8)\cong \PP^7$ the homogeneous bundle $\wedge^4 \cQ$, where $\cQ$ is the rank seven quotient bundle on $\PP(V_8)$. By the Bott-Borel-Weil Theorem $H^0(\PP(V_8),\wedge^4 \cQ)=\wedge^4 V_8$. Notice that the fiber of $\wedge^4 \cQ$ is naturally isomorphic to $\wedge^4 \CC^7$; this representation of $\GL_7$ admits a finite number of orbits (it is a parabolic representation). Let us denote by $Y_1$ (respectively $Y_4$) the unique codimension one (resp. four) orbit closure (\cite{KWE7}). Then, for $v\in \wedge^4 V_8\cong H^0(\PP(V_8),\wedge^4 \cQ)$, one can define the following orbital degeneracy loci (ODL):
$$ D_{Y_1}(v):=\{ x\in \PP(V_8)\mid v(x)\in Y_1\subset \wedge^4 \CC^7\cong (\wedge^4 \cQ)_x \} \subset \PP(V_8),$$
$$ D_{Y_4}(v):=\{ x\in \PP(V_8)\mid v(x)\in Y_4\subset \wedge^4 \CC^7\cong (\wedge^4 \cQ)_x \} \subset \PP(V_8).$$
It turns out (\cite{GSW}) that $D_{Y_4}(v)$ is a Kummer threefold, and more precisely the quotient of the Jacobian $J(C)$ of the curve $C$ by $\pm 1$. Moreover, $D_{Y_1}(v)$ is a quartic hypersurface,  isomorphic to the moduli space $\SU_C(2,\cO_C)$ of semistable vector bundles on $C$ of rank $2$ and trivial determinant. The hypersurface $\cC:=D_{Y_1}(v)$ is also characterized by the property that it is the only quartic which is singular along $J(C)/\pm 1$; this was essentially proved by Coble, hence the name of Coble quartic. In representation theoretical terms, the morphism associating to $v$ the Coble hypersurface $\cC$ corresponds to a $\SL_8$-equivariant rational map  $$\Gamma_1: S^7(\wedge^4 V_8)\rightarrow S^4V_8.$$

\subsection{The Heisenberg group and the Cartan subspace} The group of two torsion points $J(C)[2]\cong (\ZZ_2)^6$ acts on the linear system $V_8=H^0(J(C),2\Theta)$, where $\Theta$ is the theta polarization of $J(C)$. The Heisenberg group $H$ is defined as an extension
$$ 1 \to \ZZ_2 \to H \to J(C)[2] \to 1. $$
Let us fix a basis $x_{000},x_{001},x_{010},\dots,x_{111}$ of $V_8^\vee$ given by the Schr\"odinger coordinates. The Heisenberg group is generated by the six operations: 
$$\begin{array}{lll}
x_{i,j,k}\mapsto x_{i+1,j,k},\qquad &  x_{i,j,k}
\mapsto x_{i,j+1,k}, \qquad & x_{i,j,k}\mapsto x_{i,j,k+1}, \\ 
x_{i,j,k}\mapsto (-1)^ix_{i,j,k}, & x_{i,j,k}\mapsto (-1)^jx_{i,j,k}, &  x_{i,j,k}\mapsto (-1)^kx_{i,j,k}.
\end{array}$$

\begin{remark}
The Weyl group $W$ of $E_7$ can be seen as $N_{\GL_8}(H')/H'\times
\{\pm 1\}$, where $N$ stands for normalizer, and $H'$ is the extension of $H$ by multiplication by scalars in $\CC^*$.
\end{remark}

It turns out that the Cartan subspace $\fc\subset \wedge^4 V_8$ is the Heisenberg-invariant subspace. 
Let us change notation slightly, renaming $v_0=x_{000}^\vee$, $v_1=x_{001}^\vee$, $\dots$, $v_7=x_{111}^\vee$. A Killing-orthogonal basis of $\fc$ formed by coroots is given by the following four-vectors (this basis appears in \cite{rsss}, but Cartan subspaces are not mentionned): 
\begin{align*}
h_1=v_{0123}+v_{4567}, & \qquad h_2=v_{0145}+v_{2367},\\h_3=v_{0246}+v_{1357}, &\qquad h_4=v_{0356}+v_{1247},\\h_5=v_{0257}+v_{1346}, &\qquad h_6=v_{0347}+v_{1256},\\h_7=v_{0167}+v_{2345}.
\end{align*}
That this is a Cartan subspace can be proved along the same lines as Lemma \ref{cartan1}: that the $h_i$ commute follows from the definition of the Lie bracket and the fact that for any $v_{abcd}$ in $h_i$ and any $v_{efgh}$ in $h_j$, the quadruples $abcd$ and $efgh$ have two common indices; that they are semisimple 
is a consequence of the fact that they are sums of root vectors for pairs of 
opposite roots.

\begin{remark}
The vectors $v_1,\ldots ,v_7$ are in natural bijection with 
a Fano plane $\PP^2_{\FF_2}$, whose lines are 
the triples $abc$ such that $v_{0abc}$ appears in some $h_i$. So $h_1,\ldots , h_7$ are
also in natural bijection with a Fano plane, which is the dual to the previous one. 
For example $v_1$ corresponds to the line $h_1h_2h_7$ in the dual Fano plane. 
\end{remark}

\begin{center}
\setlength{\unitlength}{4mm}
\begin{picture}(30,12)(-.5,-.7)
\put(0.5,0){ \resizebox{!}{1.5in}{\includegraphics{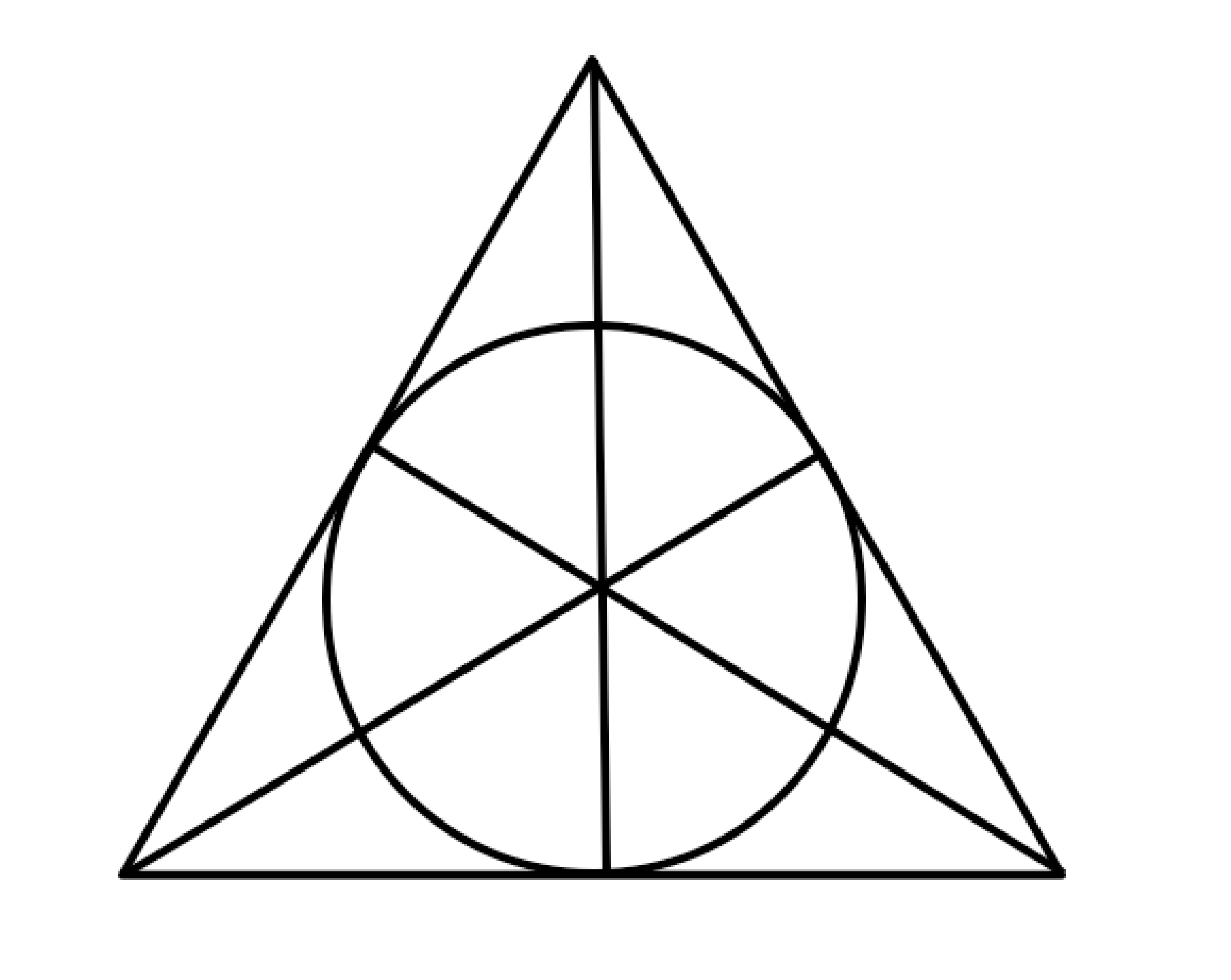}}}
\put(3.5,5){$2$}\put(9.3,5){$6$}
\put(6.5,0){$4$}\put(1.5,0){$1$}
\put(11,0){$5$}\put(6.5,9.2){$3$}\put(5.5,3.5){$7$}

\put(15.5,0){ \resizebox{!}{1.5in}{\includegraphics{fanosansfleches.eps}}}
\put(18.5,5){$2$}\put(24.3,5){$5$}
\put(21.5,0){$4$}\put(16.5,0){$1$}
\put(26,0){$6$}\put(21.5,9.2){$7$}\put(20.5,3.5){$3$}
\end{picture} 
\end{center}

\centerline{\scriptsize{The Fano plane and its dual}}

Consistently with Appendix \ref{sec_heis_group}, we get:
\begin{prop}\label{Heisenberg_E7}
The centralizer $H$ in $\SL_8$ of a Cartan subspace is a finite Heisenberg group,
acting on $V_{8}$ by a Schr\"odinger representation.
\end{prop}

\proof Each $h_i$ can be written in a unique way as the sum of two decomposable tensors, 
so any element in the centralizer must either fix or permute the two corresponding 
four planes. This yields a morphism $h$ from $H$ to $\ZZ_2^7$. 

Let $K\subset H$ denote the kernel. Any $k\in K$ fixes the $14$ four-planes appearing in $h_1,\ldots , h_7$. Since any line $L_i=\CC v_i$ is the intersection of the seven
four planes among these fourteen,  that contains it, it is also stabilized by $k$. 
So $k$ must be diagonal; let us denote its diagonal coefficients by $k_0,\ldots, k_7$
with $k_0\cdots k_7=1$. The condition that $k$ fixes $h_1,\ldots , h_7$ is equivalent to 
$k_0k_ak_bk_c=1$ for any line $abc$ in the Fano plane. Taking the product of these conditions for the three lines intersecting in $a$ yields $(k_0k_a)^2=1$, hence $e_a:=k_0k_a=\pm 1$. Plugging into the previous condition gives $k_0^2=e_ae_be_c$ for any line $abc$, in particular $k_0$ must be a fourth root of unity.
Suppose first that $k_0=1$. If we identify points in the Fano plane with nonzero vectors in $\ZZ_2^3$, the condition can be rewritten as $e_{a+b}=e_ae_b$, meaning that $e$
must belong to $\Hom(\ZZ_2^3,\ZZ_2)$. We conclude that $K\simeq \ZZ_4\times \Hom(\ZZ_2^3,\ZZ_2)$.

Now consider some $x$ in the image $I$ of $h$, that we identify with a sequence $x_1,\ldots ,x_7$ with $x_i=-1$ if and only if $x$ exchanges the two four-planes associated to $h_i$.
Consider for example $h_1, h_2, h_7$, which correspond to the three lines in the Fano plane that contain $1$. If $x_1=1$, $v_1$ must be sent to $\langle v_0,v_1,v_2,v_3\rangle$; if also $x_2=1$, $v_1$ must be sent to $\langle v_0,v_1,v_4,v_5\rangle$, hence in fact to  $\langle v_0,v_1\rangle$, and this implies that $x_7=1$. Similarly, if $x_1=-1$, $v_1$ must be sent to $\langle v_4,v_5,v_6,v_7\rangle$;
if also $x_2=-1$,  $v_1$ must be sent to $\langle v_2,v_3,v_6,v_7\rangle$, hence in fact to  $\langle v_6,v_7\rangle$, and this implies that $x_7=1$. We conclude that $x_ax_bx_c=1$ for any line $abc$ in the {\it dual} Fano plane. Since there is a canonical identification of $\ZZ_2^m$ with $\Hom((\ZZ_2^m)^\vee,\ZZ_2)$, we conclude that $I\simeq \ZZ_2^3$. Finally, $I$ can be lifted to a permutation group in $H\subset \SL_8$. 
Consider for example $v_1$; it corresponds to the line $h_1h_2h_7$ in the dual Fano plane, hence to the element $x$ in $I$ defined by $x_1=x_2=x_7=1$ and $x_3=x_4=x_5=x_6=-1$. 
This implies that the planes  $\langle v_0,v_1\rangle$,  $\langle v_2,v_3\rangle$,  $\langle v_4,v_5\rangle$,  $\langle v_6,v_7\rangle$ must be fixed, and it is easy 
to check that the product of transpositions $(01)(23)(45)(67)$ is a suitable lift. 
Hence the claim.\qed

\medskip
The restriction of $\Gamma_1$ to the Cartan subspace yields a rational map 
$$ \gamma_1:\PP(\fc)\to \PP(\Theta_{15})\subset \PP(S^4V_8).$$
 It was computed explicitly in \cite{rsss} (see also \cite[IX.7]{do} and \cite{freitag-sm}), who verified that the span
$\Theta_{15}$ of its image is the space of Heisenberg invariant degree four polynomials,  of dimension fifteen. 
The closure of the image of $\gamma_1$ is the so-called G\"opel variety $\cG$. 
It has degree $175$ and is cut-out by $35$ cubics and $35$ quartics.
Moreover, it is a compactification of $(\PP^2)^7/\hspace{-1mm}/\PGL_3$. Indeed $\gamma_1$, which is generically $24:1$, factorizes through $\delta$ and a birational map $\eta:(\PP^2)^7/\hspace{-1mm}/\PGL_3\to \cG$. The fact that $\gamma_1$ is $24:1$ translates into the fact that the Coble hypersurface $\cC$ only depends on the curve $C$ and not on the choice of the flex point, which is an avatar of Coble's uniqueness result together with the Torelli theorem for $J(C)$. We record:

\begin{coro} 
The non-hyperelliptic G\"opel variety $\cG$ in $\PP^{14}=\PP(\Theta_{15})$ is birational to the moduli space of $7$ points on $\PP^2$. 
\end{coro} 

\begin{remark}
The Weyl group $W$ acts on $\fc$, $(\PP^2)^7/\hspace{-1mm}/\PGL_3$ and $\Theta_{15}$ (because it  normalizes the Heisenberg group), and all maps $\gamma_1$, $\delta$ and $\eta$ are $W$-equivariant. The morphism $\gamma_1$  realizes $\Theta_{15}\subset S^7[\fc^\vee]$ as a Macdonald representation for the action of $W$. As such a representation, $\Theta_{15}$ is linearly generated by the $W$-orbit of the product of seven orthogonal roots in $\fc$ (here we are identifying $\fc\cong \fc^\vee$ via the Killing form).
\end{remark}

We will provide in the next subsection a similar description  for the Coble quadric constructed in \cite{coblequadric}.

\subsection{The Coble quadric and its equations}
\label{sec_coble_quadric_genus3}

Consider on $G(2,V_8)$ the homogeneous bundle $\wedge^4 \cQ$, where $\cQ$ is the rank six quotient bundle on $G(2,V_8)$. By the Borel-Weil Theorem $H^0(G(2,V_8),\wedge^4 \cQ)=\wedge^4 V_8$. Notice that the fiber of $\wedge^4 \cQ$ is naturally isomorphic to $\wedge^4 \CC^6\cong \wedge^2 \CC^6$, i.e. the space of skew-symmetric $6\times 6$-matrices. Let us denote by $Z_1$ (respectively $Z_6$) the orbit closure of matrices with rank at most four (resp. two). Then, for $h\in \wedge^4 V_8$, one can define the following orbital degeneracy loci (ODL):
$$ D_{Z_1}(h):=\{ x\in G(2,V_8)\mid h(x)\in Z_1\subset  (\wedge^4 \cQ)_x \} \subset G(2,V_8),$$
$$ D_{Z_6}(h):=\{ x\in G(2,V_8)\mid h(x)\in Z_6\subset  (\wedge^4 \cQ)_x \} \subset G(2,V_8).$$
It was shown in \cite{coblequadric} that $D_{Z_6}(h)$ is isomorphic to the moduli space $\SU_C(2,\cO_C(p))$ of semistable vector bundles on $C$ of rank $2$ and determinant equal to $\cO_C(p)$, for some $p\in C$. Moreover, $Q:=D_{Z_1}(h)$ is the only quadric hypersurface which is singular along $\SU_C(2,\cO_C(p))$, which is why $Q$ was coined the Coble quadric. 

In representation theoretical terms, the morphism associating to $h$ the Coble hypersurface $Q$ corresponds to a $\SL_8$-equivariant rational map $$\Gamma_2: S^3(\wedge^4 V_8)\rightarrow S_{22}V_8=H^0(G(2,V_8),
\cO(2)).$$ 
Using Schur's lemma, it is easy to check that there exists a unique such map, up to nonzero scalar. 

\medskip
By restricting  $\Gamma_2$ to the Cartan subspace, and  denoting  by $\Theta_{21}\subset S_{22}V_8$ the linear span of the image (which will indeed turn out to be $21$-dimensional), we get
$$ \gamma_2:\PP(\fc) \to \PP(\Theta_{21})\subset \PP(S_{22}V_8). $$
By analogy with the case of Coble quartics, we will 
call {\it G\"opel variety} the closure $\cG_2$ of the image of $\gamma_2$. The aim of the rest of this section is to describe $\gamma_2$ and $\cG_2$. 
Since $\Gamma_2$ is $\SL_8$-equivariant, $\gamma_2$ is $W$-equivariant, and as we will see, also $H$-invariant. Therefore $\Theta_{21}$ is a representation of the group $W$, and we will show that, similarly to $\Theta_{15}$, it 
is irreducible, and another one of the so-called Macdonald representations.

\smallskip
Denote by $c_1,\dots,c_7$ the dual basis of $h_1,\dots,h_7$. 
Let us fix $h=\sum_i c_i h_i \in \fc$. Let us denote by $x_{ij}=x_i \wedge x_j$ the coordinates of $\wedge^2 \fh$. Since $S_{2,2}V_8^\vee \subset S^2(\wedge^2 V_8^\vee)$, we will first interpret $\gamma_2(h)$ as a polynomial $\widetilde{\gamma_2}(h)\in S^2(\wedge^2 V_8^\vee)$. 
Then we will deduce $\gamma_2(h)$ via the projection 
from $S^2(\wedge^2 V_8^\vee)$ to $S_{22}V_8^\vee$.

Let us define the following elements of $S^2(\wedge^2 V_8^\vee)$:
$$\begin{array}{rcl} Q_{ij|kl|mn|op} & := & x_{ij}^2+x_{kl}^2+x_{mn}^2+x_{op}^2, \\
R_{ij|kl|mn|op} & := & x_{ij}x_{kl}+x_{mn}x_{op}, 
\\
S_{ijkl}& := & x_{ij}x_{kl}
-x_{ik}x_{jl}+x_{il}x_{jk}.
\end{array}$$ 
We will also denote $c_{ijk}=c_ic_jc_k$.

\begin{lemma}
The polynomial $\widetilde{\gamma_2}(h)$ is the following:
$$c_{127}Q_{01|23|45|67}
+c_{1}(c_{2}^2+c_{7}^2)R_{01|23|45|67}+c_{2}(c_{1}^2+c_{7}^2)R_{01|45|23|67}+c_{7}(c_{1}^2+c_{2}^2)
R_{01|67|23|45}
$$
$$-c_{135}Q_{02|13|46|57}-c_{1}(c_{3}^2+c_{5}^2)R_{02|13|46|57}+c_{3}(c_{1}^2+c_{5}^2)R_{02|46|13|57}+c_{5}(c_{1}^2+c_{3}^2)R_{02|57|13|46}$$
$$+c_{146}Q_{03|12|47|56}+c_{1}(c_{4}^2+c_{6}^2)R_{03|12|47|56}+c_{4}(c_{1}^2+c_{6}^2)R_{03|56|12|47}+c_{6}(c_{1}^2+c_{4}^2)R_{03|47|12|56}$$
$$
-c_{236}Q_{04|15|26|37}-c_{2}(c_{3}^2+c_{6}^2)R_{04|15|26|37}-c_{3}(c_{2}^2+c_{6}^2)R_{04|26|15|37}-c_{6}(c_{2}^2+c_{3}^2)R_{04|37|15|26}$$
$$
+c_{245}Q_{05|14|36|27}+c_{2}(c_{4}^2+c_{5}^2)R_{05|14|27|36}-c_{4}(c_{2}^2+c_{5}^2)R_{05|36|14|27}-c_{5}(c_{2}^2+c_{4}^2)R_{05|27|14|36}$$
$$-c_{347}Q_{06|24|35|17}+c_{3}(c_{4}^2+c_{7}^2)R_{06|24|17|35}+c_{4}(c_{3}^2+c_{7}^2)R_{06|35|17|24}-c_{7}(c_{3}^2+c_{4}^2)R_{06|17|24|36}$$
$$
+c_{567}Q_{07|25|34|16}+c_{5}(c_{6}^2+c_{7}^2)R_{07|25|16|34}+c_{6}(c_{5}^2+c_{7}^2)R_{07|34|16|25}+c_{7}(c_{5}^2+c_{6}^2)R_{07|16|25|34} $$
$$
+c_1^3(S_{0123}+S_{4567})+c_2^3(S_{0145}+S_{2367})+c_3^3(S_{0246}+S_{1357})+c_4^3(S_{0356}+S_{1247})$$
$$+c_5^3(S_{0257}+S_{1346})+c_6^3(S_{0347}+S_{1256})+c_7^3(S_{0167}+S_{2345}).
$$
\end{lemma}

\begin{proof}
We want to construct a nonzero degree three equivariant morphism
from $\wedge^4 V_8$ to $S_{2,2}V_8$, and we know that this morphism is unique up to scalar. From the inclusion $S_{2,2}V_8 \subset S^2(\wedge^2 V_8)$ we interpret this as a morphism $S^3(\wedge^4 V_8) \otimes S^2(\wedge^2 V_8^\vee) \to \det(V_8) \cong \CC $. We  construct this morphism as the appropriate symmetrization of the morphism $\phi:(\wedge^4V_8)^{\otimes 3}\otimes (\wedge^2 V_8^\vee)^{\otimes 2} \to \det(V_8)$ defined as follows:. for $w_1,w_2,w_3 \in \wedge^4 V_8$ and $z_1,z_2 \in \wedge^2 V_8^\vee$, let $$\phi((w_1 \otimes w_2 \otimes w_3)\otimes(z_1 \otimes z_2))= w_3 \wedge (z_1 \lrcorner w_1) \wedge (z_2 \lrcorner w_2) \in \det(V_8),$$
where $\lrcorner$ is the usual contraction map. The symmetrization of $\phi$ is $\SL_8$-equivariant and 
nonzero. So it must coincide with the unique (up to scalar) morphism $S^3(\wedge^4 V_8) \otimes S^2(\wedge^2 V_8^\vee) \to \det(V_8)$.

Let us apply the above to some specific cases. So, for instance, if $w_1=w_2=w_3=h_1$, $z_1=x_{01}$, $z_2=x_{34}$, we get $\phi((w_1 \otimes w_2 \otimes w_3)\otimes(z_1 \otimes z_2))=v_{45672301}=v_{01234567}$, which accounts for the term $c_1^3x_{01}x_{23}$ in $\widetilde{\gamma_2}(h)$. If $w_1=w_2=h_1$, $w_3=h_2$, $z_1=x_{01}$, $z_2=x_{45}$, we get $\phi((w_1 \otimes w_2 \otimes w_3)\otimes(z_1 \otimes z_2))=v_{01234567}$, which accounts for the term $c_2c_1^2x_{01}x_{45}$ in $\widetilde{\gamma_2}(h)$. If $w_1=h_1$, $w_2=h_3$, $w_3=h_5$, $z_1=x_{02}$, $z_2=x_{02}$, we get $\phi((w_1 \otimes w_2 \otimes w_3)\otimes(z_1 \otimes z_2))=-v_{01234567}$, which accounts for the term $-c_{135}x_{02}^2$ in $\widetilde{\gamma_2}(h)$. Modulo permutation of the indices, these are the only non-zero terms that one obtains, and one deduces the expression for $\widetilde{\gamma_2}(h)$.
\end{proof}

Finally, we can explicitly describe the morphism $\gamma_2$. Let us denote  $$\pi_{i|jk|lm}(c):=c_i(c_j^2+c_k^2-c_l^2-c_m^2).$$

\begin{prop}
\label{prop_formula_quadric}
The equation $\gamma_2(h)$ of the Coble quadric  %$:\PP(\fh) \dasharrow \PP(\Delta_1)$ 
is given by
%h=\sum_i c_i h_i\mapsto \gamma_1(h)=
$$c_{127}Q_{01|23|45|67} +\pi_{1|27|35}(c)R_{01|23|45|67}
+\pi_{2|17|36}(c)R_{01|45|23|67}+\pi_{7|12|34}(c)R_{01|67|23|45}$$
$$-c_{135}Q_{02|13|46|57}(c)+\pi_{3|15|26}(c)R_{02|46|13|57}
+\pi_{5|13|24}(c)R_{02|57|13|47}$$
$$+c_{146}Q_{03|12|47|56}+\pi_{1|46|35}(c)R_{03|12|47|56}
+\pi_{4|16|25}(c)R_{03|56|12|47}+\pi_{6|14|23}(c)R_{03|47|12|56}$$
$$-c_{236}Q_{04|15|26|37}$$
$$+c_{245}Q_{05|14|36|27}+\pi_{2|45|36}(c)R_{05|14|27|36}$$
$$-c_{347}Q_{06|24|35|17}+\pi_{3|47|26}(c)R_{06|24|17|35}
+\pi_{4|37|25}(c)R_{06|35|17|24}$$
$$+c_{567}Q_{07|25|34|16}+\pi_{5|67|24}(c)R_{07|25|16|34}+\pi_{6|57|23}(c)R_{07|34|16|25}+\pi_{7|56|34}(c)R_{07|16|25|34}.$$
\end{prop}

\begin{remark}
In the expression above, we have written $\gamma_2(h)$ in terms of a basis. For a more symmetric expression (but with redundant terms, see the proof below) just eliminate from $\tilde{\gamma_2}(h)$ the terms where the $S_{ijkl}$'s appear. 
\end{remark}

\begin{proof}
The ideal of $G(2,V_8)$ is $\wedge^4V_8=\ker(S^2(\wedge^2 V_8) \to S_{2,2}V_8)$ and it is generated by the Pl\"ucker quadrics, which are of the form $ S_{ijkl}$. 
%$$ x_{ij}x_{kl}+x_{mn}x_{op}-x_{ik}x_{jl}-x_{mo}x_{np}+x_{il}x_{jk}+x_{mp}x_{no}. $
This implies the relations 
$$R_{ij|kl|mn|op}-R_{ik|jl|mo|np}+R_{il|jk|mp|no}=0.$$
Using these relations, we readily deduce $\tilde{\gamma_2}$ from $\gamma_2$. 
\end{proof}

\begin{prop}
$\Theta_{21}$ has the following properties:
\begin{itemize} 
\item it is generated by the $28$ polynomials $Q_{ij|kl|mn|op}$ and $R_{ij|kl|mn|op}$, 
\item its dimension is $21$, 
\item it is an $H$-invariant subspace of  $S_{22}V_8$,
\item it is a Macdonald irreducible representation of  $W$.
\end{itemize}
\end{prop}

\begin{proof}
One checks by hand that the polynomials that generate $\Theta_{21}$ are $H$-invariant. Then the only non-trivial claim is the fact that it is a Macdonald representation. Clearly it is a representation of $W$ inside $S^3 \fc^\vee$, and it contains the monomial $c_{127}$, which is a product of orthogonal roots. Moreover, one checks that the Macdonald representation generated by $c_{127}$ is $21$-dimensional (with \cite{Macaulay2} for instance, see attached file {\bf gopel\_e7}), so it must coincide with $\Theta_{21}$. 
\end{proof}

There are only two $21$-dimensional irreducible representations of $W$ (\cite{Car}), that we denote $V_1$ and $V_2$. One of them, say $V_1$, is simply $\wedge^2 \fc$.  

\begin{lemma}
$\Theta_{21}$ is isomorphic to $V_2$ as a $W$-representation.
\end{lemma}

\begin{proof}
It is sufficient to show that $\Theta_{21} \neq V_1$ as a $W$-representation. Let $s_i \in W$ be the reflection corresponding to $h_i$ for $i=1,\dots,7$, and let $w:= s_1s_2s_7 \in W$. Thus $w$ sends $h_i$ to $-h_i$ if $i=1,2,7$ and to $h_i$ if $i=3,4,5,6$. From the explicit basis of $\Theta_{21}$ provided above one computes that the trace of the action of $w$ on $\Theta_{21}$ is $5$, while a direct check shows that the trace of the action of $w$ on $\wedge^2 \fh$ is $-3$; therefore $\Theta_{21} \neq V_1$.\end{proof}

\begin{theorem}
The rational map $\gamma_2:\PP(\fc) \dasharrow \cG_2 \subset \PP(\Theta_{21})\cong \PP^{20}$ is  birational. Its base locus is given by the $W$-orbit of $[h_1] \in \PP(\fc)$, i.e. by the classes of the $63$ roots in $\fc$. The image has $336$ singular points, which are the $\gamma_2$-images of $336$ lines corresponding to $A_2$-subsystems of $\fe_7$. Explicitly, these lines are the $W$-orbit of the line $\ell_0$ with equations $c_{1}=c_{2}+c_{3}+c_{4}=c_{3}-c_{5}=c_{2
       }+c_{3}+c_{6}=c_{2}-c_{7}=0$.
\end{theorem}

\begin{proof}
Proposition \ref{prop_formula_quadric} gives an explicit expression of $\gamma_2$ in terms of coordinates. One can plug this in \cite{Macaulay2} (see attached file {\bf gopel\_e7}) and compute the degree of this rational map, which turns out to be equal to one, and its base locus. 

Next, again using \cite{Macaulay2}, we analyze the subscheme $Z$ of $\PP(\fc)$ where the Jacobian of $\gamma_2$ fails to have maximal rank. It turns out that $Z$ contains $\ell_0$, is of dimension 
$1$ and has degree $336$. Since the subscheme $Z$ is $W$-invariant, it must contain the $W$-orbit of $\ell_0$, which consists of $336$ lines, so this orbit must agree with the top-dimensional part of $Z$. 

On the other hand, $\ell_0$ is orthogonal to the roots 
$\beta_1=h_2-h_4-h_6+h_7$, $\beta_2=h_3-h_4+h_5-h_6$,  and $\beta_3=\beta_1-\beta_2$. Since $\beta_1,\beta_2,\beta_3$ form an $A_2$-subsystem of $\fe_7$ and $A_2$ subsystems form a $W$-orbit of $336$ elements (see \cite{oshima}), the top-dimensional part of $Z$ is precisely given by $A_2$-subsystems of $\fe_7$.
Also, \cite{Macaulay2} allows to check that $\ell_0$ is contracted to a point by $\gamma_2$, so the same holds for each line of the $W$-orbit of $\ell_0$. One verifies that the resulting points are pairwise distinct. These are singular points of the image of $\gamma_2$, 
since the exceptional locus of $\gamma_2$ is 1-dimensional, by the above analysis of $Z$.
\end{proof}

The space $\PP(\fc)$ parametrizes curves of genus $3$ plus a fixed flex point, and the birationality above translates into the fact that Coble quadrics $Q$, unlike Coble quartics $\cC$, depend on the choice of the flex of the curve $C$. This is not unexpected since we know that $Q_h:=D_{Z_1}(h)$ is isomorphic to $\SU_C(2,\cO_C(p))$ for a fixed point $p\in C$; we believe that $(C,p)$ is exactly the genus $3$ curve and the flex point corresponding to $[h]\in \PP(\fc)$. Let us fix now a curve $C$ with its $24$ flexes $p_1,\dots,p_{24}$. Let $l_1,\dots,l_{24}$ be the corresponding points in $\PP(\fc)$.

\begin{coro}
$Q_{l_1},\dots,Q_{l_{24}}$ are not $\PGL_8$-equivalent inside $G(2,V_8)$.
\end{coro}

\begin{proof}
If $Q_{l_i}$ and $Q_{l_j}$ are $\PGL_8$-conjugate then by the birationality of $\gamma_1$ we deduce that $l_i$ and $l_j$ are $\PGL_8$-conjugate inside $\PP(\fc) \subset \PP(\fe_7)$. However, by the Chevalley Theorem, this is equivalent to the fact that $l_i$ and $l_j$ belong to the same $W$-orbit, which in turn means that $p_i=p_j$. 
\end{proof}

\section{Genus two curves and three-forms in nine dimensions}
\label{sec_genustwo_GS}

This case is induced by the  $\ZZ_3$-grading 
$$\fe_8=\fsl_9\oplus \wedge^3\CC^9\oplus \wedge^6\CC^9,$$
studied in \cite{ev}. 
Its relationship with Coble cubics was first observed in \cite{GSW}, 
and then considered in greater detail in \cite{gs, bmt}. 

\smallskip
A Cartan subspace $\fc$ of $\wedge^3V_9$ is given in \cite{ev}:
$$\begin{array}{rcl}
h_1 & = & e_1\wedge e_2\wedge e_3+e_4\wedge e_5\wedge e_6+e_7\wedge e_8\wedge e_9, \\
h_2 & = & e_1\wedge e_4\wedge e_7+e_2\wedge e_5\wedge e_8+e_3\wedge e_6\wedge e_9, \\
h_3 & = & e_1\wedge e_5\wedge e_9+e_2\wedge e_6\wedge e_7+e_3\wedge e_4\wedge e_8, \\
h_4 & = & e_1\wedge e_6\wedge e_8+e_2\wedge e_4\wedge e_9+e_3\wedge e_5\wedge e_7.
\end{array}$$
Observe that if we denote $e_1=a_1, e_2=b_1, e_3=c_1, e_4=b_2, e_5=c_2, e_6=a_2, e_7=c_3, e_8=a_3, e_9=b_3,$ this gives 
$$\begin{array}{rcl}
h_1 & = & a_1\wedge b_1\wedge c_1+a_2\wedge b_2\wedge c_2+a_3\wedge b_3\wedge c_3, \\
h_2 & = & a_1\wedge b_2\wedge c_3+a_2\wedge b_3\wedge c_1+a_3\wedge b_1\wedge c_2, \\
h_3 & = & a_1\wedge b_3\wedge c_2+a_2\wedge b_1\wedge c_3+a_3\wedge b_2\wedge c_1, \\
h_4 & = & a_1\wedge a_2\wedge a_3+b_1\wedge b_2\wedge b_3+c_1\wedge c_2\wedge c_3,
\end{array}$$
suggesting a strong relationship with the case of Rubik's cubes. 

The corresponding complex reflection group is $G_{32}$ in the Shephard-Todd classification; it is isomorphic with $\ZZ_3\times Sp_4(\FF_3)$. 

\medskip
The equation of the Coble cubic in terms of coordinates on the Cartan space was computed in 
\cite[(3.9)]{gs}; its coefficients are particularly simple quartics, defining a rational map 
$$\gamma_1:\PP(\fc)\dasharrow \PP^4=\PP(\Theta_5).$$  
Its image is the Burkhardt quartic, which compactifies the  Siegel modular variety $A_2(3)$ parametrizing abelian surfaces with a full level three structure; it is embedded inside the projectivization of a  Macdonald representation $\Theta_5$ of $G_{32}$. Moreover
the base locus of $\gamma_1$ consists in $40$ points and form a single orbit of $G_{32}$ \cite[Remark 3.24]{gs}. Each of these points is the intersection of three
of the $40$ reflection hyperplanes, and appear among the "flats" listed in
\cite[Table 1]{gs}. 

The degree of the map is six, and the meaning of the map $\gamma_1$ is the following: $\PP(\fc)/G_{32}$  parametrizes genus two curves with a chosen Weierstrass point (see \cite[Theorem 3.4]{gs}). The morphism $\gamma_1$, which associates to an alternating form a Coble cubic, is the forgetful map that forgets the Weierstrass point. Thus, the Burkhardt quartic can be seen as a \emph{G\"opel} variety for genus two curves.

\medskip
As shown in \cite{BBFM2}, one can also define a Coble type quadric in $G(3,V_9)$. In the sequel we give the equation of this quadric and we study the corresponding G\"opel map. This is the restriction (and the projectivization) to $\PP(\fc)$ of the unique $\SL_9$-equivariant map 
$$\Gamma_2 : S^4(\wedge^3V_9)\lra H^0(G(3,V_9),\cO(2))\simeq S_{222}V_9^\vee \simeq S_{222222}V_9.$$
Let us denote this restriction, which is a $G_{32}$-equivariant map, by:
$$
\gamma_2:\PP(\fc)\dashrightarrow \cG_2 \subset\PP(\Theta_{30}), 
$$
where $\cG_2$ is the image of $\gamma_2$ and $\Theta_{30} \subset  S_{222222}V_9$ is the $G_{32}$-invariant linear span of $\cG_2$
(which will indeed turn out to be $30$-dimensional). 

\begin{theorem}\label{wedge3c9}
The space $\Theta_{30}$ is the $30$-dimensional $G_{32}$-subrepresentation of $S^4 \fc^\vee$ generated by the orbit of $c_1^4$. The map $\gamma_2$ is a birational immersive morphism, obtained by composing the $4$th Veronese embedding with the projection from $$\Theta_5= \langle c_1c_2c_3c_4, c_1(c_2^3 +c_3^3+ c_4^3), c_2(c_1^3 +c_3^3 -c_4^3), c_3(c_1^3 -c_2^3 +c_4^3), c_4(c_1^3+ c_2^3 -c_3^3) \rangle \simeq \PP^4.$$ The ideal of $\cG_2$ is generated by $300$ quadrics.
\end{theorem}

\begin{remark}
The relationship with Rubik's cubes is confirmed by the fact that if we let $c_4=0$ in the previous Theorem, we recover the picture of Proposition \ref{toy_gopel}.
\end{remark}

\begin{remark}
The equation of the Coble quadric, depending on variables $c_1,c_2,c_3,c_4$, can be deduced from the quadrics appearing in the following proof and by permutation of the indices.
\end{remark}

\begin{proof}
$\Gamma_2$ can be computed by using the morphism $\phi : S^2(\wedge^3V_9)\lra V_9\otimes 
\wedge^5V_9$ defined by sending 
$(a_1\wedge a_2\wedge a_3)(b_1\wedge b_2\wedge b_3)$ to 
$$a_1 \otimes a_2\wedge a_3\wedge b_1\wedge b_2\wedge b_3+
a_2 \otimes a_3\wedge a_1\wedge b_1\wedge b_2\wedge b_3
+a_3 \otimes a_1\wedge a_2\wedge b_1\wedge b_2\wedge b_3+$$
$$+b_1 \otimes b_2\wedge b_3\wedge a_1\wedge a_2\wedge a_3+
b_2 \otimes b_3\wedge b_1\wedge a_1\wedge a_2\wedge a_3
+b_3 \otimes b_1\wedge b_2\wedge a_1\wedge a_2\wedge a_3.$$

Then we can compose 
$$S^4(\wedge^3V_9)\hookrightarrow S^2(\wedge^3V_9)\otimes S^2(\wedge^3V_9)
\stackrel{\phi\otimes\phi}{\lra} (V_9\otimes 
\wedge^5V_9)\otimes (V_9\otimes 
\wedge^5V_9) \stackrel{(13)}{\lra}\hspace*{3cm} $$
$$\hspace*{3cm} \stackrel{(13)}{\lra}(V_9\otimes 
\wedge^5V_9)\otimes (V_9\otimes 
\wedge^5V_9) \lra \wedge^6V_9\otimes \wedge^6V_9\lra S_{222222}V_9,$$
where the morphism $(13)$ just permutes the first and third terms $V_9$.  
For instance, one gets 
$$\begin{array}{rcl}
\frac{1}{2}\phi(h_1^2) & =& 
\;\;\; e_1\otimes (e_{23456}+e_{23789})+e_2\otimes (e_{31456}+e_{31789})+e_3\otimes (e_{12456}+e_{12789})+ \\
 & & +e_4\otimes (e_{56123}+e_{56789})+e_5\otimes (e_{64123}+e_{64789})+e_6\otimes (e_{45123}+e_{45789})+ \\
 & & +e_7\otimes (e_{89123}+e_{89456})+e_8\otimes (e_{97123}+e_{97456})+e_9\otimes (e_{78123}+e_{78456}),\\
 \frac{1}{2}\phi(h_2^2) & =& 
\;\;\; e_1\otimes (e_{47258}+e_{47369})+e_2\otimes (e_{58147}+e_{58369})+e_3\otimes (e_{14769}+e_{25869})+ \\
 & & +e_4\otimes (e_{71258}+e_{71369})+e_5\otimes (e_{82147}+e_{82369})+e_6\otimes (e_{93147}+e_{93258})+ \\
 & & +e_7\otimes (e_{14258}+e_{14369})+e_8\otimes (e_{25147}+e_{25369})+e_9\otimes (e_{36147}+e_{36258}),\\
\phi(h_1h_2) & = &  \;\;\;  e_1\otimes e_{12347}+e_2\otimes e_{12358}+e_3\otimes e_{12369}+e_4\otimes e_{14567}+e_5\otimes e_{24568}+\\
 & & +e_6\otimes e_{34569}+e_7\otimes e_{14789}+e_8\otimes e_{25789}+e_9\otimes e_{36789}, \\
\phi(h_2h_3) & = & \;\;\;  e_1\otimes e_{14759}+e_2\otimes e_{25678}+e_3\otimes e_{34869}+e_4\otimes e_{13478}+e_5\otimes e_{12589}+\\
 & & +e_6\otimes e_{23769}+e_7\otimes e_{14267}+e_8\otimes e_{23458}+e_9\otimes e_{15369}.
\end{array}$$
Using this formula we can determine the image by $\Gamma_2$ of $c=c_1h_1+c_2h_2+c_3h_3+c_4h_4$. 
A direct computation shows that the coefficient of $c_1c_2c_3c_4$ vanishes; since the $G_{32}$-representation spanned by the orbit of $c_1c_2c_3c_4$ is generated by $c_1c_2c_3c_4, c_1(c_2^3 +c_3^3+ c_4^3), c_2(c_1^3 +c_3^3 -c_4^3), c_3(c_1^3 -c_2^3 +c_4^3), c_4(c_1^3+ c_2^3 -c_3^3)$ (verification made with \cite{Macaulay2}), we deduce that the coefficient of each of these elements also vanishes. 

Denote by $\gamma_{ijkl}$ the coefficient of $c_ic_jc_kc_l$. 
Starting from the above computations, we just need to compose all morphisms to get the explicit expression of $S^4(\wedge^3 V_9)\to S_{222222}V_9$. We compute  directly  that
$$\begin{array}{rcl}
\gamma_{1111} & =& 6(f_{789}f_{789}+f_{123}f_{123}+f_{456}f_{456})-6(f_{789}f_{456}+f_{789}f_{123}+f_{456}f_{123})+\\
 & & +2(-f_{237}f_{189}+f_{238}f_{179}-f_{239}f_{178}+f_{137}f_{289}-
 f_{138}f_{279}+f_{139}f_{278}-f_{127}f_{389}+\\
 & & +f_{128}f_{379}-f_{129}f_{378}-f_{234}f_{156}+f_{235}f_{146}-
 f_{236}f_{145}+f_{134}f_{256}-f_{135}f_{246}+\\
& & 
+f_{136}f_{245}-f_{124}f_{356}+f_{125}f_{346}-f_{126}f_{345}
-f_{567}f_{489}+f_{568}f_{479}-f_{569}f_{478}+ \\
& & +f_{467}f_{589}-f_{468}f_{579}+f_{469}f_{578}-f_{457}f_{689}+f_{458}f_{679}-f_{459}f_{678}),
\end{array}$$
$$\begin{array}{rcl}
\gamma_{1112}&=& 36(-f_{179}f_{346}+f_{178}f_{245}+f_{146}f_{379}-f_{145}f_{278}
+f_{479}f_{136}-f_{478}f_{125}+ \\
 & & +f_{289}f_{356}-f_{256}f_{389}-f_{134}f_{679} +f_{124}f_{578}-f_{589}f_{236}+f_{235}f_{689}),
\end{array}$$
$$\begin{array}{rcl}
\gamma_{1122}&=&24(f_{689}f_{379}-f_{589}f_{278}-f_{568}f_{245}+f_{569}f_{346}+f_{679}f_{389}+f_{479}f_{178}+ \\ 
& & +f_{469}f_{356}+f_{467}f_{145}-f_{578}f_{289}+f_{478}f_{179}
-f_{458}f_{256}+f_{457}f_{146}+\\
& & +f_{239}f_{136}-f_{238}f_{125}+f_{139}f_{236}+f_{137}f_{124}
-f_{128}f_{235}+f_{127}f_{134}),
\end{array}$$
$$\begin{array}{rcl}
\gamma_{1123}&=&
48(f_{179}f_{679}+f_{145}f_{458}+f_{278}f_{478}+f_{256}f_{569}+f_{389}f_{589}+f_{346}f_{467}+ \\
 & & +f_{134}f_{139}+f_{125}f_{127}+f_{236}f_{238}).
\end{array}$$
By permuting the indices we get a total number of $34$ elements in $S^2(\wedge^3 V_9^\vee)$. In order to interpret them as quadrics, one needs to quotient by the elements in $S^2(\wedge^3 V_9^\vee)$ which are contained in the ideal of 
$G(3,V_9)\subset\PP(\wedge^3 V_9)$. This can be done with the help of \cite{Macaulay2} (see attached file {\bf construction\_gamma\_w3C9}) and one finally gets $30$ quadric sections of $G(3,V_9)$; since this is also the dimension of the Macdonald representation generated by $c_1^4$ and since the coefficient of $\gamma_{1111}$ is different from zero, we deduce that this representation coincides with $\Theta_{30}$. Once $\gamma_2$ is written in coordinates, again a computation with \cite{Macaulay2} (see attached file {\bf gopel\_w3C9}) allows to deduce the remaining statements, by checking that the base locus of $\gamma_2$ is empty, that $\gamma_2$ has degree one, and that the subscheme defined by the maximal minors of the Jacobian matrix of $\gamma_2$ is also empty.
\end{proof}

\begin{remark}
Recall from \cite{gs} that $\fc$ parametrizes triples $(C,p,\omega)$, where $C$ is a genus two curve, $p$ is a Weierstrass point in $C$ and $\omega$ is a level three structure. Moreover $G_{32}$ acts transitively on level three structures, so that $\fh/G_{32}\cong \wedge^3 V_9 /\hspace{-1mm}/ \SL_9$ parametrizes couples $(C,p)$. The degree six rational map $\gamma_1:\PP(\fc)\to \PP^4$ whose image is the Burkhardt quartic consists in forgetting the Weierstrass point $p$; thus Coble cubics, modulo $\GL_9$, parametrize curves of genus two. On the other hand, $\gamma_2:\PP(\fc)\to \cG_2\subset \PP(\Theta_{30})$ is birational; thus, modulo $\GL_9$, Coble quadrics parametrize couples $(C,p)$, i.e. genus two curves with a chosen Weierstrass point; moreover, there exists a degree six rational map from $\cG_2$ to the Burkhardt quartic which consists in forgetting the Weierstrass point.
\end{remark}

\section{Genus two curves and fourtuples of spinors in ten dimensions}
\label{sec_fortuples_spinors}

This case is connected to the $\ZZ_4$-grading 
$$\fe_8=\fsl_4\times\fso_{10}\oplus (\CC^4\otimes\Delta_+)\oplus (\wedge^2\CC^4\otimes V_{10}) \oplus (\wedge^3\CC^4\otimes\Delta_-).$$
Orbits in $\CC^4\otimes\Delta_+$ were classified in \cite{degraaf} and the geometry of the corresponding sections of the spinor tenfold was studied in \cite{liu-manivel}. One can associate to an element of
 $\CC^4\otimes\Delta_+$ the following objects: a quadratic line complex (i.e. a quadric section of $G(2,4)$), a quadric section of the flag manifold $Fl_4\subset\PP^3\times\check{\PP}^3$, and a quartic Kummer surface in $\PP^3$. 
 The first and third constructions already appear in \cite{GSW}; the second one will be explained below. They can be interpreted in terms of three equivariant morphisms 
 $$\begin{array}{rcccl}
 \Gamma_1 & : & S^4(\CC^4\otimes\Delta_+) &\lra& S_{22}\CC^4\simeq S_{22}(\CC^4)^\vee, \\
 \Gamma_2 & : & S^8(\CC^4\otimes\Delta_+) &\lra& S_{422}\CC^4\subset S^4\CC^4\otimes S^4(\CC^4)^\vee, \\
 \Gamma_3 & : & S^{12}(\CC^4\otimes\Delta_+) &\lra& S_{444}\CC^4\simeq S^4(\CC^4)^\vee. \end{array}$$
 
A Cartan subspace $\fc$ has a basis
$$\begin{array}{ccc}
 p_1 & = & a_1\otimes e_{53} +a_2\otimes  e_{1245} +a_3\otimes  e_{42} +a_4\otimes  e_{31}, \\
 p_2 & =& a_1\otimes e_{52} +a_2\otimes  e_{1345} +a_3\otimes   e_{34} +a_4\otimes   e_{12}, \\
 p_3 &=& a_1\otimes e_{1234} +a_2\otimes   1 +a_3\otimes  e_{1235} +a_4\otimes   e_{54}, \\
 p_4 &=& a_1\otimes e_{14} +a_2\otimes   e_{23} +a_3\otimes   e_{51} +a_4\otimes   e_{2345},
\end{array}$$ 
where we identify $\Delta_+$ with the even part $\wedge^+E$ of the exterior algebra of some maximal isotropic subspace $E$ of $V_{10}$, with basis $e_1,\ldots ,e_5$.  Restricting $\Gamma_1, \Gamma_2, \Gamma_3$ to $\fc$ we get rational maps (the first two turn out to be morphisms) 
$$\gamma_1:\PP(\fc)\lra \PP(\Theta_5), \qquad \gamma_2:\PP(\fc)\lra \PP(\Theta_9), \qquad \gamma_3:\PP(\fc)\dasharrow \PP(\Theta'_5), $$
whose images are instances of our G\"opel varieties. In fact it is shown in \cite{liu-manivel}
that the images of $\gamma_1$ is
 the Castelnuovo-Richmond (or Igusa) quartic, while the image of $\gamma_3$ is the Segre cubic primal in $\PP^4$. Let us therefore focus on $\Gamma_2$ and $\gamma_2$. 

\subsection{A Coble quadric in the flag manifold}  
Let $v$ be a general element of 
$\CC^4\otimes\Delta_+$. As we just mentionned (see \cite{GSW} and \cite{liu-manivel}), one can associate to $v$
\begin{itemize}
    \item a Kummer quartic surface $K_v\subset\PP^3$, 
    \item a quadratic complex of lines $Q_v\subset G(2,4)$. 
\end{itemize}
In order to define this quadratic complex, the crucial observation is that $S_{22}\Delta_+$ contains 
a (unique) $\Spin_{10}$-invariant line (defining the \emph{spinor quadratic complex} from \cite{kuz_spin}). Using $v$ we can then pull back this invariant line to a line in 
$S_{22}\CC^4\simeq H^0(G(2,4),\cO(2))$. 

This line defines a quadratic form $\kappa$ on $\wedge^2V_4$ (up to scalar) and  we can then obtain the Kummer surface $K_v$ in $\PP^3$ as  the surface of lines $L\subset V_4$ such that $\kappa$ is singular on $L\wedge V_4$. 
Generically, this surface is singular along the sixteen points where 
the restriction of $\kappa$ has rank one. Note that the kernel of this restriction 
is then a hyperplane $L\wedge H$, for some hyperplane $H$ of $V_4$ containing $L$. This shows that $Sing(K_v)$ is naturally embedded into the 
flag manifold $Fl_4=Fl(1,3,V_4)$.  And indeed, we can define the dual Kummer 
surface $K_v^\vee$ in the dual $\PP^3$, as the set of hyperplanes $H$ of $V_4$ such that $\kappa$ is singular on $\wedge^2V_4$. 

That we define Kummer surfaces can be seen by considering the rank two vector 
bundle $\cE$ on $Fl_4$ with fiber $L\wedge H$. This bundle is endowed with
a natural quadratic form induced by $\kappa$, and this section of 
$S^2\cE^\vee$ vanishes along a smooth surface $A_v$. That this surface 
is an abelian surface follows from two facts: by adjunction, its canonical bundle
is trivial; the Hodge number $h^{0,1}=2$, as follows from a simple 
computation involving the Koszul complex and the Bott-Borel-Weil theorem.
Finally, it is easy to check that the projections of $A_v$ to $\PP^3$ 
and its dual have degree two over the Kummer surface and its dual,
respectively. 

One should observe that from this perspective, the abelian surface 
$A_v$ appears as the singular locus of the hypersurface $B_v$ of $Fl_4$
where $\kappa$ degenerates. 

\begin{prop}\label{unicityF} 
For $v$ generic, the hypersurface $B_v$ is the unique quadric section of $Fl_4$ that is singular along the abelian surface $A_v$. 
\end{prop}

\proof We  start with the Koszul resolution 
$$0\lra \det(\cE)^3\lra S^2\cE\otimes \det(\cE)\lra S^2\cE\lra \cI_{A_v}\lra 0.$$
Using Lemma \ref{resolution} we get
a resolution of $\cI_{A_v}^2$, namely 
$$0\lra S_{211}(S^2\cE)\lra S_{21}(S^2\cE)\lra S^{2}(S^2\cE)\lra \cI_{A_v}^2\lra 0.$$
We have $K_2:=S_{211}(S^2\cE)=S^2\cE\otimes \det(\cE)^3$, $K_1:=S_{21}(S^2\cE)=
S^4\cE\otimes \det(\cE)\oplus S^2\cE\otimes \det(\cE)^2$ and $K_0=S^{2}(S^2\cE)=S^4\cE\oplus \det(\cE)^2$. Using the Bott-Borel-Weil theorem  we compute that the only non zero cohomology groups of their twists by $\cO(2)=\det(\cE)^{-2}$ are 
$$H^0(K_0(2))=H^3(K_2(2))=\CC, \; H^2(K_0(2))=S_{22}\CC^4, \;  H^2(K_1(2))=S_{211}\CC^4, \;  H^5(K_1(2))=\wedge^2\CC^4.$$
This easily implies that $H^0(\cI_{A_v}^2(2))=H^0(K_0(2))=\CC$, and our claim follows. \qed

\medskip 
The equation of $B_v$ as a quadratic function of $Q_v$ must be given by a morphism
$$\Gamma_0 : S^2(S_{22}V_4) \lra S_{422}V_4.$$
By the Littlewood-Richardson rule, $S_{422}V_4$ has multiplicity one inside $S_{22}V_4\otimes S_{22}V_4$, in particular there is a unique such equivariant morphism (up to scalar). We will construct it 
as a composition 
$$S^2(S_{22}V_4) \lra S^2(S^2(\wedge^2V_4))\stackrel{\Theta}{\lra} S^2V_4\otimes S^2(\wedge^3V_4)\lra S_{422}V_4\subset S^2V_4\otimes S^2V_4^\vee.$$
We define  $\Theta$ explicitly by sending $\big((a_1\wedge a_2)(b_1\wedge b_2)\big).
\big((c_1\wedge c_2)(d_1\wedge d_2)\big)$ to 
$$\begin{array}{l}
\;\;\; a_1c_1\otimes (b_1\wedge b_2\wedge c_2)(a_2\wedge d_1\wedge d_2)
-a_2c_1\otimes (b_1\wedge b_2\wedge c_2)(a_1\wedge d_1\wedge d_2) \\
-a_1c_2\otimes (b_1\wedge b_2\wedge c_1)(a_2\wedge d_1\wedge d_2)
+a_2c_2\otimes (b_1\wedge b_2\wedge c_1)(a_1\wedge d_1\wedge d_2) \\
+b_1c_1\otimes (a_1\wedge a_2\wedge c_2)(b_2\wedge d_1\wedge d_2)
-b_2c_1\otimes (a_1\wedge a_2\wedge c_2)(b_1\wedge d_1\wedge d_2) \\
-b_1c_2\otimes (a_1\wedge a_2\wedge c_1)(b_2\wedge d_1\wedge d_2)
+b_2c_2\otimes (a_1\wedge a_2\wedge c_1)(b_1\wedge d_1\wedge d_2) \\
+a_1d_1\otimes (b_1\wedge b_2\wedge d_2)(a_2\wedge c_1\wedge c_2)
-a_2d_1\otimes (b_1\wedge b_2\wedge d_2)(a_1\wedge c_1\wedge c_2) \\
-a_1d_2\otimes (b_1\wedge b_2\wedge d_1)(a_2\wedge c_1\wedge c_2)
+a_2d_2\otimes (b_1\wedge b_2\wedge d_1)(a_1\wedge c_1\wedge c_2) \\
 +b_1d_1\otimes (a_1\wedge a_2\wedge d_2)(b_2\wedge c_1\wedge c_2)
-b_2d_1\otimes (a_1\wedge a_2\wedge d_2)(b_1\wedge c_1\wedge c_2) \\
-b_1d_2\otimes (a_1\wedge a_2\wedge d_1)(b_2\wedge c_1\wedge c_2)
+b_2d_2\otimes (a_1\wedge a_2\wedge d_1)(b_1\wedge c_1\wedge c_2).
\end{array}$$
This formula has the required symmetries: it is skew-symmetric in $a_1,a_2$, symmetric in $a,b$,
and also symmetric when we exchange $a$ with $c$ and $b$ with $d$.

As always, we only need to compute  
the restriction of $\Gamma_0$ to a suitable subspace; as in \cite[Lemma 5.15]{liu-manivel}, we can choose the space 
$\fc_{qc}$ consisting of quadratic complexes of the form 
$$\Theta = 
P(\pi_{12}^2-\pi_{34}^2)+Q(\pi_{13}^2+\pi_{24}^2)+R(\pi_{14}^2+\pi_{23}^2)+2S\pi_{12}\pi_{34}  + 2T \pi_{14}\pi_{23}-2U\pi_{13}\pi_{24},$$
with $P,Q,R,S,T,U$ are coefficients such that $S+T+U=0$ and the $\pi_{ij}$'s are Pl\"ucker coordinates on $G(2,4)$ for $1\leq i<j\leq 4$. (Note that by Lemma  \ref{CartanS22}, this is a genuine Cartan subspace!). 

\begin{lemma} 
The restriction of $\Gamma_0$ to $\fc_{qc}$ is the morphism sending 
$[P,Q,R,S,T,U]$ to $$[QR, PR, PQ, S^2+Q^2-R^2,  T^2+P^2-Q^2, U^2+R^2-P^2, UT-PS, TS-QU, SU-RT].$$
\end{lemma}

Finally, we deduce $\gamma_2$
by restriction to $\fc$, as a composition  
$$\gamma_2: S^8\fc\lra S^2\fc_{qc}\subset S^2(S_{22}V_4) \stackrel{\Gamma_0}{\lra} S_{422}V_4\subset S^2V_4\otimes S^2V_4^\vee.$$
We are now ready for an explicit computation, whose result is the following: 

\begin{prop}
The image of $\gamma_2 : \PP(\fc)\longrightarrow \PP(\Theta_9)\subset \PP(S^2V_4\otimes S^2V_4^\vee)$ is an isomorphic projection of the second Veronese embedding of the Castelnuovo-Richmond (or Igusa) quartic. Moreover $\Theta_9$ is the 
$G_{31}$-representation induced from the  $9$-dimensional representation $[4,2]$ of $S_6$; it is again a Macdonald representation. 
\end{prop}

\subsection{Spinor bundles and orbital degeneracy loci}
Recall that each orthogonal Grassmannian
 $OG(k,10)$ admits a spinor bundle, whose space of global sections is $\Delta$. 
 
 \medskip
 On the quadric $\QQ^8$, denote by $\cS_\pm$ the rank eight spinor bundles, whose spaces of global sections are the half-spin representations $\Delta_\pm$. 
 Recall that the fiber of $\cS_+$ at $x\in\QQ^8$ is the quotient $\Delta_+/x\Delta_-$, where $x\Delta_-$ is the image of $\Delta_-$ by Clifford multiplication by $x$.

Given a generic $v\in \CC^4\otimes \Delta_+$, defining a subspace $P_v\subset \Delta_+$, 
the associated morphism 
 $$\psi^1_v: P_v\otimes\cO_{\QQ^8}\lra \cS_+$$
 degenerates along a smooth threefold $D(v)$, and the kernel map
 is an isomorphism $D(v)\lra\PP^3$. Indeed, $\psi^1_v$ degenerates at $x\in \QQ^8$
when $P_v\cap x\Delta_-\ne 0$. For a nonzero $\theta$ in this intersection,  $x.\theta=0$; for $v$ general $\theta$ is not a pure spinor  and there  is a unique such $x=\nu(\theta)$ (up to scalar). So $D(v)= \nu (P_v)\simeq v_2\PP^3$. 

 On the other hand, $\cS_+$ is modeled on a half-spin representation of $\Spin_8$, which is by triality equivalent to the natural representation; in particular it contains a non-degenerate quadric, given at a point $x\in\QQ^8$, by the formula
 $$q_x(\bar{\theta})=\langle \nu (\theta), x\rangle, $$
 which is easily seen to make sense. 
Globally, this corresponds to an isomorphism between $\cS_+$ and $\cS_+^\vee(1)$. 
 Transplanting the quadric  through $v$, we get a morphism from $\CC^4\otimes\cO_{\QQ^8}$ to 
 $(\CC^4)^\vee\otimes\cO_{\QQ^8}(1)$, whose determinant yields a section of $\cO_{\QQ^8}(4)$, 
 that is, a quartic section of the quadric. The corresponding morphism
 $$\Gamma_1^{10} : S^8(\CC^4\otimes\Delta_+)\lra S^{\langle 4\rangle}V_{10}$$
 is characterized by its restriction to a Cartan subspace $\fc$; this can be deduced from the fact that an element $a\in \fc$ induces a map $\tilde{a}:S^2\CC^4 \to V_{10}$ as in \cite[Lemma 5.7]{liu-manivel} (where $\tilde{a}$ is denoted $\gamma_a$); in this paper, the authors give the explicit expression of the symmetric matrix $M$ of this quadratic form:
 $$\begin{array}{rcl}
 M_{11} & = & 
-2a_2a_3e_2+2a_1a_4f_2-2a_1a_3e_3-2a_2a_4f_3,\\
M_{12} & = & (a_1^2-a_2^2)e_5+(a_3^2-a_4^2)f_5,\\
M_{13} & =& 2a_3a_4e_1+(a_1^2+a_2^2)f_1, \\
M_{14} & = &  -(a_3^2+a_4^2)e_4+2a_1a_2f_4, \\
M_{22} &= & 2a_1a_4e_2-2a_2a_3f_2+2a_2a_4e_3+2a_1a_3f_3, \\
M_{23} &=& -2a_1a_2e_4-(a_3^2+a_4^2)f_4, \\
M_{24} & = & -(a_1^2+a_2^2)e_1+2a_3a_4f_1, \\
M_{33} & =&  -2a_1a_3e_2-2a_2a_4f_2-2a_2a_3e_3-2a_1a_4f_3, \\
M_{34} & = & -(a_3^2-a_4^2)e_5+(a_1^2-a_2^2)f_5,  \\
M_{44} & = & 2a_2a_4e_2+2a_1a_3f_2+2a_1a_4e_3+2a_2a_3f_3. 
 \end{array}$$
 
 \smallskip
 Its determinant yields the restriction 
 $$\gamma_1^{10}: \PP(\fc) \dashrightarrow  \PP (S^{\langle 4\rangle}V_{10})$$
 whose image is another G\"opel type variety. 
 
 \medskip
 For $k=2$, the spinor bundle $\cS_+$ has rank four. The following observation seems new:
 
 \begin{prop}\label{horo}
 Let $s_i$ be a general section of $\cS_+^{\oplus i}$ on $OG(2,10)$, and $Z_i$ its zero locus. 
 \begin{itemize} 
 \item $Z_1$ is a horospherical variety for $\Spin_7$,
 \item $Z_2$ is the adjoint variety of the exceptional group $G_2$, 
 \item $Z_3$ is a smooth conic. 
 \end{itemize}
 \end{prop}
 
 \proof 
For each $i\le 3$, $Z_i$ is smooth of dimension $13-4i$ and index $7-2i$. This is enough to ensure that $Z_1$ is a smooth conic. According to \cite[Proposition 32]{sk}, $Z_2$ is acted on by $G_2$; but $G_2$ has no non-trivial homogeneous space of dimension smaller that five; in dimension five it has only two, the quadric of index $5$,  and the adjoint variety of index $3$, which has therefore to coincide with $F$.

By \cite[Proposition 31]{sk}, $Z_1$ is acted on by an extension of $\Spin_7$. Let us denote by $\Delta^+(V_n)$ and $\Delta^-(V_n)$  the two half-spin representation of $\Spin(V_n)$ for $n$ even, and by $\Delta(V_n)$ its unique spin representation for $n$ odd. 
A general element $\phi\in \Delta^+(V_{10})$ defines a general section $s_1$ of the spinor bundle $\cS_+$. For
an isotropic plane $P\in OG(2,V_{10})$, $s_1(P)$ vanishes if and only if $\phi$ belongs to the image of the map $$P\otimes \Delta^-(V_{10})\subset V_{10}\otimes \Delta^-(V_{10}) \to \Delta^+(V_{10}),$$ where the rightmost morphism is defined via Clifford multiplication. 
If $V_{10}=E\oplus F$ with $E$ and $F$ two transverse isotropic subspaces, we can identify $\Delta^+(V_{10})$ with $\wedge^-F$, the odd part of the exterior algebra of $F$. Inside $\Delta^+(V_{10})$, spinors that are not pure form a single $\Spin_{10}$-orbit. 
Let us choose a basis $e_1,\ldots,e_5$ of $E$, and let 
$f_1, \ldots,f_5$ be the dual basis of $F$;  we can assume that $\phi=f_1\wedge (1+f_{2345})$. Indeed this is not a pure spinor, but the sum of the two pure spinors $f_1$ and $f_{12345}$ associated to the maximal isotropic spaces $W:=\langle f_1\cdots,f_5\rangle$ and $W':=
\langle f_1,e_2\cdots,e_5\rangle$. Let $V_1:=W\cap W'$ and $V_8:=(W+W')/V_1$; the latter comes endowed with a quadratic form, induced by the restriction of the quadratic form on $V_{10}$. Let us also let $U_1=\langle e_1 \rangle$; we can thus decompose $$\wedge^2 V_{10}= (V_1\otimes V_8) \oplus \wedge^2 V_8 \oplus (V_8 \otimes U_1) \oplus (V_1\otimes U_1).$$
 Elements in $OG(2,V_{10})$ of the form $[f_1\wedge v]$ are zeroes of  $s_1$, so that the quadric $\QQ(V_8/V_1)$ is contained in $Z_1$; elements of the form $[e_1\wedge v]$ are never zeroes of $s_1$ and this implies that  the linear span of $Z_1$ is $\PP((V_1\otimes V_8) \oplus \wedge^2 V_8)$. Now notice that $\phi$ defines a generic element $\psi\in \Delta^+(V_8)$; in coordinates, this is just 
  $$\psi=1+f_{2345}\in \wedge^+ (F/\langle f_1\rangle)\simeq \Delta^+(V_8).$$ 
  In particular $\psi$ is a generic element in $\Delta^+(V_8)$, whose orthogonal (w.r.t. the unique $\Spin_8$-invariant quadratic form on $\Delta^+(V_8)$) is a $V_7\subset \Delta^+(V_8)$ endowed with a non-degenerate quadratic form; the copy of $\Spin_7$ in the stabilizer $\phi$ identifies with $\Spin(V_7)$. We deduce $\Spin_7$-equivariant identifications $\Delta(V_7)\simeq V_8$ and $\wedge^2 V_8\simeq \wedge^2\Delta^+(V_8)\simeq\wedge^2 V_7 \oplus V_7$. Following these identifications, $\psi$ defines a section of the rank $2$ spinor bundle $\cS'$ on $OG(2,V_8)$. By triality, this bundle can be exchanged with the dual tautological bundle, so the zero-locus of a general section  is isomorphic to $OG(2,V_7)$. Putting all together we deduce that both $ \QQ(V_8/V_1)\simeq OG(3,V_7)$ and $OG(2,V_7)$ are contained in $Z_1$. More than that, consider a maximal weight vector $v_3\in \Delta(V_7)\simeq V_8$, and a maximal weight vector $v_2\wedge v_3 \in \wedge^2 V_7\subset \wedge^2\Delta(V_7)$, with $v_2 \in \Delta(V_7)$. The fact that $\psi(\langle v_2\wedge v_3\rangle)=0$ means that $$\exists \widetilde{\psi_2},\widetilde{\psi_3}\in \Delta^-(V_8)\quad \mid \quad \psi=v_2\cdot \widetilde{\psi_2}+v_3 \cdot \widetilde{\psi_3} .$$ This implies that $x:=[f_1\wedge v_2 + v_2\wedge v_3]=[v_2\wedge (v_3-f_1)]\in OG(2,V_8)$ is a zero of $s_1$  since $$\phi=v_2\cdot (f_1\wedge \widetilde{\psi_2})+(v_3-f_1)\cdot (f_1\wedge \widetilde{\psi_3}).$$ 
  Finaly, since $Z_1$ is $\Spin_7$-invariant, it must contain the $\Spin_7$-orbit closure of $x\in \PP(\Delta(V_7)\oplus \wedge^2 V_7)\subset \PP(\wedge^2 V_7)$, i.e. it must contain the horospherical variety denoted by $(B_3,\omega_2,\omega_3)$ in \cite{Pasquier}. Both varieties are smooth of the same dimension, and using the Koszul complex for zero loci one shows that $H^0(Z_1,\cO_{Z_1})\cong \CC$, i.e. that $Z_1$ is connected; so the two varieties must coincide.
\qed 
 
 \medskip
 A vector $v\in \CC^4\otimes\Delta_+$ defines a morphism
 $$\psi^2_v: P_v\otimes\cO_{OG(2,10)}\lra \cS_+.$$
 Let $D_i(v)\subset OG(2,V_{10})$ be the 
 locus where the rank of $\psi^2_v$ is at most $i$. 
 
 \begin{prop} Let $v$ be general in $\CC^4\otimes \Delta_+$. 
 \begin{itemize}
     \item $D_3(v)$ is a quadric section of $OG(2,V_{10})$, singular along $D_2(v)$.
     \item $D_2(v)$ admits a resolution $\tilde{D}_2(v)$ which projects to $G(2,4)$
     with generic  fiber isomorphic to the adjoint variety of $G_2$.
     \item $D_1(v)$ is a smooth Fano fourfold of index one; it 
     admits a natural projection to $\PP^3$, which is a conic fibration     with discriminant the Kummer quartic surface $K_v$. 
 \end{itemize}
 \end{prop}
 
 \proof The first statement follows from the fact that $\det(\cS_+)=\cO(2)$. 
 The standard resolution  $\tilde{D}_2(v)$ of $D_2(v)$ is the set of pairs 
 $(L,x)$ in $G(2,P_v)\times OG(2,V_{10})$ such that  $P_v\subset L\Delta_+$. The fiber over $L$ of the projection $\tilde{D}_2(v)\lra G(2,P_v)$ only depends on the orbit of $L$ in $G(2,\Delta_+)$; it is the zero locus of a pair of sections of the spinor bundle, hence by Proposition \ref{horo} a copy of the adjoint variety of $G_2$. 
 
 Over $D:=D_1(v)$, the rank of $\psi^2_v$ is everywhere equal to one, so
 there is an exact sequence
 $$0\lra K\lra \CC^4\otimes\cO_{D}\lra\cS_{+|D}\lra C\lra 0,$$
 where the kernel $K$ and the cokernel $C$ are both vector bundles of rank three. As usual, 
 the normal bundle of $D$ in $OG(2,V_{10})$ is $N_D=Hom(K,C)$, whose determinant, by the 
 exact sequence above, is $\det (N_D)=\det(K^\vee)^3\otimes \det (C)^3=\det(\cS_{+|D})^3=\cO_D(6)$. 
 By adjunction, we deduce that $\omega_D=\cO_D(-1)$, so $D$ is Fano. 
 
 Since $K$ is a corank one subbundle of a trivial rank four bundle, it defines a morphism
  $p$ from $D$ to $\PP^3$, the space of hyperplanes. The fiber over $H\in\PP^3$ is the
  zero-locus of a triple of sections of $\cS_+$, hence a conic in general  by Proposition \ref{horo}. 
  This conic only depends on the orbit of $H$ in $G(3,\Delta_+)$, and becomes singular when $H$ degenerates to the codimension one orbit. According to \cite{liu-manivel}, this happens  exactly when $H$ belongs to the Kummer surface $K_v$.  
   \qed

 \medskip
 The equation of the quadric $D_3(v)\subset OG(2,V_{10})$ is given by a morphism 
  $$\Gamma_2^{10} : S^4(\CC^4\otimes\Delta_+)\lra S_{\langle 2,2\rangle}V_{10}.$$
 Restricting to the Cartan subspace $\fc$, we get a map 
 $$\gamma_2^{10} : \PP(\fc)\dashrightarrow  \PP(S_{\langle 2,2\rangle}V_{10})$$ 
 whose image is another G\"opel variety. One can show that  $\gamma_1^{10}$ is the composition of $\gamma_2^{10}$ with a quadratic map. Moreover, W. de Graaf checked that the morphism  $S^4\fc\lra S_{\langle 2,2\rangle}V_{10}$ is injective, so that  $\gamma_2^{10}$  is just a degree four Veronese embedding.  This is in sharp contrast with the other cases with discussed, especially since $S^4\fc$ is not an irreducible representation of the complex reflection group $G_{31}$, but splits as the sum of two submodules of dimension $5$ and $30$. Nevertheless, the image of $\gamma_2^{10}$ is a family of very special quadratic sections of $OG(2,10)$, since they are singular in dimension five, and the singular locus of their singular locus is in general a  smooth Fano fourfold, closely related to a Kummer surface.

\section{Special genus four curves and spinors in sixteen dimensions}

This case comes from the ultimate model, the $\ZZ_2$-grading 
$$\fe_8=\fso_{16}\oplus\Delta,$$ where $\Delta$ is a half-spin representation of dimension $128$. The symmetric space $E_8/\Spin_{16}$ is conjectured to admit a compactification that would deserve to be coined a bioctonionic projective plane. For this group the geometric identification of Coble type hypersurfaces in terms of moduli spaces of vector bundles on curves is yet to be fully uncovered. Even so, it is expected that the hypersurfaces of which we will give a construction in this section should be related to special non-hyperelliptic trigonal curves of genus four (see \cite{GSW}). Steven Sam and Eric Rains explained to us some unpublished results of theirs that go even further in the direction that we explore in this section.

\subsection{The grading} 
The Lie bracket on $\fe_8=\fso_{16}\oplus \Delta$, since it is compatible with the grading,  can be decomposed into three $\Spin_{16}$-equivariant morphisms: the classical bracket on $\fso_{16}$, a map $\beta_1: \fso_{16}\otimes \Delta \to \Delta$ and a map $\beta_2:\bigwedge^2 \Delta \to \fso_{16}$. The fact that these morphisms are $\Spin_{16}$-equivariant identifies them uniquely up to nonzero scalar, because the target is an irreducible representation appearing in the source with multiplicity one. 

In particular $\beta_1$ must be given by the action of $\fso_{16}$ on $\Delta$. It can be explicitly described in terms of the Clifford action, 
 once we identify $\fso_{16}$ with $\wedge^2 V$. Indeed, there is an equivariant 
 morphism $V\otimes \Delta \to \Delta_-$, the other half-spin representation. 
 Recall that $\Delta$ and $\Delta_-$ are constructed by choosing a decomposition
 $V=E\oplus F$ into maximal isotropic subspaces, and letting $\Delta=\wedge^+E$ and $\Delta_-=\wedge^-E$. Then $e\in E$ and $f\in F$ act by sending $v\in\Delta$ or $\Delta_-$ to the wedge product 
 $e\wedge v$ and the contraction $f\lrcorner \, v$, respectively. By composition, 
 we readily deduce a morphism $\wedge^2V\otimes \Delta\ra V\otimes \Delta_-\ra
 \Delta$, which has to coincide with $\beta_1$, up to nonzero scalar. 
 We can then construct $\beta_2$ from $\beta_1$.
Indeed,  
since $\Delta$ and $\fso_{16}$ are selfdual, $\beta_1$ induces a morphism
 $$ \wedge^2\Delta\hookrightarrow \Delta\otimes \Delta\cong \Delta\otimes \Delta^\vee \to \fso_{16}^\vee \cong \fso_{16},$$ 
 which can be checked to be nonzero and must therefore be proportional to $\beta_2$. An invariant quadratic
 form on $\Delta=\wedge^+E$ can be defined by the wedge 
 product followed with the projection to $\wedge^8E\simeq \CC$. We can then give an explicit
 formula in terms of a basis $(a_i)$ of $V$ and the dual basis $(b_i)$:
 \begin{equation}\label{beta2}
 \beta_2(h,h')=\sum_{i<j}\langle (a_i\wedge a_j)h,h'\rangle b_i\wedge b_j.
 \end{equation}
 
 \subsection{A Cartan subspace}
 From the general theory \cite{vinberg}, we know that $\Delta$ contains Cartan subspaces; in this case these are just Cartan subalgebras of $\fe_8$, contained in $\Delta$.
 Here is one, given in terms of a basis $e_0,\ldots ,e_7$ of $E$. We will also denote 
 by $f_0,\ldots ,f_7$ the dual basis of $F$. 
 
\begin{prop}\label{cartanbasis}
A Killing-orthonormal basis of a Cartan subspace $\fc\subset \Delta$  is given by %the following tensors:
\begin{align*}
h_1=e_{0123}+e_{4567}, & \qquad h_2=e_{0145}+e_{2367},\\h_3=e_{0246}+e_{1357}, &\qquad h_4=e_{0356}+e_{1247},\\h_5=e_{0257}+e_{1346}, &\qquad h_6=e_{0347}+e_{1256},\\h_7=e_{0167}+e_{2345}, & \qquad h_8=1_E+e_{01234567}.
\end{align*}
\end{prop}

\begin{proof}
We need to check that $h_1,\ldots , h_8$ are semisimple and commute. 
Let $\ft$ denote 
the standard Cartan subalgebra of $\fso_{16}$, for which the $e_i$'s and $f_j$'s are 
weight vectors in $V$. Then $\ft$ is also a Cartan subalgebra of $\fe_8$, 
whose roots 
are thus partitioned into the roots of $\fso_{16}$ and the weights of $\Delta$.
Now, observe that $e_{0123}$ and $e_{4567}$ are in $\fe_8$ root vectors for opposite 
roots $\alpha$ and $-\alpha$; the corresponding root spaces generate a subalgebra
 $\fg_\alpha+\fg_{-\alpha}+[\fg_\alpha,\fg_{-\alpha}]\simeq\fsl_2$. As an element of this $\fsl_2$,  $h_1$ is represented by a matrix with zero diagonal coefficients and nonzero off diagonal 
 coefficients. So $h_1$ is semisimple, and by the same argument $h_2,\ldots, h_8$ are semisimple as well. 
 
 Now we check that they commute, by applying formula 
 (\ref{beta2}) for the basis $(e_i,f_j)$, whose dual is the basis $(f_i,e_j)$. It is clear that $[h_i,h_8]=\beta_2(h_i\wedge h_8)=0$ for any $i\le 7$, 
 because since $h_i\in\wedge^4E$, we always have $(a_p\wedge a_q)h_i\in \wedge^2E\oplus \wedge^4E\oplus \wedge^6E=(\wedge^0E\oplus\wedge^8E)^\perp$. 
 For $i,j\le 7$, the same argument shows that $(a_p\wedge a_q)h_i$ is orthogonal to $h_j$ unless
 possibly when $a_p\wedge a_q$ is of the form $e_r\wedge f_s$. The action of such an operator on a monomial
 $e_{abcd}$ replaces the index $s$, if present, by $r$. 
 But two monomials from  $h_i$ and $h_j$ always have 
 two indices in common; after the action of  $e_r\wedge f_s$ they will still have at least one index in common, and their wedge product will therefore vanish.  
\end{proof}

\begin{remark}
By computing the rank of $ad(h_i)$ one can check that 
each $h_i$ is a coroot. In order to make an explicit connexion with the root system of $E_8$, it might 
be useful to exhibit an octuple of eight orthogonal roots, to show how we can recover from those the whole root system and the full Weyl group. There is only one 
such octuple up to the Weyl group action, and one instance is the following:

$$\begin{array}{cccc}
\xymatrix@=.1em{0&1&0&0&0&0&0\\&&0&&&&} & 
\xymatrix@=.1em{0&0&0&1&0&0&0\\&&0&&&&} & 
\xymatrix@=.1em{0&0&0&0&0&1&0\\&&0&&&&} & 
\xymatrix@=.1em{0&0&0&0&0&0&0\\&&1&&&&} \\
 & & & \\
\xymatrix@=.1em{0&1&2&1&0&0&0\\&&1&&&&} & 
\xymatrix@=.1em{0&1&2&2&2&1&0\\&&1&&&&} & 
\xymatrix@=.1em{2&3&4&3&2&1&0\\&&2&&&&} & 
\xymatrix@=.1em{2&4&6&5&4&3&2\\&&3&&&&}
\end{array}$$

\end{remark}

\medskip
More generally, we have the following Lemma
(see \cite[10.2]{oshima}).

\begin{lemma}\label{lem_subsystems_E8}
Up to the action of the Weyl group, the root system $E_8$ contains a unique subsystem
of type $2A_1, 3A_1$ or $8A_1$, and two distinct root subsystems of type $4A_1$.
\end{lemma}

Note that a subsystem of type $mA_1$ is just given by $m$ pairwise orthogonal roots.
In particular any $m$-tuple of orthogonal roots is equivalent to $h_1,\ldots , h_m$,
up to the action of the Weyl group, for $m=1,2,3,8$ but not for $m=4$. We will see that  $h_1, h_2, h_3, h_4$ and  $h_1, h_2, h_4, h_7$ belong to 
different orbits.

\begin{remark}
Since the Weyl group of $E_8$ acts transitively on roots, $h_8^\perp$ is
equivalent to the  system orthogonal to a highest root, which is 
a root system of type $E_7$ (indeed in this situation $h_8=\omega_8$,
and $E_7$ is generated by roots $\alpha_1,\dots,\alpha_7$ in Bourbaki's notation). To make this even clearer, notice that $\bigwedge^4 E \subset \Delta$, and recall from Section \ref{sec_genus_three} that $\bigwedge^4 E\cong \bigwedge^4 \CC^8$ contains a Cartan subspace of type $E_7$. One can compare the elements $h_1,\dots,h_7$ defined in this section with those  defined in section \ref{sec_genus_three} and notice that they are essentially the same.
\end{remark}

\begin{prop}
\label{prop_Heisenberg_spin16}
The centralizer of a Cartan subspace is a finite Heisenberg subgroup of $\Spin_{16}$, 
acting on $V_{16}$ by a Schr\"odinger representation.
\end{prop}

\proof Each $h_i$ can be decomposed in a unique way as the sum of two pure spinors, and this defines as in the proof of Proposition \ref{Heisenberg_E7}, a morphism $h$ from the 
centralizer $H$ to $\ZZ_2^8$. 

Let $k$ belong to the kernel $K$ of $h$. Then $k$ fixes the sixteen isotropic 
eight-planes $\langle e_0, \ldots , e_7\rangle$, $\langle f_0, \ldots , f_7\rangle$, $\langle e_0,e_a,e_b,e_c,f_p,f_q,f_r,f_s\rangle$ and $\langle e_p,e_q,e_r,e_s,f_0,f_a,f_b,f_c\rangle$, for $abc$ a line in the Fano plane with 
complement $pqrs$. By considering their intersections, we deduce that $k$ 
must act diagonally on the $e_i$'s and $f_j$'s, and belongs to the corresponding maximal torus in $\Spin_{16}$. 

As for the image $I$, as in Proposition \ref{Heisenberg_E7} we encode an element 
$x$ in $I$ by a sequence $x_1,\ldots , x_8$. First note that swapping $e_i$ and $f_i$ for each $i$ defines an element $x^{sw}$ of $I$ such that $x^{sw}_j=-1$  for all 
$j=1,\ldots ,8$. Multiplying by this element if necessary, we can then suppose 
that a given $x$ in $I$ verifies $x_8=1$. Then reasoning as in the proof of Proposition \ref{Heisenberg_E7} shows that again  $x_ax_bx_c=1$ if 
$abc$ is a line in the Fano plane, and we conclude the proof in the same way.
\qed

\subsection{The Heisenberg group.}\label{heis} The Heisenberg group of $(\fe_8,\theta)$ in Proposition \ref{prop_Heisenberg_spin16} is an extension
$$ 1 \to \ZZ_2 \to H \to (\ZZ_2)^8 \to 1 $$
acting linearly on $V$ (see also Appendix \ref{sec_heis_group}). 
%Let us decompose $V=E\oplus F$ as a direct sum of isotropic subspaces with respect to the symmetric form on $V$ (this means that $F\cong E^\vee$). 
Let us fix a basis $x_{0000},x_{0001},x_{0010},\dots,x_{0111}$ of $E^\vee$ given by Schr\"odinger coordinates, and a basis $x_{1000},x_{1001},x_{1010},\dots,x_{1111}$ of $F^\vee$. The Heisenberg group is explicitly generated by the following eight operations: 
$$\begin{array}{llll}
x_{i,j,k,l}\mapsto x_{i+1,j,k,l},\qquad &  x_{i,j,k,l}
\mapsto x_{i,j+1,k,l}, \qquad & x_{i,j,k,l}\mapsto x_{i,j,k+1,l}, & x_{i,j,k,l}\mapsto x_{i,j,k,l+1}, \\ 
x_{i,j,k,l}\mapsto (-1)^ix_{i,j,k,l}, & x_{i,j,k,l}\mapsto (-1)^jx_{i,j,k,l}, &  x_{i,j,k,l}\mapsto (-1)^kx_{i,j,k,l}, &  x_{i,j,k,l}\mapsto (-1)^lx_{i,j,k,l}.
\end{array}$$

Let us change the notation, renaming $e_0=x_{0000}^\vee$, $e_1=x_{0001}^\vee$, $\dots$, $e_7=x_{0111}^\vee$, which is a basis of $E$, and $f_0=x_{1000}^\vee$, $f_1=x_{1001}^\vee$, $\dots$, $f_7=x_{1111}^\vee$, which is a basis of $F$. Of course the renumbering 
9must be coherent with the Fano plane $\PP^2(\ZZ_2)$, the $x_{0ijk}$
being seen as points in $\ZZ_2^{\oplus 3}$. 
A direct examination leads to the following observation:

\begin{prop}\label{hheisinv}
The space of $H$-invariants in $\Delta$ is the Cartan subspace $\fc$.
\end{prop}

\subsection{Special trigonal genus four curves}

A non-hyperelliptic genus 4 curve lies on a unique quadric surface in its canonical embedding in $\PP^3$,
and the locus of curves $C$ where this quadric is singular (i.e., $C$ has a vanishing
thetanull) has codimension $1$ in the moduli space. This condition also implies that $C$ has a unique $g_3^1$, i.e. a unique degree $3$
map to $\PP^1$. If we further impose that this map has a point with non-simple ramification (i.e. triple ramification), the
locus loses another dimension, and hence has dimension $7$. 

Recall the following construction of genus four curves, starting from $8$ general points on $\PP^2$. If we blow them up, we get a degree one del Pezzo surface, which comes with a natural degree $2$ map (defined by twice its anticanonical divisor) to a quadric cone in $\PP^3$. The branch locus is a genus $4$ curve $C$, with a vanishing theta-null. If the only $g_3^1$ of $C$ has a point of triple ramification, then the corresponding  $8$ points lie on a cuspidal plane cubic curve (\cite[Remark 8.1]{GSW}) - which is not the case for $8$ generic points in $\PP^2$. Notice also that the $8$ points in $\PP^2$ define a pencil of plane cubics, which generically should contain a unique cuspidal plane cubic. 

Let $\fc$ be a Cartan subspace in $\Delta$, which is also a Cartan subalgebra of $\fe_8$. Thorne has shown how to construct a family of genus four curves parametrized by $\PP(\fc)/W$, and has given explicit equations of these curves \cite[Theorem 1.2 and Theorem 3.8]{Thorne}. It was observed by Steven Sam (personal communication) that those genus four curves have a $g_3^1$ (i.e. a vanishing theta-null) with a point of triple ramification. This gives a map from $\PP(\fc)$ (or $\PP(\fc)/W$) to the moduli space of genus four curves having a $g_3^1$ with a point of triple ramification. 

\begin{conjecture}
\label{conj_dodulispace} The quotient $\PP(\fc)/W$ is birational to the subvariety of the moduli space of genus four curves parametrizing curves with a $g_3^1$ with a point with triple ramification. 
\end{conjecture}

We expect this birational morphism to even extend to a map from $\fc$ to the configuration space of $8$
points in $\PP^2$ that lie on a cuspidal cubic in a $W$-equivariant way (the action of $W=W(E_8)$ on the configuration space of eight points in $\PP^2$ by Cremona transformations was constructed in \cite{do}). See also Remark \ref{rmk_cartan_to_eight_points} for more details.

\subsection{Spinor bundles and orbital degeneracy loci}
Exactly as in the previous sections we can consider $\Delta$ as the space of global sections of a spinor
bundle on any orthogonal Grassmannian $OG(k,16)$. Let us look more in the details the cases $k=1$ and $k=2$. 
\medskip

For $k=1$, we get a spinor bundle $\cS$ modeled on a half-spin representation of $\Spin_{14}$. Since the latter admits a degree eight invariant, 
this induces a map 
$$\Gamma_1: S^8\Delta \lra H^0(\QQ_{14},S^8\cS)\lra H^0(\QQ_{14},\cO(4))=S^{\langle 4\rangle}\CC^{16}.$$
So a generic element $v$ of $\Delta$ defines a quartic section $Qu_v$ of the quadric $\QQ_{14}$. We will study this quartic in Section \ref{sec_quartic_spin}. We can do even better. As mentioned, the fibers of the spinor bundle $\cS$ on the quadric $\QQ_{14}$ are naturally isomorphic to the half-spin representation $\Delta_{14}$ of $\Spin_{14}$. This is a parabolic representation with a finite number of orbits, classified in \cite{Weyman_E8}. Let us focus on those of small codimension. Let us denote by $Y_i$ the unique codimension $i$ orbit closure in $\Delta_{14}$ for $i=1,5,10,14$. We get the following stratification:
$$ Y_{14}\subset Y_{10}\subset Y_{5}\subset Y_1\subset \Delta_{14}. $$
Considering the associated orbital degeneracy loci 
$$ D_{Y_i}(v)=\{ x\in \QQ_{14}\mid v(x)\in Y_i\subset \Delta_{14}\}, $$
we get a stratification
$$ D_{Y_{14}}(v)\subset D_{Y_{10}}(v)\subset D_{Y_{5}}(v)\subset
D_{Y_1}(v)=Qu_v\subset \QQ_{14}, $$
with $D_{Y_i}(v)$ of codimension $i$ for general $v$. Following the analogy with the other cases presented in this paper, we expect the following conjecture to hold true:

\begin{conjecture}
Let $v$ be a general element in $\Delta$.
\begin{enumerate}[label=\roman*)]
    \item $D_{Y_{10}}(v)$ is the Kummer fourfold $\Jac(C)/\{\pm \id\}$ of a non-hyperelliptic genus four curve $C$ with a vanishing theta null and unique degree $3$ map to $\PP^1$; 
    \item $D_{Y_{14}}(v)$ consists of the $256$ $2$-torsion points in the Kummer;
    \item $D_{Y_{5}}(v)$ is the moduli space $\SU_C(2,\cO_C)$ of semistable rank $2$ vector bundles on $C$ with trivial determinant;
    \item $D_{Y_1}(v)=Qu_v$ is the unique quartic which is singular along $\SU_C(2,\cO_C)$.
\end{enumerate}
\end{conjecture}

We provide some evidence for this conjecture by checking that the Hilbert polynomials are the expected ones. 

\begin{prop} We have the following:
\begin{enumerate}[label=\roman*)]
    \item The Hilbert polynomial of $D_{Y_5}(v)$ is the same as that of  $\SU_C(2,\cO_C)$;
    \item the Hilbert polynomial of $D_{Y_{10}}(v)$ is the same as that of the Kummer fourfold, $8x^4+8$;
    \item $D_{Y_{14}}(v)$ consists of $256$ points.
\end{enumerate}
\end{prop}

\begin{proof}
The three results follow from the fact that we know a locally free resolution of $\cO_{D_{Y_\bullet}(v)}$, given in terms of homogeneous bundles on the quadric. Using the Borel-Weil-Bott Theorem and \cite{LiE}, 
we can deduce the result. 

More precisely, following \cite{bfmt2}, in order to obtain a locally free resolution of $\cO_{D_{Y_\bullet}(v)}$ it is sufficient to \emph{relativize} a locally free resolution of $\cO_{Y_\bullet}$. These resolutions for $\bullet=5, 10, 14$ were obtained in  \cite{Weyman_E8}. 

For instance, the resolution of  $\cO_{D_{Y_5}(v)}$ is 
$$
0\to \cO_{\QQ_{14}}(-10) \to \cE_{\omega_7}(-7) \to \cE_{\omega_3}(-7) \to \cE_{\omega_3}(-5) \to \cE_{\omega_8}(-4) \to \cO_{\QQ_{14}}\to \cO_{D_{Y_5}(v)} \to 0,
$$
where $\cE_\omega$ is the homogeneous vector bundle of highest weight $\omega$ on $\QQ_{14}$, and $\omega_1,\ldots,\omega_8$ are the fundamental weights of $D_8$ with Bourbaki's numeration. Our notation is such that 
the space of global sections of $\cE_\omega$ is $V_\omega$, the irreducible representation of $\Spin_{16}$ with highest weight $\omega$. More generally, for any $k\ge 0$,  $h^0(\QQ_{14},\cE_\omega(k))=\dim V_{\omega+k\omega_1}$  is given by the Weyl dimension formula as a polynomial $P_\omega(k)$. We then get the Hilbert polynomial of   $D_{Y_5}(v)$ as 
$$H_{D_{Y_5}(v)}(x)=P_0(x)-P_{\omega_8}(x-4)+P_{\omega_3}(x-5)-P_{\omega_3}(x-7)+P_{\omega_7}(x-7)-P_0(x-10).$$ 
Plugging-in the Weyl dimension formula, we deduce that 
    $$H_{D_{Y_5}(v)}(x) = 1 + \frac{233 x}{70} + \frac{1979 x^2}{420} + \frac{29041 x^3}{7560} + \frac{31 x^4}{15} + \frac{71 x^5}{90} + \frac{13 x^6}{60} + \frac{103 x^7}{2520} + \frac{x^8}{210} + \frac{x^9}{3780}. $$
    On the other hand, the Hilbert function, or polynomial, of 
    $\SU_C(2,\cO_C)$ is given by the Verlinde formula, namely for genus four \cite{bs}, 
    $$H_{\SU_C(2,\cO_C)}(x)=\left(\frac{x+2}{2}\right)^3\; \sum_{i=1}^{x+1}
    \sin^{-6}\left(\frac{i\pi}{x+2}\right).$$
    This is not clearly polynomial in $x$, so we compute the first ten values for $x=0,\ldots , 9$ and find $1, 16, 136, 800, 3611, 13328, 42048, 117072, 294525, 681472$. These values coincide with those obtained for the degree nine polynomial $H_{D_{Y_5}(v)}$, so we can conclude that the two polynomials coincide. 
\end{proof}

\begin{remark}
Following the previous conjecture, $Qu_v$ is a Heisenberg invariant quartic which is singular along $\SU_C(2,\cO_C)$. For a genus four curve with no vanishing theta-null, it was proved 
in \cite{OxburyPauly} that 
the linear system $|2\Theta|$ embeds the moduli space in $\PP^{15}$ in such a way that it is contained in the singular locus of a unique Heisenberg-invariant quartic. That such a quartic should exist is suggested by the first values of the Hilbert polynomial: $1, 16, \binom{16+1}{2}, \binom{16+2}{3}-16$; in particular the minimal degree equations are the sixteen derivatives of the quartic, and the moduli space is not contained in any quadric. Moreover there is a conjectural interpretation of this quartic as a moduli space $\cN_C$ of {\it Galois $\Spin_8$-bundles} associated to a triple cover $C\ra\PP^1 $ \cite{OxburyRamanan}. 

When $C$ has a vanishing theta-null, $\cN_C$ is expected to be a double cover of $\QQ^{14}$ branched over a quartic section. Such a conjecture appears in \cite{bsw}, and is motivated by the similar behavior in genus three. According to Conjecture \ref{conj_dodulispace}, $\PP(\fh)$ describes an even more degenerate situation, i.e. $C$ has a vanishing theta-null with a triple ramification: in this case as well we expect  $\SU_C(2,\cO_C)$ to be contained in a quadric (image of a double cover from $\cN_C$), and also in the singular locus of a unique Heisenberg-invariant quartic (section of the quadric!), i.e. exactly $Qu_v$.
\end{remark}

For $k=2$, the spinor bundle $\cS$ on $OG(2,16)$ is modeled on a half-spin representation of $\Spin_{12}$,
which admits a degree four invariant. We thus get a map
$$\Gamma_2: S^4\Delta \lra H^0(OG(2,16),S^4\cS)\lra H^0(OG(2,16),\cO(2))=S_{\langle 2,2\rangle}\CC^{16}.$$
So a generic element $v$ of $\Delta$ defines a quadric section $Q_v$ of $OG(2,16)$, which we will study more in detail in Section \ref{sec_quadric_spin}. As before, the fibers of the spinor bundle $\cS$ on $OG(2,16)$ are naturally isomorphic to the half-spin representation $\Delta_{12}$ of $\Spin_{12}$. This is a parabolic representation with a finite number of orbits, classified in \cite{KWE7}. Let us denote by $Z_i$ the unique codimension $i$ orbit closure in $\Delta_{12}$ for $i=1,7,16$. We get the following stratification:
$$ Z_{16}\subset Z_{7}\subset Z_1\subset \Delta_{12}. $$
Considering the associated orbital degeneracy loci
$$ D_{Z_i}(v)=\{ x\in OG(2,16)\mid v(x)\in Z_i\subset \Delta_{12}\}, $$
we get a stratification
$$ D_{Z_{16}}(v)\subset D_{Z_{7}}(v)\subset
D_{Z_1}(v)=Q_v\subset OG(2,16), $$
with $D_{Z_i}(v)$ of codimension $i$ for general $v$. Again, we believe in the following conjecture:
\begin{conjecture}
Let $v$ be a general element in $\Delta$.
\begin{enumerate}[label=\roman*)]
    \item $D_{Z_{16}}(v)$ is the (smooth) moduli space $\SU_C(2,\cO_C(p))$ of semistable rank $2$ vector bundles on $C$ with determinant equal to $\cO_C(p)$ for a certain point $p\in C$;
    \item $D_{Z_1}(v)=Q_v$ is the unique quadric which is singular along $\SU_C(2,\cO_C(p))$.
\end{enumerate}
\end{conjecture}

\begin{remark}
As in the previous remark, it would be interesting to guess, and possibly prove, a similar statement for the general genus four curve. Is the moduli space contained in the singular locus of a special quadric section of the Grassmannian $G(2,16)$? Recall that for a general genus 3 curve, there is already such a phenomenon \cite{coblequadric}: the moduli space of semistable rank two vector bundles with odd determinant over a curve $C$ of genus three embeds inside $G(2,8)$, and it is contained in a unique quadric as its singular locus for generic $C$.
\end{remark}

\subsection{The $\Spin_{16}$ Coble quadric and its equation}
\label{sec_quadric_spin}
We denote by 
$$ \gamma_2:=:\PP(\fc) \dashrightarrow \PP(\Theta_2) \subset \PP(S_{\langle 2,2\rangle}V) $$
the restriction of $\Gamma_2$ to the Cartan subspace $\fc$ of $\Delta$ appearing in Proposition \ref{cartanbasis}, sending a point to a Heisenberg invariant quadric section of the orthogonal Grassmanninan. Here  
$\Theta_2$ denotes the linear span of the image of $\gamma_2$. The map $\gamma_2$ is $W$-equivariant since $\Gamma_2$ is $\Spin_{16}$-equivariant; hence $\Theta_2$ is a representation for $W$. Our second aim is to identify $\Theta_2$; we will see that it is (once again!) a Macdonald representation. 
 
 Let us denote by $c_1,\dots,c_8$ the dual basis of $h_1,\dots,h_8$. Moreover, to avoid confusion, we will sometimes denote by $x_0,\dots,x_7,y_0,\dots,y_7$ the dual basis of $e_0,\dots,e_7,f_0,\dots,f_7$.
 
 \begin{lemma}
 The morphism $\Gamma_2$ factorizes as 
 $$S^4 \Delta \hookrightarrow{} S^2( S^2 \Delta) \xrightarrow{S^2 \eta} S^2( \bigwedge^4 V) \xrightarrow{\zeta} S_{\langle 2,2\rangle}V $$
 where $\eta:S^2 \Delta \to \bigwedge^4 V$ and $\zeta:S^2 (\bigwedge^4 V) \to S_{\langle 2,2\rangle}V$ are uniquely defined (up to scalar) by their  $\Spin_{16}$-equivariance.
 \end{lemma}

\begin{proof}
The uniqueness follows from Schur's lemma and the decomposition of $S^2 \Delta$ and $S^2(\bigwedge^4 V) $ into irreducible factors. From the fact that there exists a unique $\Spin_{16}$-equivariant morphism $S^4\Delta \lra S_{\langle 2,2\rangle}V$ (verified in the same way) we deduce that the composition of the above maps is equal to $\Gamma_2$ or the trivial map. As the computations that follow in this section will show, the composition above is not trivial.
\end{proof}

\medskip
Note that the morphism $\zeta$ can be obtained by first contracting by the quadratic form:
$$(a_1\wedge a_2\wedge a_3\wedge a_4,b_1\wedge b_2\wedge b_3\wedge b_4)\mapsto 
\langle a_1\wedge a_2,b_1\wedge b_2\rangle (a_3\wedge a_4)(b_3\wedge b_4)+\cdots ,$$
where $\langle , \rangle$ denotes as usual the autoduality induced by $q$. This yields a tensor in $S^2(\wedge^2V)$, that we can project to $S_{22}V$, and then to 
$S_{\langle 2,2\rangle}V$; the first projection amounts to restricting quadrics on 
$\PP(\wedge^2V^\vee)$, first to the Grassmannian $G(2,V^\vee)$ (that is, we mod out by the 
Pl\"ucker relations), and then to the orthogonal Grassmannian. 
In the sequel we will write formulas involving quadrics that will always
be considered as defined on the orthogonal Grassmannian. 

The restriction from the usual Grassmannian kills 
a copy of $S^2V$ inside $S^2(\wedge^2V)$, obtained as the image of the morphism
$$v^2\mapsto \sum_i (v\wedge e_i)(v\wedge f_i), \qquad v\in V.$$
The equivariant complement $S_{\langle 2,2\rangle}V$ embeds in $S_{ 2,2}V$ as the kernel of the contraction map to $S^2V$ induced by the quadratic form (as well as $S_{ 2,2}V$ embeds in $S^2(\wedge^2V)$ as the kernel of the natural morphism to $\wedge^4V$). 

Alternatively, these extra relations can be deduced from the isotropy condition as follows.
If a plane $P$ in $V$ has Pl\"ucker coordinates given by 
 $$\sum_{ij}\Big( \frac{1}{2}(e_ie_j)e_i\wedge e_j+ (e_if_j)e_i\wedge f_j
 + 
 \frac{1}{2}(f_if_j)f_i\wedge f_j\Big),$$
 we can deduce $P$ itself as being generated by the vectors obtained by contracting 
 this tensor with a generic linear form   $\phi=\sum_i(x_ie_i^*+y_if_i^*)$. This contraction yields
 $$\sum_{ij}\Big( (e_ie_j)x_ie_j+ (e_if_j)(x_if_j-y_je_i)+ 
 (f_if_j)y_if_j\Big),$$
 which is isotropic for any $\phi$ if and only if the following relations hold:
$$\sum_k\Big((e_ie_k)(e_jf_k)+(e_je_k)(e_if_k)\Big)=
\sum_k\Big((f_if_k)(f_je_k)+(f_jf_k)(f_ie_k)\Big)=0,$$
$$\sum_k\Big((e_ie_k)(f_jf_k)-(e_if_k)(e_kf_j)\Big)=0.$$

\medskip
As a consequence, consider the sums
$$I_0 := \sum_i(e_if_i)^2, \qquad I_1=\sum_{i<j}(e_if_i)(e_jf_j),  
\qquad I_2=\sum_{i<j}(e_ie_j)(f_if_j),  \qquad I_3=\sum_{i<j}(e_if_j)(e_jf_i).$$ 

From the latter relations for $i=j$, plus the usual Pl\"ucker identities, we deduce that
\begin{equation}\label{sommes}
I_2=\frac{1}{2}I_1+\frac{1}{4}I_0, \qquad I_3=\frac{1}{2}I_1-\frac{1}{4}I_0.
\end{equation}

\subsection*{The morphism $\eta$}
We will first need to understand $\eta$ and its actions on monomials $h_ih_j$. 
Notice that each $h_i$ is obtained as the sum of two pure spinors associated to complementary isotropic subspaces. Let us denote by $E_i$ and $E_i'$ these two subspaces. For instance,  
$$E_1=\langle e_0,e_1,e_2,e_3,f_4,f_5,f_6,f_7 \rangle \quad \mathrm{and} \quad E'_1=\langle f_0,f_1,f_2,f_3,e_4,e_5,e_6,e_7 \rangle,$$ 
while $E_8=E$ and $E'_8=F$.
For two distinct indices $k,l$ we denote 
$$E_{kl}:=E_k\cap E_l,\quad  E'_{kl}:=E_k\cap E_l',\quad  E''_{kl}:=E_k'\cap E_l, \quad E_{kl}''':=E_k'\cap E_l'.$$ 
Notice that these four subspaces are always four-dimensional. 
They consist of two pairs of dual spaces (in the sense that they are put in duality by the restriction of the quadratic form), and we will denote them in a such a way that these two pairs are $(E_{ab}, E_{ab}''$) and
$(E_{ab}', E_{ab}''')$. 
If $e_i,e_j,e_k,e_l$ (resp. $f_i,f_j,f_k,f_l$ or $e_i,e_j,f_k,f_l$) is a basis of $E_{ab}^\bullet$ with $i<j<k<l$ (resp. $i<j<k<l$ or $i<j$, $ k<l$), let $s_{E_{ab}^\bullet}=e_{ijkl}\in \bigwedge^4 V$ (resp. $f_{ijkl}$ or $e_{ij}\wedge f_{kl}$). We will denote 
$$I_1^k:=\{i<j\mid  (e_i,e_j\in E_k) \mbox{ or } (e_i,e_j\in E_k')\}, \qquad I_2^k:=\{i<j\mid (i,j)\notin I_1^k\}.$$ 
Finally, it will be convenient to introduce the notation $\epsilon_k(i,j):=1$ if $(i,j)\in I_1^k$, $\epsilon_k(i,j):=-1$ otherwise. 

\begin{prop}\label{eta}
We have the following formulas:
$$\eta(h_k^2)  =  \frac{1}{2}\sum_{i<j}\epsilon_k(i,j) e_{ij}\wedge f_{ij} ,\qquad 
\eta(h_kh_l)= s_{E_{kl}}+s_{E_{kl}'}+s_{E_{kl}''}+s_{E_{kl}'''}.$$
%$$\begin{array}{rcl}
%\eta(h_k^2) & = & \frac{1}{2}(\sum_{(i,j)\in I_1^k} e_{ij}\wedge f_{ij} - \sum_{(i,j)\in I_2^k} e_{ij}\wedge f_{ij} ),\\ 
%\eta(h_kh_l)&=& s_{E_{kl}}+s_{E_{kl}'}+s_{E_{kl}''}+s_{E_{kl}'''}.
%\end{array}$$
\end{prop}

For instance, we get 
$$\eta(h_1h_2)=e_{01}f_{67}+e_{23}f_{45}+e_{45}f_{23}+e_{67}f_{45}, \qquad 
\eta(h_1h_8)=e_{0123}+e_{4567}+f_{0123}+f_{4567}.$$ 

\begin{proof}
In terms of a basis $(a_i)$ of $V_{16}$, with dual basis $(b_i)$, we have 
$$\eta(\delta_1\delta_2)=\frac{1}{2}\sum_{i<j<k<l}\Big(\langle \delta_1, (a_i\wedge a_j\wedge a_k\wedge a_l).\delta_2\rangle+\langle \delta_2, (a_i\wedge a_j\wedge a_k\wedge a_l).\delta_1\rangle\Big) \; b_i\wedge b_j\wedge b_k\wedge b_l.$$
Using the basis $(e_i,f_j)_{i,j}$ and plugging in $\delta_1$ and $\delta_2$ the elements $h_1,\dots,h_8$, one obtains the statement.
\end{proof}

\smallskip
We can now state our final results for $\gamma_2(h)=\sum_\alpha c_\alpha P_\alpha$, where $P_\alpha$ is the quadric on the orthogonal Grasmannian
that is the coefficient of the monomial $c_\alpha$. In fact the 
$Q_\alpha$ will not be independent, and it is more satisfactory 
to write 
$$\gamma_2(h)=\sum_\beta \kappa_\beta(c) Q_\beta,$$
where the $Q_\beta$'s are independent and their coefficients $\kappa_\beta(c)$ are some quartic polynomials. The following result is proved in Section \ref{app_C}, where one can find the necessary definitions of $Quad$, $F_{ijkl}^{(u)}$, $F_{ij|kl}^{(1_u)}$ and $F_{ij|kl}^{(2_u)}$, $R_a$ and $R_a'$.

\begin{prop}
\label{prop_formula_quadric_spin16}
The equation $\gamma_2(h)$ of the Coble quadric is given by the sum of $84$ terms:
\begin{itemize}
    \item $14$ terms parametrized by the quadruples $\beta=(ijkl)$ from $Quad$, for which 
    $$\kappa_\beta(c)=c_ic_jc_kc_l, \qquad Q_\beta = 24\sum_{u=1}^8 \det(F_{ijkl}^{(u)})^2.$$ 
    \item $56$ terms parametrized by $\beta$ consisting of the following data: a pair $ij$, and two of the three quadruples in $Quad$ that contain it, say  $ijkl$ and  $ijmn$, but not $ijop$; then
    $$\kappa_\beta(c) = c_ic_j(c_k^2+c_l^2-c_o^2-c_p^2), \qquad 
    Q_\beta= -6\sum_{u=1}^4 \det(F_{ij|kl}^{(1_u)}) \cdot \det(F_{ij|kl}^{(2_u)}).$$ 
    \item $7$ terms parametrized by $\beta=a$ a point in the (dual) Fano plane, with 
    $$\kappa_\beta(c) = c_a^4-c_a^2\sum_{1\le b\ne a\le 7}c_b^2+\sum c_u^2c_v^2, \quad Q_\beta = 6R_a-3I_1+\frac{3}{2}I_0,$$ 
    where the sum is over the three lines $(auv)$ through $a$.
    \item $7$ terms again parametrized by  $\beta=a$ a point in the (dual) Fano plane, with 
     $$\kappa_\beta(c) =
    c_a^4-c_a^2\sum_{b\ne a}c_b^2 - 
    (c_8^4-c_8^2\sum_{b\ne 8} c_b^2), \quad Q_\beta = 6R_a+12R_a'-5I_1+\frac{1}{2}I_0-2I_2.$$
\end{itemize}
\end{prop}

\begin{proof}
The result follows at once from Lemma \ref{lem_coef_c_ijkl}, \ref{lem_coef_c_iijk} and Corollary \ref{cor_coef_c_iijj}.
\end{proof}

A natural way to interpret $\gamma_2$, since it is equivariant under the 
Weyl group action, is to understand it as providing an isomorphism 
(sending each $\kappa_\beta(c)$ to $Q_\beta$) between 
an invariant  space $\Theta_2$ of quartic polynomials on the Cartan subspace, and a space of quadrics on the orthogonal Grassmannian.
We have enough information to identify this representation of the Weyl group of $E_8$.

\begin{prop}\label{84}
The space $\Theta_2$  is an irreducible Macdonald representation of dimension $84$. 
\end{prop}

\begin{proof}
Since $\Theta_2$ contains some quartics of the form $c_ic_jc_kc_l$ for 
$i,j,k,l$ distinct, it has to contain a nontrivial Macdonald representation. Since $\Theta_2$ is generated by the $84$ quartics $\kappa_\beta$, the following Lemma implies that its dimension is 
in fact equal to $84$, the $\kappa_\beta$ being a basis, and that $\Theta_2$ is equal to this Macdonald representation. \end{proof}
 
\begin{lemma}
\label{lem_dim_84or210}
Consider the Macdonald representations of the Weyl group of type $E_8$ generated by the orbit of some monomial $c_ic_jc_kc_l$ for $i,j,k,l$ pairwise distinct. There are only two such representations and their dimensions are $84$ and $210$.
\end{lemma}

\begin{proof}
That there are at most two such representations is a consequence of the fact that there are only two orbits (under the action of the Weyl group) of root subsystems of type $4A_1$ inside $E_8$ (Lemma \ref{lem_subsystems_E8}). An explicit computation with \cite{Macaulay2} shows that these orbits generate spaces of quartics of respective dimensions $84$ and $210$ (see attached file {\bf gopel\_e8}).
\end{proof}

Note that the dimension of $S^4\fh^\vee$ is $330=84+210+36$, and that $36$ is the dimension of $S^2\fh^\vee$, which embeds inside $S^4\fh^\vee$ by multiplication with the Cartan-Killing form. 

\subsection*{The G\"opel variety of the quadric}
Let us denote by $\cG_2$ the 
G\"opel variety of $\Spin_{16}$, i.e. the image via $\gamma_2$ of $\PP(\fc)$ inside $\PP(\Theta_2)$.

\begin{theorem}
The rational map $\gamma_2:\PP(\fc) \dasharrow \cG_2 \subset \PP(\Theta_2)\cong \PP^{83}$ is a birational morphism.
\end{theorem}

\begin{proof}
The computations above gives an explicit expression of $\gamma_2$ in terms of coordinates; however, for our computations it was easier to use the expression given by the program in the file {\bf gopel\_e8}. One can plug this in \cite{Macaulay2} and compute the degree of this rational map, which turns out to be equal to one, and its base locus, which turns out to be empty. 
\end{proof}

\subsection{The $\Spin_{16}$ Coble quartic}
\label{sec_quartic_spin}
Recall that we have also defined a Coble type hypersurface via the $\Spin_{16}$-equivariant map
$$\Gamma_1: S^8\Delta \lra S^{\langle 4\rangle}V.$$ Let us give some results about the map $$ \gamma_1:= \Gamma_1 |_{\fc}:\PP(\fc) \dashrightarrow \PP(\Theta_1) \subset \PP(S^{\langle 4\rangle}V) ,$$
 where $\Theta_1$ is defined as the linear span of the image of $\gamma_1$. The map $\gamma_1$ is $W$-equivariant since $\Gamma_1$ is $\Spin_{16}$-equivariant, hence   $\Theta_1$ is a representation for $W$. As the previous section has shown, explicitly computing the equations of the Coble quartic in this case is quite challenging. Since we are interested in the algebro-geometric consequences, we just give some structural results, leaving the rest to future mathematicians/programmers. 
 
 \begin{lemma}
 The map $\Gamma_1$ factorizes through $\Gamma_2$ and a morphism $\eta_2: S^2 (S_{\langle 2,2 \rangle}V) \to S^{\langle 4 \rangle}V$.
 \end{lemma}

\begin{proof}
The result follows from a computation with \cite{LiE}. Indeed, one checks that there exists a unique (up to scalar) equivariant morphism $S^8 \Delta \to S^{\langle 4 \rangle} V$, as well as a unique equivariant morphism $S^2(S_{\langle 2,2 \rangle}V) \to S^{\langle 4 \rangle}
V$. This implies the claim.
\end{proof}

As a consequence, it is  in principle possible to derive the equation of the $\Spin_{16}$ Coble quartic from Proposition \ref{prop_formula_quadric_spin16}. Indeed the morphism $\eta_2: S^2(S_{\langle 2,2 \rangle}V) \to S^{\langle 4 \rangle}$ is easy to define in terms of contraction with the invariant quadratic form, by mapping 
\begin{equation}\label{eta2}
(a_{ij}a_{kl})(a_{mn}a_{op})\mapsto \langle a_i,a_m\rangle \langle a_k,a_o \rangle a_ja_la_na_o  + \cdots 
\end{equation}

Of course we expect to end up once again with a Macdonald representation. 
Let us note the following:

\begin{lemma}
The unique Macdonald $W$-representation $\Theta_1'$ generated by eight orthogonal roots in $E_8$ has dimension $50$.
\end{lemma}

\begin{proof}
The uniqueness follows from Lemma \ref{lem_subsystems_E8}. The dimension statement is a computation with \cite{Macaulay2} (see attached file {\bf gopel\_e8}).
\end{proof}

\begin{lemma}
The space $\Theta_1$ contains the Macdonald $W$-representation $\Theta_1'$.
\end{lemma}

\begin{proof}
We claim that the coefficient of $e_1^4\in S^{\langle 4 \rangle}V$ in $\gamma_1(h)$ is a non-zero multiple of $c_1c_2\cdots c_8$; this implies the claim. By  formula (\ref{eta2}), the only nonzero contributions to $e_1^4$ are obtained by multiplying two terms among the first $14$ quadrics from Proposition \ref{prop_formula_quadric_spin16}. More precisely, they can only be  obtained by multiplying quadrics corresponding to complementary coefficients $c_{ijkl}$ and $c_{mnop}$, where $\{ i,j,k,l,m,n,o,p \}=\{1,\dots,8\}$. One checks that each 
such product contributes by a positive integer, and the result follows
in $S^4V$. This has to remain true in $S^{\langle 4 \rangle}V$ since we mod out by products of the invariant quadric $q=\sum_ie_if_i$, which cannot
affect the coefficient of $e_1^4$. 
\end{proof}

\begin{theorem}\label{50}
We have an isomorphism $\Theta_1 \simeq \Theta_1'$. 
\end{theorem}

\begin{proof} Let $H$ denote the genus 4 (and level 2) Heisenberg group. By the explicit action of $H$, described in Section \ref{heis}, one sees that $H\subset SO(16)$ since it preserves the natural quadric hypersurface. The map $\Gamma_1$ is $SO(16)$-equivariant, hence it is also $H$-equivariant. By Proposition \ref{hheisinv}, $\fc$ is the $H$-invariant subspace of $\Delta$, hence its image via $\Gamma_1$ is contained in the space of Heisenberg invariant quartics on the 14-dimensional smooth quadric. Now, let $H'$ denote the \it continuous \rm Heisenberg group, defined by  the extension:
$$ 1 \to \CC^* \to H' \to (\ZZ_2)^8 \to 1. $$

Let $N(H')$ denote the normalizer of $H'$. By \cite[Ex. 6.14]{BL},  $N(H')/H'\cong \Sp(2,8)$. According to  \cite{vGdP}, the space of Heisenberg invariant quartics in $\PP^{15}$ is a $51$ dimensional representation of $\Sp(2,8)$. Recall also that $\Sp(2,8)$ contains  $O^+_8(\ZZ_2)$, the stabilizer subgroup of the quadric. By restriction to $O^+_8(\ZZ_2)$, the $51$-dimensional $\Sp(2,8)$-representation  decomposes into a $50$ dimensional representation plus the one dimensional trivial representation of $W$, generated by the square of the invariant quadric. Of course the latter quartic is killed by restriction to the quadric, and we can identify the $50$-dimensional Macdonald representation $\Theta_1'$ with the space of Heisenberg invariant quartic sections  of $Q$, hence with $\Theta_1.$
\end{proof}

\subsection*{The G\"opel variety of the quartic}
Let us denote by $\cG_1$ the \emph{Gopel} variety, i.e. the image via $\gamma_1$ of $\PP(\fc)$ inside $\PP(\Theta_1)$.

\begin{theorem}
The rational map $\gamma_1:\PP(\fc) \dasharrow \cG_2 \dasharrow \cG_1 \subset \PP(\Theta_1)$ is  birational. 
\end{theorem}

\begin{proof}
We used the expression given by our program {\bf gopel\_e8} to find the equations of the composition of the morphism $\gamma_1$ with the projection $\Theta_1\to \Theta_1'$. One can plug this in \cite{Macaulay2} and compute the degree of the composition; it  turns out to be equal to $1$.
\end{proof}

Again according to \cite{gs}, the base locus of $\gamma_1$ has dimension one and degree $8960$. Notice that there are $1120$ subsystems of type $A_2$ in the root system $E_8$, and that each subsystem defines a line in the projectivized Cartan algebra. We get thus 1120 lines, which are actually the orbit of the line $\ell_0$ defined by the homogeneous ideal $I_{\ell_0}=(c_7-c_4,c_1+c_4,c_2,c_3,c_5,c_6)$, under the action of the Weyl group. According to  \cite{gs}, for each such line $\ell$, the ideal of the base locus contains the homogeneous ideal $I_\ell^2$, which has degree $7$. However, in the base locus we should get a subscheme of degree 8, rather than 7, supported at each line $\ell$. This point remains to be elucidated. 

\begin{remark}
\label{rmk_cartan_to_eight_points}
As in \cite[Equation 4.4]{rsss}, we conjecture that one can construct a $W(E_8)$-equivariant birational morphism
\begin{equation}
f: \PP(\fh)\cong\PP^7 \to (\PP^2)^8/\hspace{-1mm}/\PGL(3)
\end{equation}
defined as
\begin{equation}\label{map8}
\begin{pmatrix}
1 & 1 & 1 & 1 & 1 & 1 & 1 & 1 \\
d_1 & d_2 & d_3 & d_4 & d_5 & d_6 & d_7 & d_8 \\
d_1^3 & d_2^3 & d_3^3 & d_4^3 & d_5^3 & d_6^3 & d_7^3 & d_8^3 \\
\end{pmatrix}
\end{equation}
whose image is the 7 dimensional subvariety given by $8$-tuplets lying on a cuspidal cubic ($zx^2-y^3=0$). Then the map given by degree 8 polynomials obtained via the Macdonald representation should be the same as the composition of $f$ with the $49$-dimensional linear system given by the Göpel functions $\PP^7\to \PP^{49}$. In the $g=3$ case the fiber of the corresponding map had degree 24 since in each net there were 24 cuspidal cubics. In this case it seems that the degree should be just one, since for $8$-general points on a cuspidal cubic, the pencil of cubics containing the $8$ points should contain just one cuspidal cubic.
\end{remark}

 \section{Appendix A: Heisenberg groups}
 \label{sec_heis_group}
 
 Let us assume we are given a $\ZZ_m$-graded simple Lie algebra $\fg=\fg_0\oplus \fg_1\oplus \cdots\oplus \fg_{m-1}$,  defined by an automorphism $\theta$ of $\fg$ generating $\ZZ_m$, and  a Cartan subspace $\fc\subset \fg_1$. Let $G_0$ and $G$ be simply connected groups whose Lie algebras are $\fg_0$ and $\fg$; the image of 
 $G_0$ in $G$ is the subgroup $G^\theta$ of $\theta$-invariant elements of $G$ \cite{vinberg}. We will use the following notation:
 \begin{enumerate}
 \item  $Z\subset G$ is the center of $G$, or the kernel of its action on $\fg$,
 \item $Z_0\subset G_0$ is the center of $G_0$, or the kernel of its action on $\fg_0$,
 \item $Z_0^\fg\subset Z_0$ is the kernel of the action of $G_0$ on $\fg$, 
\item  $(Z_0^\fg)^\theta\subset Z$ is the image of $Z_0^\fg$ in $G$, 
\item  $K_0 \subset Z_0^\fg$ is the kernel of its map to (the center $Z$ of) $G$.
 \end{enumerate}
 So by definition there is an exact sequence 
 $$ 1 \to K_0 \to Z_0^\fg \to (Z_0^\fg)^\theta \to 1. $$
 
 We will use these groups to analyze our main character, which is not the 
 centralizer $Z(\fc)$ of $\fc$ in $G$.

 \begin{definition}
 The \emph{Heisenberg group} $Z^+(\fc)$ of $(\fg,\theta)$ is  the centralizer of $\fc$  in $G_0$. 
 \end{definition}

Let $r$ be the rank of $\fg$ and $r_1$ the dimension of $\fc$. The smallest algebraic subalgebra of $\fg$ containing  $\fc \subset \fg_1$ is a Lie subalgebra $\ft$ of commuting semisimple elements of dimension $ r_1\varphi(m)$, where $\varphi$ is Euler’s function. More precisely, one shows that $\ft = \bigoplus_{(i,m)=1} \ft_i$, where $\ft_i = \ft \cap \fg_i$, $\ft_1=\fc$, and each term in this decomposition  has dimension $r_1$. The torus $T$ whose Lie algebra is $\ft$ is also the closure of $\exp(\fc)$, where $\exp : \fg \to G$ is the exponential map. In all our examples we will assume that $T$ is a maximal torus of $G$, or equivalently that 
$r=r_1\varphi(m)$  (i.e. that $\theta$-$\corank(\fg)=0$ as in \cite{bened_maniv_discriminants}).
 
 \begin{prop}
 \label{prop_heisenberg_235}
 If  $m$ is prime and $r=r_1\varphi(m)$, then $Z(\fc)\cong T$ and there is an  exact sequence of groups
 $$ 1 \to K_0 \to Z^+(\fc) \to (\ZZ_m)^{r_1} \to 1. $$
 \end{prop}
 
 \begin{proof}
 An element in $Z(\fc)$ acts as the identity on $\exp(\fc)$, hence also on its closure $T$. Since $T$ is a Cartan subalgebra it is its own centralizer, hence 
 $Z(\fc)= T$. 
 The image of $Z^+(\fc)$ in $G$ is $Z(\fc) \cap G^\theta$. Moreover $K_0$ is contained in  $Z^+(\fc)$, hence   an exact sequence of groups
 $$1 \to K_0 \to Z^+(\fc) \to Z(\fc) \cap G^\theta \to 1. $$
 So we need to show that $T\cap G^\theta=(\ZZ_m)^{r_1}$.  There is a primitive $m$th-root  $\zeta$ of unity such that the automorphism $\theta$ acts on $\ft_i$ by multiplication by $\zeta^i$. 
 This implies that one can find a basis $\{\alpha_{i,j}\}_{i,j}$ of $\ft$, where $1 \leq i \leq m-1$ and $1 \leq j \leq r_1$,  on which $\theta$ acts as follows:
  $$\theta(\alpha_{i,j})=\alpha_{i+1,j} \; \mathrm{for}\; i=1,\dots,m-2, \qquad 
  \theta(\alpha_{m-1,j})=-\sum_i \alpha_{i,j}.$$ 
  Taking exponentials  one deduces the action of $\theta$ on $(t_{1,1},\dots,t_{m-1,1}, \dots, t_{1,r_1},\dots,t_{m-1,r_1}) \in T$:  $$\theta(t_{1,j})=(t_{m-1,j})^{-1}, \qquad  \theta(t_{i,j})=t_{i-1,j}(t_{m-1,j})^{-1} \; \mathrm{for} \; i=2,\dots,m-1.$$ 
  It is straightforward to check that $\theta(t_{1,1},\dots,t_{m-1,r_1})=(t_{1,1},\dots,t_{m-1,r_1})$ if and only if $(t_{1,j},\dots,t_{m-1,j})=(\zeta^{i_j},\zeta^{2i_j},\dots,\zeta^{(m-1)i_j})$ for all $j$ and for integers $i_1,\dots i_{r_1}$, so that the set of such $\theta$-invariants elements is isomorphic to $(\ZZ_m)^{r_1}$.
 \end{proof}
 
 In the previous proof, when $m$ is not prime, one can still say something. We will be interested in only one graded Lie algebra with $m$ not prime, specifically $m=4$; we will treat this case separately below. 
 
 \subsubsection{Explicit cases of interest}
 
 Let us denote by $V_n$ a complex vector space of dimension $n$, by $\Delta_{2n}$ the half-spin representation of the group $\Spin_{2n}$, by $A_1,A_2,A_3$ three $3$-dimensional vector spaces.
 
 \begin{description}
          \item[$S^{\langle 2 \rangle}\CC^{2g+2}$] Here $\fg=\fsl_{2g+2}$, $m=2$, $r_1=2g+1$, $\fg_0=\fso_{2g+2}$, $\fg_1=S^{\langle 2 \rangle}\CC^{2g+2}$. The group $Z_0$ is isomorphic to $\ZZ_4$ when $g+1$ is odd and $\ZZ_2\times \ZZ_2$ when $g+1$ is even, and it coincides with $Z_0^\fg$. The group $(Z_0^\fg)^\theta$ is given by $\pm \id$, so $K_0\cong \ZZ_2$ is the fundamental group of $\SO_{2g+2}$. We deduce two equivalent exact sequences defining $Z^+(\fc)$:
     $$ 1 \to K_0 \to Z^+(\fc) \to (\ZZ_2)^{2g+1} \to 1, $$
     $$ 1 \to Z_0 \to Z^+(\fc) \to (\ZZ_2)^{2g} \to 1. $$
     \item[$\wedge^4 V_8$] Here $\fg=\fe_7$, $m=2$, $r_1=7$, $\fg_0=\fsl_{8}$, $\fg_1=\wedge^4 V_8$. The group $Z_0$ is isomorphic to $\ZZ_8$, while $Z_0^\fg =2Z_0\cong \ZZ_4$. The group $(Z_0^\fg)^\theta=Z$ is given by $\pm \id$, so $K_0\cong \ZZ_2$. We deduce the exact sequence defining $Z^+(\fc)$:
     $$ 1 \to Z_0^\fg \to Z^+(\fc) \to (\ZZ_2)^{6} \to 1. $$
     
     \item[$\Delta_{16}$] Here $\fg=\fe_8$, $m=2$, $r_1=8$, $\fg_0=\fso_{16}$, $\fg_1=\Delta_{16}$. The group $Z_0$ is isomorphic to $\ZZ_2\times \ZZ_2$, and it coincides with $Z_0^\fg $. The group $Z$ is trivial so $K_0= \ZZ_2\times \ZZ_2$. We deduce the exact sequence defining $Z^+(\fc)$:
     $$ 1 \to Z_0 \to Z^+(\fc) \to (\ZZ_2)^{8} \to 1. $$
     
     \item[$\wedge^3 V_9$] Here $\fg=\fe_8$, $m=3$, $r_1=4$, $\fg_0=\fsl_{9}$, $\fg_1=\wedge^3 V_9$. The group $Z_0$ is isomorphic to $\ZZ_9$, while $Z_0^\fg =3Z_0\cong \ZZ_3$. The group $Z$ is trivial so $K_0= \ZZ_3$. We deduce the exact sequence defining $Z^+(\fc)$:
     $$ 1 \to Z_0^\fg \to Z^+(\fc) \to (\ZZ_3)^{4} \to 1. $$
     
     \item[$A_1\otimes A_2\otimes A_3$] Here $\fg=\fe_6$, $m=3$, $r_1=3$, $\fg_0=\fsl_{3}\times \fsl_3 \times \fsl_3$, $\fg_1=A_1\otimes A_2\otimes A_3$. The group $Z_0$ is isomorphic to $(\ZZ_3)^3$, while $Z_0^\fg \cong (\ZZ_3)^2$. 
     Moreover $Z\cong \ZZ_3$. Recall Cartan's description of the simply connected $E_6$ as the stabilizer of a cubic form $K$ in $27$ variables. This form can be made explicit by choosing three vector spaces $A,B,C$ of dimension $3$, letting $V=\Hom(A,B)\oplus \Hom(B,C)\oplus \Hom(C,A)$ and defining $K(\gamma,\beta,\alpha)=\det(\alpha)+\det(\beta)+\det(\gamma)-\mathrm{trace} (\alpha\circ\beta\circ\gamma)$. If $\zeta^3=1$, the triple $(\id_A, \zeta \id_B, \zeta^2 \id_C)\in \SL(A)\times \SL(B)\times \SL(C)$ is mapped to $\zeta \id_{V}$, 
     and we conclude that $(Z_0^\fg)^\theta=Z$, so $K_0\simeq \ZZ_3$.  We deduce the exact sequence defining $Z^+(\fc)$:
     $$ 1 \to \ZZ_3 \to Z^+(\fc) \to (\ZZ_3)^{3} \to 1, $$ or equivalently (using the exact sequence in which $Z_0^\fg$ appears):
     $$ 1 \to Z_0^\fg \to Z^+(\fc) \to (\ZZ_3)^{2} \to 1. $$
     
     \item[$A_5^\vee\otimes \wedge^2 B_5^\vee$] Here $\fg=\fe_8$, $m=5$, $r_1=2$, $\fg_0=\fsl_{5}\times \fsl_5 $, $\fg_1=A_5^\vee\otimes \wedge^2 B_5^\vee$. The group $Z_0$ is isomorphic to $(\ZZ_5)^2$, while $Z_0^\fg \cong \ZZ_5$ (which, inside $Z_0 \cong \ZZ_5^2 $, is given by couples $(\xi,\xi^2)$, where $\xi$ is a $5$th root of unity). We have that $Z$ is trivial and $K_0\cong  \ZZ_5$. We deduce the exact sequence defining $Z^+(\fc)$:
     $$ 1 \to \ZZ_5 \to Z^+(\fc) \to (\ZZ_5)^{2} \to 1. $$
     
     \item[$V_4\otimes \Delta_{10}$] Here $\fg=\fe_8$, $m=4$, $r_1=4$, $\fg_0=\fsl_{4}\otimes \fso_{10}$, $\fg_1=V_4\otimes \Delta_{10}$. Beware that in this case $m=4$ is not prime. The group $Z_0$ is isomorphic to $\ZZ_4\times \ZZ_4$, while $Z_0^\fg =\ZZ_4$. The group $Z$ is trivial so $K_0= Z_0^\fg$. We deduce the exact sequence defining $Z^+(\fc)$:
     $$ 1 \to Z_0^\fg \to Z^+(\fc) \to T\cap G^\theta \to 1, $$
     where $T$ and $G^\theta$ are as in Proposition \ref{prop_heisenberg_235}. The action of $\theta$ on a basis $\alpha_{1,j},\alpha_{2,j}$, $j=1,2,3,4$ of $\ft$ is given by $\theta(\alpha_{1,j})=\alpha_{2,j}$, $\theta(\alpha_{2,j})=-\alpha_{1,j}$. The corresponding action on a $\theta$-invariant torus $T$ is given by the formula: $\theta((t_{1,j},t_{2,j}))=(t_{2,j}^{-1},t_{1,j})$, and the $\theta$-fixed elements form a copy of $(\ZZ_2)^4$.
 \end{description}

 \begin{remark}
 \label{rem_thorne_heis}
 As we have seen above, in all cases, the exact sequence
 $$
 1 \to Z_0^\fg \to Z^+(\fc) \to Z^+(\fc)/Z_0^\fg \to 1
 $$
 is such that the quotient $Z^+(\fc)/Z_0^\fg$ is isomorphic to $\ZZ_m^{2g}$, for a certain integer $g$ (except when $m=4$, in which case we have $\ZZ_{m/2}^{2g}$). Throughout the paper we have seen that we can associate to each of the examples above a family of curves of genus $g$. Therefore we can identify abstractly the group $Z^+(\fc)/Z_0^\fg$ with the group of $m$-torsion (or $m/2$ if $m=4$) points of the Jacobian of the curve of genus $g$. In the case $m=2$ this was already discovered by Thorne (see \cite[Theorem 1.3]{Thorne}).
 \end{remark}
 
 \begin{remark}
It is not clear a priori whether the above definition gives Heisenberg groups in the sense of \cite{ThetaIII}. In Proposition \ref{prop_heisenberg_235} we have shown that $Z^+(\fc)$ is indeed an extension of an abelian group by $Z_0^\fg$, but $Z_0^\fg$ is not always cyclic and, more importantly, we have not shown that this extension is always non-split. When $m=2$ the fact that we get indeed Heisenberg groups in the sense of \cite{ThetaIII} is \cite[Theorem 1.3]{Thorne} (see also Remark \ref{rem_thorne_heis}). In order to check that we get indeed a non-trivial splitting: for the cases of $\wedge^3 V_9$ and $A_5^\vee\otimes \wedge^2 B_5^\vee$ we refer to \cite{gs}; for $A_1\otimes A_2\otimes A_3$ we refer to the discussion following Lemma \ref{cartan1}, while we live aside the case $V_4\otimes \Delta_{10}$ in this paper.
\end{remark}
 
 \section{Appendix B: Uniqueness of hyperelliptic Coble quadrics}\label{uniq}
 
 In this Appendix we prove Proposition \ref{unicity}. Because of Lemma \ref{resolution}, uniqueness follows from the conditions that 
$$H^i(OG(k,2n), S_{k1^i}(S^2U)(2))=0 \quad\forall i>0,$$
and $H^0(OG(k,2n),S_k(S^2U)(2))=\CC$. There is a classical plethysm formula 
$$S_k(S^2U)=\bigoplus_{|\alpha|=k}S_{2\alpha}U,$$
which easily implies the latter assertion; indeed, for $\alpha =(1,...,1)$ we get 
$S_{2\alpha}U(2)=\mathcal O$, and otherwise $S_{2\alpha}U(2)$ has no section by the 
Borel-Weil theorem. So we focus on the first assertion. We need to prove that for any 
component $S_\beta U$ of $S_{k1^i}(S^2U)$, with $i>0$, we have 
$$H^i(OG(k,2n), S_\beta U (2))=0.$$
Such cohomology groups are governed by Bott's theorem, which in our situation applies as follows. 
In terms of the fundamental weights of $\fso_{2n}$, the highest weight of  $S_\beta U (2)$ is
$$\lambda(\beta)=(\beta_{k-1}-\beta_k)\omega_1+(\beta_{k-2}-\beta_{k-1})\omega_2+\cdots 
+(\beta_{1}-\beta_2)\omega_{k-1}+(2-\beta_1)\omega_k.$$
Then we need to evaluate $\lambda(\beta)+\rho$ on the positive roots $\alpha$ that have positive coefficient on $\alpha_k$. There are two types of roots, that express in terms of simple 
roots either as a sum of 
distinct roots, or with a coefficient two. So suppose first that the coefficient on $\alpha_k$ 
is one, and  that the 
first nonzero coefficient is on $\alpha_{k-\ell+1}$ for $1\le \ell\le k$. For $\alpha=\alpha_{k-\ell+1}+\cdots +\alpha_k$ we get 
$$\langle \lambda(\beta)+\rho,\alpha\rangle = \ell-\beta_\ell+2.$$
For the same coefficients on the first $k$ roots, the maximal possible root is 
$\alpha_+=\alpha_{k-\ell+1}+\cdots +\alpha_k+2\alpha_{k+1}+\cdots +2\alpha_{n-2}+\alpha_{n-1}+\alpha_n$,
for which 
$$\langle \lambda(\beta)+\rho,\alpha_+\rangle =\ell-\beta_\ell+2(n-k).$$
Between these two extremes, there are a total of $2(n-k)$ roots, covering all the values
between $\ell-\beta_\ell+2$ and $\ell-\beta_\ell+2(n-k)$.

Now suppose that the coefficient on $\alpha_k$ 
is two, which means that $\alpha$ is of the form $\alpha_{k-a+1}+\cdots +\alpha_{k-b}+2\alpha_{k-b+1}+2\alpha_{n-2}+\alpha_{n-1}+\alpha_n$, for $1\le b<a\le k$.
Then 
$$\langle \lambda(\beta)+\rho,\alpha\rangle =2(n-k+1)+a+b-\beta_a-\beta_b.$$
So Bott's theorem tells us the following: if   $S_\beta U (2)$ is not acyclic, then 
\begin{enumerate}
    \item each $\beta_\ell$ must be either $\le \ell+1$ or $\ge 2(n-k)+\ell+1$; let $\ell_0$ be the maximal integer for which the second possibility holds;
    \item for any pair of distinct integers $a,b$ between $1$ and $k$, the sum 
    $\beta_a+\beta_b$ must be different from $2(n-k+1)+a+b$; let $m_0$ be the number of 
    pairs for which it is bigger. 
\end{enumerate}
Then $H^i(OG(k,2n), S_\beta U (2))$ is nonzero if and only if $i=2(n-k)\ell_0+m_0$.

\smallskip
Recall that $i$ is such that $S_\beta U$ is a component of $S_{k1^i}(S^2U)$. 
Since $S_{k1^i}(S^2U)\subset S_{k}(S^2U)\otimes \wedge^i(S^2U)$, this implies that 
$S_\beta U$ is a component of a tensor product $S_{2\alpha}U\otimes S_\gamma U$
for some partition $\alpha$ of size $k$ \cite[Proposition 2.3.8 (a)]{weyman-book}, and some partition $\gamma$ of size $2i$
obtained by putting together a skew-partition of size $i$ with its mirror along the diagonal, as in  \cite[Proposition 2.3.9 (a)]{weyman-book}.

\begin{prop}\label{uniqueness}
Suppose that $2n\ge \frac{k(k+5)}{2}$. Then $H^0(OG(k,2n), \mathcal I_M^k(2))=\CC.$
\end{prop}

\proof Following the discussion and notations above, the condition that $i=2(n-k)\ell_0+m_0$
implies $\ell_0=0$ since $i<\frac{k(k+1)}{2}$. But this means that $\beta_1\le 2$, and then $i=0$. \qedsymbol

\medskip 
The combinatorics becomes very messy when $k$ gets bigger compared to $n$, but 
we can get more evidence that Proposition \ref{uniqueness} should hold under much more general hypotheses by making explicit computations for $k$ or $n$ small enough. 
For example we can show: 

\begin{prop}
Suppose that $k=2,3,4,5$ and $n\geq k+2$. Then $H^0(OG(k,2n), \mathcal I_M^{k}(2))=\CC.$
\end{prop}

\begin{proof}
With \cite{LiE} we can compute the plethysms $S_{k,1,\dots,1}S^2U(2)$ explicitly. 
For instance, when $k=2$ the resolution takes the following simple form:
$$
0\to S^2U(-1) \to S^2U\oplus S^4U(1) \to \cO\oplus S^4U(2)\to \cI^2_M(2) \to 0.
$$
Then we can compute the cohomologies of these bundles by applying the Bott-Borel-Weil theorem, and conclude as before. 
\end{proof}

\medskip 
We can only deal with low values of $k$ because the number of irreducible components of $S_{k,1,\dots,1}S^2U$ grows too fast with $k$. Just as an example, here is the resolution for $k=3$:
$$ 0 \to S_{2,2} U (-2)\oplus S^4 U (-2) 
\to S_{4,4} U \oplus S_{3,2} U (-1)\oplus S^2 U (-2)\oplus S_{4,1} U (-1)\oplus S_{6,2} U  \to $$
$$ \to S_{3,3} U \oplus S_{2,1} U (-1)\oplus S_{5,4} U (1)\oplus S_{4,2} U \oplus 2S^3 U (-1)\oplus S_{6,3} U (1)\oplus S_{5,1} U \oplus S_{7,2} U (1) \to $$
$$ \to S_{5,5} U (2)\oplus  U (-1)\oplus S_{4,3} U (1)\oplus 2S_{3,1} U \oplus 2S_{5,2} U (1)\oplus S_{7,3} U (2)\oplus S_{6,1} U (1)\oplus S^7 U (1) \to $$
$$\to S_{3,2} U (1)\oplus S^2 U \oplus S_{5,3} U (2)\oplus S_{4,1} U (1)\oplus S_{6,2} U (2)\oplus S_{7,1} U (2) \to \hspace*{2cm}$$
$$\hspace*{5cm}\to\cO\oplus S_{4,2} U (2)\oplus S^6 U (2)  \to \cI^3_M(2) \to 0. $$

However, we could also check that:

\begin{prop}
Suppose that $k=2,3$ and $n= k+1$. Then $H^0(OG(k,2n), \mathcal I_M^{k}(2))=\CC.$
\end{prop}

\begin{remark}
For bigger $k$ the same result could possibly be true but we could not check it with our methods since we cannot control cancellations of cohomology groups in long exact sequences.
\end{remark}

 \section{Appendix C: Computations for the Coble quadric in genus four} 
 \label{app_C}
 In this Appendix we provide the details of the computations that lead to Proposition \ref{prop_formula_quadric_spin16}. 
 
 For $h=c_1h_1+\cdots +c_8h_8$, the goal is to compute 
$$\Gamma_2(h)=\zeta(\eta(h^2),\eta(h^2))=\sum_{ijkl}
\zeta(\eta(h_ih_j),\eta(h_kh_l)).$$
So we need to calculate the image by $\zeta$ of the products $\eta(h_ih_j)\eta(h_kh_l)$. We will start with the case where 
$i,j,k,l$ are distinct; recall from Lemma \ref{lem_subsystems_E8} 
that there must be two cases up to the Weyl group action. This is 
immediately visible when we compute that 
the image of  $\eta(h_1h_2)\eta(h_4h_7)$ by $\zeta$ vanishes, while that  of $\eta(h_1h_2)\eta(h_3h_4)$ is 
$$(e_0f_7)^2+(e_7f_0)^2+(e_1f_6)^2+(e_6f_1)^2+(e_2f_5)^2+(e_5f_2)^2+(e_3f_4)^2+(e_4f_3)^2.$$
A direct inspection leads to the following conclusion:

\begin{lemma}\label{14quadruples}
There are only $14$ quadruples $ijkl$ of distinct integers
for which the coefficient of $c_ic_jc_kc_l$ is non zero:
$$1234, \quad  1256, \quad 1367, \quad 1457, \quad 2357, \quad 2467,
\quad 3456,$$ \vspace*{-7mm}
$$5678,  \quad 3478, \quad  2458,\quad 2368, \quad  1468, \quad  1358, \quad 1278.$$
\end{lemma}

Denote by $Quad$ this collection of quadruples; this is a {\it Steiner 
quadruple system} $S(3,4,8)$ that already appeared in \cite{config}
in connection with $\fe_8$.

Note that the quadruples from the second series all contain $8$, 
and that suppressing it we get the seven lines of a Fano plane. 
The quadruples from the first series are their complements. But the 
integer $8$ has in fact no special r\^ole, and can be replaced by any other integer. A consequence is that any pair $ij$ is contained in
exactly three quadruples from $Quad$. 

\smallskip
For each quadruple $ijkl$ from $Quad$, the nonzero 
pairwise intersections of the four subspaces $E_{ij}^\bullet$ with the four subspaces $E_{kl}^\bullet$ are eight subspaces $F_{ijkl}^{(u)}$, $1\leq u\leq 8$. The resulting decomposition of $V$ is not affected 
if we permute the indices, for example we obtain the 
 same subspaces is we intersect $E_{ik}^\bullet$ and $E_{jl}^\bullet$. Moreover each $F_{ijkl}^{(u)}$ has a basis given by two vectors from the basis $e_1,\dots,e_8,f_1,\dots,f_8$; we will denote by $\det(F_{ijkl}^{(u)})$ the wedge product of these two vectors. The order in which we take the wedge of the two vectors is as follows: $e_v\wedge e_w$ and $f_v\wedge f_w$ for $v<w$; $e_v\wedge f_w$ for any $v,w$.
(This does not matter for the next Lemma but will for the next ones.) 

\begin{lemma}\label{lem_coef_c_ijkl}
Let $ijkl$ be one of the $14$ quadruples from $Quad$. 
The coefficient of $c_{ijkl}$  is   
$$24\sum_{u=1}^8 \det(F_{ijkl}^{(u)})^2.$$
\end{lemma}

\begin{proof}
This is a direct computation; the coefficient $ 24$ comes from counting all the permutations of $ijkl$.
\end{proof}

Now, given a pair $ij$, denote by $ijkl$, $ijmn$, $ijop$ the three quadruples from $Quad$ that contain it. 
Let us focus on one of them, say $ijkl$. The eight planes 
$F_{ijkl}^{(u)}$ are divided in four pairs,  each of which
generates one of the four subspaces $E_{ij}$, $E_{ij}'$, $E_{ij}''$, $E_{ij}'''$. Let us denote these four pairs by $(F_{ij|kl}^{(1_u)},F_{ij|kl}^{(2_u)})$ for $1\leq u\leq 4$, where we introduce a bar between $ij$ and $kl$ to stress that, contrary to the 
previous case, $ijkl$ cannot be permuted. 

\begin{lemma}\label{lem_coef_c_iijk}
For each quadruplet $ijkl$ from $Quad$, the coefficient of $c_ic_jc_k^2$ is the same as that of $c_ic_jc_l^2$, and is equal to $$ -6\sum_{u=1}^4 \det(F_{ij|kl}^{(1_u)}) \cdot \det(F_{ij|kl}^{(2_u)}) .$$
\end{lemma}

Consider the three quadruples $ijkl$, $ijmn$, $ijop$ from $Quad$ that contain $ij$. The Pl\"ucker relations yield the identities
$$  \det(F_{ij|kl}^{(1_u)}) \cdot \det(F_{ij|kl}^{(2_u)}) + \det(F_{ij|mn}^{(1_u)}) \cdot \det(F_{ij|mn}^{(2_u)}) + \det(F_{ij|op}^{(1_u)}) \cdot \det(F_{ij|op}^{(2_u)}) =0 .$$
Summing up over $u$, we deduce that the sum of the coefficient 
of $c_ic_jc_k^2$ (which is the same as that of $c_ic_jc_l^2$)
with that of $c_ic_jc_m^2$ (or $c_ic_jc_n^2$) and that of 
$c_ic_jc_o^2$ (or $c_ic_jc_p^2$), is zero. As a consequence, we obtain
(at most) two independant quadrics for each pair $ij$, hence a total 
amount of  $56$ quadrics in $S_{\langle 2,2 \rangle }V$ that we checked 
to be independent.

\begin{lemma}\label{coef31}
The coefficient of  $c_a^3c_b$ vanishes for any $a\ne b$. 
\end{lemma}

\proof
Recall from Lemma \ref{lem_subsystems_E8} that the Weyl group acts transitively on pairs of orthogonal roots, so
we can suppose that  $a=8$ and $b=1$. We have $\eta(h_1h_8)=e_{0123}+e_{4567}+f_{0123}+f_{4567}.$ A simple computation
gives $$\zeta(\sum_{i<j}e_{ij}f_{ij},e_{0123})=-2\Big(
(e_{01})(e_{23})-(e_{02})(e_{13})+(e_{03})(e_{12})\Big),$$
which again is zero by the Pl\"ucker identities. The other terms 
yield the same result, and our second claim follows.
\qed

\medskip 
In order to compute the remaining coefficients, 
we will need the following statement:

\begin{lemma}
\label{lem_abc}
Let $a,b,c$ define three distinct lines that are colinear in 
 the dual Fano plane. 
\begin{equation}\label{abc}
\epsilon_a(i,j)\epsilon_b(i,j)\epsilon_c(i,j)=1,\end{equation}
\begin{equation}\label{sumepsilon}
\sum_i\epsilon_a(i,j)\epsilon_b(i,k)=-\epsilon_a(j,k)-\epsilon_b(j,k),\end{equation}
\vspace{-4mm}
\begin{equation}\label{six}
\sum_i\epsilon_a(i,j)\epsilon_a(i,k)=6\epsilon_a(j,k).\end{equation}
\end{lemma}

\begin{remark}\label{aa8}
Notice that the first identity is trivially true if $(a,b,c)=(a,a,8)$. This suggests considering $(a,a,8)$ as a sort of line in a finite geometry that would extend the one of the Fano plane.
\end{remark}

Our goal will be to express the coefficients $c_a^4$ and $c_a^2c_b^2$ in terms of the following fourteen quadrics (where $c\neq 8$):
$$R_c:=\sum_{\epsilon_c(i,j)=1}(e_if_i)(e_jf_j), \qquad
R'_c:=\sum_{\epsilon_c(i,j)=1}(e_ie_j)(f_if_j).$$
We  will also use 
$$S_{a,b}=\sum_{\epsilon_a(i,j)=\epsilon_b(i,j)=1}(e_if_i)(e_{j}f_j), \qquad 
S'_{a,b}=\sum_{\epsilon_a(i,j)=\epsilon_b(i,j)=1}(e_ie_j)(f_if_j).$$
The latter has the advantage of involving only four distinct terms each. 

\begin{lemma}
Given $a\ne b$, let $d$ be line that they span  
in the Fano plane. The quadrics $S_{a,b}$ and $S'_{a,b}$ only depend on $d$, we denote them $S_d$ and $S'_d$. Then
$$R_c=\sum_{c\in d}S_d \qquad \mathit{and} \qquad R'_c=\sum_{c\in d}S'_d.$$
\end{lemma}

\proof Given a line $d=(abc),$ we can split the remaining lines into 
three pairs: there are two lines other than $d$ passing through $a$,
respectively $b$ and $c$. Adding $(0d)$, this gives a partition of $0,\ldots ,7$  into four pairs, and one can check they are exactly 
those pairs $(j,k)$ such that $ \epsilon_a(j,k)=\epsilon_b(j,k)=1$. 
In particular this does not depend on the choice of the pair $(a,b)$
on the line $d$. 

The second claim can be checked explicitly: starting from the line $(123)$ defined by $h_1$, we get the three pairs of lines $(27)$ (passing through $1$), $(35)$ (passing through $2$), and $(46)$ (passing through $3$). 
These three pairs yields three partitions of $0,\ldots ,7$ into
four pairs, respectively 
$$(01)(23)(45)(67), \quad (02)(13)(46)(57), \quad (03)(12)(47)(56).$$
The union of these pairs is  $I_1^1$, and this exactly means that 
$R'_1=S_1+S_2+S_3$.\qed 

\medskip Note that $I_0, I_1, I_2, I_3$ can be expressed in terms of the $R_c, R'_c$ and $S_d$ since 
$$I_1 = \frac{1}{3}\sum_c R_c, \qquad I_2 = \frac{1}{3}\sum_c R'_c, $$
and by (\ref{sommes}), $I_3=I_1-I_2$ and $I_0=4I_2-2I_1$.

\medskip
Now we compute 
$$\zeta(e_{ij}f_{ij},e_{ij}f_{ij})=(e_if_i)^2+(e_jf_j)^2+
2(e_if_j)(e_jf_i)-2(e_ie_j)(f_if_j).$$

\smallskip
By the Pl\"ucker identity $(e_if_j)(e_jf_i) = (e_if_i)(e_jf_j)-(e_ie_j)(f_if_j)$,
this can be rewritten as
\begin{equation}\label{contract1}
\zeta(e_{ij}f_{ij},e_{ij}f_{ij})=\Big( (e_if_i)+(e_jf_j)\Big)^2-4(e_ie_j)(f_if_j),
\end{equation}
\vspace{-4mm}
\begin{equation}\label{contract2}
\zeta(e_{ij}f_{ij},e_{ik}f_{ik})=(e_jf_j)(e_kf_k)\qquad \mathrm{for}\; j\ne k.
\end{equation}
\smallskip

\begin{lemma}\label{coef4}
The coefficient of $c_a^4$ is $6R_a-3I_1+\frac{3}{2}I_0$.
\end{lemma}

\proof This coefficient is $\zeta(\eta(h_a^2),\eta(h_a^2))$, that we compute, up to a coefficient four,
as 
$$\sum_{i<j}\sum_{k<l}\epsilon_a(i,j)\epsilon_a(k,l)\zeta(e_{ij}f_{ij},e_{kl}f_{kl}).$$

A first contribution to this sum is obtained for $(k,l)=(i,j)$, which by (\ref{contract1}) and 
(\ref{sommes}) yields
$7I_0+2I_3-2I_2=6I_0$.

A second contribution is obtained when $(k,l)$ and $(i,j)$ share exactly one integer. There are four possibilities for this to occur: either $k=i$, or $k=j$, or $l=i$, or $l=j$. We get 
$$\sum_{i<j}\sum_{i<l\ne j}\epsilon_a(i,j)\epsilon_a(i,l)\zeta(e_{ij}f_{ij},e_{il}f_{il})+
\sum_{i<j}\sum_{j<l}\epsilon_a(i,j)\epsilon_a(j,l)\zeta(e_{ij}f_{ij},e_{jl}f_{jl})+$$
$$\sum_{i<j}\sum_{k<i}\epsilon_a(i,j)\epsilon_a(k,i)\zeta(e_{ij}f_{ij},e_{ki}f_{ki})
+\sum_{i<j}\sum_{i\ne k<j}\epsilon_a(i,j)\epsilon_a(k,j)\zeta(e_{ij}f_{ij},e_{kj}f_{kj}).$$
Using \ref{contract2} and permuting the indices, this leads to 
$$\sum_{l\ne j}(e_jf_j)(e_lf_l)\sum_{i<j,l}\epsilon_a(i,j)\epsilon_a(i,l)+
\sum_{l<j}(e_jf_j)(e_lf_l)\sum_{l<i<j}\epsilon_a(i,j)\epsilon_a(l,i)+$$
$$+\sum_{l\ne j}(e_jf_j)(e_lf_l)\sum_{i>j,l}\epsilon_a(j,i)\epsilon_a(l,i)+
\sum_{l>j}(e_jf_j)(e_lf_l)\sum_{j<i<l}\epsilon_a(j,i)\epsilon_a(i,l).$$
Splitting the first and third sum and permuting the indices, we can rewrite this as
$$2\sum_{l<j}(e_jf_j)(e_lf_l)\sum_{i\ne j,l}\epsilon_a(i,j)\epsilon_a(i,l)=
12\sum_{l<j}\epsilon_a(l,j)(e_jf_j)(e_lf_l),$$
where the last  equality follows from (\ref{six}). 
Since $2R_a=\sum_{i<j}\epsilon_a(i,j)(e_if_i)(e_jf_j)+I_1$, this implies the claim.\qed

\medskip
Now we turn to the coefficients of monomials of type $c_a^2c_b^2$.
In the next two Lemmas $a$ and $b$ are two distinct integers between $1$ and $7$, that we consider as points in the Fano plane. Then they are joined by a unique line, containing a unique extra point that we denote by $c$.

\begin{lemma} 
$\zeta(\eta(h_ah_b)^2)=-2R'_c+2R_c-2S_{a,b}.$
\end{lemma}

\proof We know that $\eta(h_ah_b)=s_{E_{ab}}+s_{E_{ab}'}+s_{E_{ab}''}+s_{E_{ab}'''}$, 
where $(E_{ab}, E_{ab}'')$ and $(E_{ab}',E_{ab}''')$ are the only nonorthogonal pairs. 
So we get
$$\zeta(\eta(h_ah_b)^2)=
2\zeta(s_{E_{ab}},s_{E_{ab}''})+2\zeta(s_{E_{ab}'},s_{E_{ab}'''}).$$
This is equal to 
$$ -2\sum_{\epsilon_a(i,j)=\epsilon_b(i,j)=1}e_{ij}f_{ij}+2 \sum_{\epsilon_a(i,j)=\epsilon_b(i,j)=-1}(e_if_j)(e_jf_i) ,$$
which, by the Pl\"ucker relations, gives 
$$ -2\sum_{\epsilon_a(i,j)=\epsilon_b(i,j)=1}e_{ij}f_{ij}+2 \sum_{\epsilon_a(i,j)=\epsilon_b(i,j)=-1}(e_{ij}f_{ij}(e_if_i)(e_jf_j)),$$
which yields the result.
\qed 

\begin{lemma}
 $ \zeta(\eta(h_a^2),\eta(h_b^2))=R_c-R_a-R_b-2R'_c-\frac{1}{4}I_0+\frac{1}{2}I_1+I_2.$
\end{lemma}

\proof We proceed as in the proof of Lemma \ref{coef4}: up to a factor four, we need to compute 
$$\sum_{i<j}\sum_{k<l}\epsilon_{a}(i,j)\epsilon_{b}(k,l) \zeta(e_{ij}f_{ij}, e_{kl}f_{kl}).$$
There is a first contribution coming from the case where $(k,l)=(i,j)$, which yields 
$$\sum_{i<j}\epsilon_{c}(i,j)\Big( ((e_if_i)+(e_jf_j))^2-4(e_ie_j)(f_if_j)\Big)$$
because of identity (\ref{abc}). Since $\sum_{i\ne j}\epsilon_{c}(i,j)=-1$ for any $j$, this can be rewritten as 
$$-\sum_i(e_if_i)^2+2\sum_{i<j}\epsilon_{c}(i,j)(e_if_i)(e_jf_j)
-4\sum_{i<j}\epsilon_{c}(i,j)(e_ie_j)(f_if_j).$$

There is a second contribution coming from the case where  $(k,l)$ and $(i,j)$ have exactly one
integer in common, which yields, after some rewriting,
$$\sum_{j\ne l}\Big(\sum_{i\ne j,l}\epsilon_{a}(i,j)\epsilon_{b}(i,l)\Big) (e_jf_j)(e_lf_l).$$
Using identity (\ref{sumepsilon}), this is also 
$$-2\sum_{j<l}(\epsilon_{a}(j,l)+\epsilon_{b}(j,l))(e_jf_j)(e_lf_l).$$ 
Dividing by four and using the identities $$2R_a=\sum_{i<j}\epsilon_a(i,j)(e_if_i)(e_jf_j)+I_1, \qquad 2R_a'=\sum_{i<j}\epsilon_a(i,j)(e_ie_j)(f_if_j)+I_2,$$ 
we finally get the result. 
\qed 

\medskip
The previous computation is also valid for $b=8$ and $c=a$ (see Remark \ref{aa8}). We get:

\begin{lemma}
$\zeta(\eta(h_a^2),\eta(h_8^2))=-2R'_a-\frac{1}{2}I_1-\frac{1}{4}I_0+I_2.$
%-3\sum_{(i<j)\in I_1^1}(e_if_i)(e_jf_j).$
\end{lemma}

One last computation remains to be done:

\begin{lemma}
$\zeta(\eta(h_ah_8)^2)= -2R'_a$.
\end{lemma}

\proof Suppose for example that $a=1$. We know that $\eta(h_1h_8)=e_{0123}+e_{4567}+f_{0123}+f_{4567}$, so we need to compute $2\zeta(e_{0123},f_{0123})+2\zeta(e_{4567},f_{4567})$, which easily yields the result. \qed 
\medskip

Denote by $Q_{a,b}$ the coefficient of $c_a^2c_b^2$. 

\begin{coro}\label{cor_coef_c_iijj}
For $a\ne b$, let $c$ denote the extra point on the line that joins them in the dual Fano plane; for $b=8$, let $c=a$. The coefficient of $c_a^2c_b^2$ is 
$$10R_c-8S_{a,b}-2R_a-2R_b-12R'_c-\frac{1}{2}I_0+I_1+2I_2.$$ 
\end{coro}

\providecommand{\bysame}{\leavevmode\hbox to3em{\hrulefill}\thinspace}
\providecommand{\MR}{\relax\ifhmode\unskip\space\fi MR }
% \MRhref is called by the amsart/book/proc definition of \MR.
\providecommand{\MRhref}[2]{%
  \href{http://www.ams.org/mathscinet-getitem?mr=#1}{#2}
}
\providecommand{\href}[2]{#2}

\footnotesize

\bigskip 
Vladimiro Benedetti,
Universit\'e C\^ote d'Azur, CNRS, Laboratoire J.-A. Dieudonn\'e, Parc Valrose, F-06108 Nice Cedex 2, {\sc France}.

{\it Email address}: {\tt vladimiro.benedetti@univ-cotedazur.fr}

\smallskip 

Michele Bolognesi,
Institut Montpelli\'erain Alexander Grothendieck, Universit\'e de Montpellier, CNRS, 
 Place  Eug\`ene Bataillon, 34095 Montpellier Cedex 5, {\sc France}.
 
{\it Email address}: {\tt michele.bolognesi@umontpellier.fr}

\smallskip 

Daniele Faenzi,
Université Bourgogne Europe, CNRS, IMB UMR 5584, F-21000 Dijon, France.

{\it Email address}: {\tt daniele.faenzi@u-bourgogne.fr}

\smallskip 

Laurent Manivel,
Institut de Math\'ematiques de Toulouse, 
Paul Sabatier University, 118 route de Narbonne, 31062 Toulouse
Cedex 9, {\sc France}.

{\it Email address}: {\tt manivel@math.cnrs.fr}

\end{document}